\documentclass[a4paper,11pt,oneside,reqno]{amsart}
\usepackage{style}
\title{Nonlinear weak error expansion of McKean-Vlasov stochastic differential equations}
\author[Jourdain]{Benjamin Jourdain}
\address{CERMICS, ENPC, Institut Polytechnique de Paris, INRIA, Marne-la-Vallée, France}
\email[Jourdain]{benjamin.jourdain@enpc.fr}
\author[Le]{Anh-Dung Le}
\address{CERMICS, ENPC, Institut Polytechnique de Paris, INRIA, Marne-la-Vallée, France}
\email[Le]{leanhdung1994@gmail.com}
\date{\today}
\thanks{This research benefited from the support of the ``Chaire Risques Financiers'', Fondation du Risque}

\begin{document}

\begin{abstract}
According to Talay and Tubaro \cite{talay_expansion_1990}, the weak error between the solution to a stochastic differential equation with smooth coefficients and its Euler-Maruyama scheme can be expanded in powers of the time-step. In the present paper, we generalize this result to the case when the error is measured by a smooth functional on the Wasserstein space of probability measures in place of the linear functional given by the expectation of a smooth function considered in \cite{talay_expansion_1990}. Since this does not complicate our analysis based on the master partial differential equation, we even deal with the McKean-Vlasov case when the coefficients of the stochastic differential equation may depend on its current marginal distribution.
\end{abstract}

\maketitle



\tableofcontents

\section{Introduction} \label{intro}
Let $\bF \coloneq \{ \cF_t \}_{ t \in \bR_+ }$ be a filtration on a probability space $(\Omega, \cA, \bP)$ and $\{ B_t \}_{t \in \bR_+}$ be a $d$-dimensional $\bF$-Brownian motion. Let $\sP_2 (\bR^d)$ be the space of Borel probability measures on $\bR^d$ with finite second moment. We endow $\sP_2 (\bR^d)$ with the Wasserstein metric $\sW_2$. We fix $T \in (0, \infty)$ and let $\bT \coloneq [0, T]$. We consider smooth functions
\begin{align}
b & \colon \bT \times \bR^d \times \sP_2 (\bR^d) \to \bR^d , \\
\sigma & \colon \bT \times \bR^d \times \sP_2 (\bR^d) \to \bR^d \otimes \bR^d . 
\end{align}

For a random variable $Y$, we denote by $\lbrbrak Y \rbrbrak$ its probability distribution. For $\nu \in \sP_2 (\bR^d)$, we consider the SDE
\begin{equation} \label{intro:main-sde}
\begin{dcases}
X_t = X_0 + \int_0^t b (s, X_s, \mu_s) \diff s + \int_0^t \sigma (s, X_s, \mu_s) \diff B_s
\qtextq{for} t \in \bT , \\
\nu = \lbrbrak X_0 \rbrbrak \qtextq{and} \mu_s \coloneq \lbrbrak X_s \rbrbrak .
\end{dcases}
\end{equation}

Above, $X_0$ is $\cF_0$-measurable. In \zcref{intro:main-sde}, the SDE involves the marginal distributions of the solution, so it is called \textit{McKean-Vlasov} or \textit{mean-field} SDEs. The existence and uniqueness of the solution to \zcref{intro:main-sde} has been investigated intensively. Its mathematical study was started by McKean \cite{mckean1966class}, who was inspired by Kac’s seminal paper \cite{kac1956foundations} about kinetic theory. For \zcref{intro:main-sde} to have a unique strong solution, it suffices that $b(s, x, \mu)$ and $\sigma(s, x, \mu)$ are Lipschitz in $(x, \mu)$ with affine growth, uniformly in $s$. See e.g. \cite[Proposition 1.2]{jourdain_nonlinear_2008} and \cite[Theorem 4.21]{carmona2018probabilistic}. There are weaker conditions on $b$ and $\sigma$ to ensure the well-posedness of \zcref{intro:main-sde}. The existence and uniqueness were obtained with one-sided global Lipschitz $b$ and global Lipschitz $\sigma$ in \cite{wang_distribution_2018,dos_reis_freidlinwentzell_2019,dos_reis_simulation_2022}; with Lyapunov-type conditions in \cite{mehri_weak_2019,hammersley_weak_2021,liu_existence_2022}; and with local Lipschitz conditions in \cite{kloeden_stochastic_2010,li_strong_2023}. The literature on \zcref{intro:main-sde} has become so vast that we invite the readers to consult the references in those mentioned papers.

The Euler-Maruyama scheme $X^n \coloneq \{ X^{n}_t \}_{ t \in \bT }$ with $n \in \bN^*$ steps for \zcref{intro:main-sde} is defined as follows. Let $\eps_n \coloneq T/n$ and $t_k \coloneq k \eps_n$ for $k \in \llbracket 0, n  \rrbracket$. Let $\tau^n_t \coloneq t_k$ if $t \in [t_k, t_{k+1})$ for some $k \in \llbracket 0, n-1  \rrbracket$. Let
\begin{equation} \label{intro:euler-scheme}
\begin{dcases}
X^{n}_t & \coloneq X_0 + \int_{0}^t b (\tau^{n}_s , X^{n}_{\tau^{n}_s}, \mu^{n}_{\tau^{n}_s}) \diff s + \int_0^t \sigma (\tau^{n}_s , X^{n}_{\tau^{n}_s}, \mu^{n}_{\tau^{n}_s}) \diff B_s , \\
\mu^{n}_{s} & \coloneq \lbrbrak X^{n}_{s} \rbrbrak .
\end{dcases}
\end{equation}

The purpose of this paper is to extend the expansion of $\bE[\phi(X^n_T)]-\bE[\phi(X_T)]$ in powers of $\frac 1n$ obtained by Talay and Tubaro \cite{talay_expansion_1990} when the coefficients of the SDE \zcref{intro:main-sde} do not depend on the marginal distribution $\mu_s$ and $\phi \colon \bR^d \to \bR$ is smooth. Our extension of \cite{talay_expansion_1990} is about the expansion of $\cU (\mu^{n}_T) - \cU (\mu_T)$ when $\cU \colon \sP_2 (\bR^d) \to \bR$ is smooth and the SDE is of McKean-Vlasov type. Our main result, \zcref{main-thm}, states that, given sufficient regularity, there exists a sequence $\{ C_i \}_{i\ge 1}$ of real constants not depending on $n$ such that
\[
\cU (\mu^{n}_T) - \cU (\mu_T) = \sum_{i=1}^m \frac{C_i}{n^i} + \cO \Big ( \frac{1}{n^{m+1}} \Big )
\qtextq{for all} m\in\bN^* ,
\]
where $\cO$ is the big O notation.

Let us give a brief review of the related literature about the expansion $\bE[\phi(X^n_T)]-\bE[\phi(X_T)]$ in powers of $\frac{1}{n}$. Talay and Tubaro \cite{talay_expansion_1990} proved this expansion when $(b, \sigma)$ are smooth with bounded derivatives of any order $\ge 1$ and $\phi$ is smooth with polynomially-growing derivatives of any order. The proof is by the Feynman-Kac PDE associated with \zcref{intro:main-sde}. Since then, many authors have generalized the result in \cite{talay_expansion_1990} in several directions. Using Malliavin calculus, Bally and Talay \cite{bally_law_1996} extended \cite{talay_expansion_1990} to the case when $\phi$ is bounded and measurable, and an additional uniform hypoellipticity condition on the infinitesimal generator of \zcref{intro:main-sde}. They proved in \cite{bally_law_1996b} the first-order expansion for the difference between the probability density functions of $X^n_t$ and $X_t$. The authors in \cite{protter_euler_1997,jacod_approximate_2005} extended \cite{talay_expansion_1990} to the case of Lévy-driven SDEs. Using the parametrix method, Konakov and Mammen \cite{konakov_edgeworth_2002} extended the expansion in \cite{bally_law_1996b} to arbitrary order. Guyon \cite{guyon_euler_2006} extended \cite{talay_expansion_1990} to the case when $\phi$ is a tempered distribution, with an additional uniform ellipticity condition on $\sigma$. The expansion has also been established for Hamiltonian systems \cite{talay2002stochastic} and kinetic Langevin diffusions \cite{kopec_weak_2015,journel_weak_2024}. We aslo refer the readers to the bibliography in those papers.

In the McKean-Vlasov case, the Euler-Maruyama scheme \zcref{intro:euler-scheme} cannot be implemented without approximating the probability measures $\{ \mu^n_{t_k} \}_{k\in \llbracket 0, n  \rrbracket}$. We do not analyze such approximations in the present paper. Nevertheless, in the case when the coefficients $(b, \sigma)$ do not depend on the measure argument and the implementation of the Euler-Maruyama scheme is straightforward, our result already generalizes to nonlinear maps $\cU\colon \sP_2  (\bR^d)\to \bR$ the Talay-Tubaro expansion which only deals with linear ones $\cU(\mu)=\int_{\bR^d}\phi(x)\diff\mu( x)$. Then $\cU(\mu_T)$ may be approximated by $\cU(\frac 1 J\sum_{j=1}^J\delta_{X^{n,j}_T})$ where $\{ X^{n,j} \}_{j\in \llbracket 1, J  \rrbracket}$ are independent copies of $X^n$. Going from standard SDEs with general maps $\cU$ to McKean-Vlasov SDEs does not complicate the proofs.

One method of approximating $\{ \mu_{t} \}_{t \in \bT}$ in \zcref{intro:main-sde} is by interacting particle systems $\{ \bar X^{m,j} \}_{j\in \llbracket 1, m  \rrbracket}$ where $\bar X^{m,j} \coloneq \{ \bar X^{m,j}_t \}_{t \in \bT}$
\begin{equation} \label{intro:IPS}
\bar X^{m, j}_t = X_0^{(j)} + \int_0^t b \Big (s, \bar X^{m, j}_s, \frac{1}{m} \sum_{j=1}^m \delta_{\bar X^{m, j}_s} \Big ) \diff s + \int_0^t \sigma \Big (s, \bar X^{m, j}_s, \frac{1}{m} \sum_{j=1}^m \delta_{\bar X^{m, j}_s} \Big ) \diff B_s^{(j)} .
\end{equation}

Above, $\{ (X_0^{(j)}, B_s^{(j)})_{s \in \bR_+} \}_{j \in \bN^*}$ are i.i.d copies of $(X_0, B_s)_{s \in \bR_+}$. With enough regularity on $(X_0, b, \sigma)$, the random empirical measure $\frac{1}{m} \sum_{j=1}^m \delta_{\bar X^{m, j}_s}$ converges to $\mu_s$ in some suitable sense as $m \to \infty$. This phenomenon is called \textit{propagation of chaos} (see e.g. \cite{burkholder_ecole_1991,chaintron_propagation_2022a,chaintron_propagation_2022b}). The paper \cite{chassagneux_weak_2022} concerns \zcref{intro:IPS} and is about the expansion of $\bE [ \cU ( \frac{1}{m} \sum_{j=1}^m \delta_{\bar X^{m, j}_T} ) ] - \cU ( \mu_T )$ in the number $m$ of particles without time discretization.

To simulate \zcref{intro:main-sde}, we need to further discretize $\bar X^{m,j}$ in time. This is accomplished in \cite{bernou_particle_2022,liu_particle_2024,frikha_convergence_2025} whose objectives are different from ours. The focus of \cite{bernou_particle_2022,liu_particle_2024} is on the convergence rate of the empirical measure $\mu^{n, m}_{t_k}$ at time $t_k$ of the time-discretized particle system to $\mu^n_{t_k}$. The focus of \cite{frikha_convergence_2025} is on estimating both the strong and the weak errors between $\mu^{n,m}_t$ and $\mu_t$. In particular, \cite[Theorem 2.3]{frikha_convergence_2025} states that there exists a constant $C \in \bR_+$ such that
\[
\sup_{t \in \bT} | \bE [\cU (\mu_t^{n, m} ) - \cU (\mu_t) ] | \le C \Big ( \frac{1}{m} +  \frac{1}{n} \Big ) .
\]

The present paper contributes in two aspects.
\begin{itemize}
\item Previous papers on Talay-Tubaro's expansion only concern the linear case when $\cU(\mu)=\int_{\bR^d}\phi(x)\diff\mu( x)$ and $(b, \sigma)$ do not depend on distribution variable. Here we work with smooth but nonlinear $\cU$ and McKean-Vlasov SDEs.

\item Our proof relies on the regularity of the master PDE \zcref{cal:master-PDE:eq2} in both time, space and measure variables. The two previous papers \cite{buckdahn_mean-field_2017,chassagneux_weak_2022} established the higher-order regularity in space and measure variables. As a by product, we establish the higher-order regularity in time for \zcref{cal:master-PDE:eq2}.
\end{itemize}

Throughout the paper, we adopt the following conventions:
\begin{itemize}[parsep=2.5pt]
\item Let $\bN$ be the set of natural numbers including $0$. We denote $\bN^* \coloneq \bN \setminus \{0\}$.

\item For $i, m \in \bN$ with $i \le m$, let $\llbracket i, m \rrbracket \coloneq \{i, i+1, \ldots, m\}$.

\item For $x, y \in \bR^d$, we denote by $x \cdot y$ their dot product. For $x, y \in \bR^d \otimes \bR^m$, we denote by $x : y$ their Frobenius inner product.

\item For $x \in \bR^d \otimes \bR^m$, we denote by $x^\top \in \bR^m \otimes \bR^d$ its transpose. For $x \in \bR^d \otimes \bR^m$ and $y \in \bR^m \otimes \bR^n$, we denote by $xy \in \bR^d \otimes \bR^n$ their matrix multiplication.

\item Let $(\bR^d)^{\otimes m}$ denote the tensor product of $m$ copies of $\bR^d$. For $x \in (\bR^d)^{\otimes m}$ and $i \in \llbracket 1, d \rrbracket^m$, we denote by $[x]_i \in \bR$ its entry.

\item Let $\sP (\bR^d)$ be the space of Borel probability measures on $\bR^d$. For a random variable $Y$, we denote by $\lbrbrak Y \rbrbrak$ its probability distribution.

\item For a measurable map $\psi \colon \bR^m \to \bR^d$ and $\mu \in \sP (\bR^m)$, we denote by $\psi_\sharp \mu \in \sP (\bR^d)$ the push-forward measure of $\mu$ through $\psi$.

\item For $(m, t) \in \bN^* \times \bR_+$, we denote the Lebesgue space $L^2 (\cF_t; \bR^m) \coloneq L^2 (\Omega, \cF_t, \bP; \bR^m)$.

\item For $q\in\bN^*$, let $\Delta^q \coloneq \{(s_1, \ldots, s_q) \in \bT^q ; s_1 \ge \cdots \ge s_q \}$ where we recall that $\bT$ is the closed interval $[0, T]$. We denote by $\intr \Delta^q$ the interior of $\Delta^q$ in $\bR^q$.

\item If $U \colon \Delta^q \times \bR^r \times \sP_2 (\bR^d) \to \bR$ is said to be continuous, we mean that it is jointly continuous in all variables.
\end{itemize}

We assume that $(\Omega, \cA, \bF, \bP)$ satisfies the usual conditions and $(\Omega, \cF_0, \bP)$ is rich enough such that $\sP_2 (\bR^m) = \{ \lbrbrak X \rbrbrak ; X \in L^2 (\cF_0; \bR^m) \}$ for every $m \in \bN^*$.

\section{Calculus on $\sP_2 (\bR^d)$} \label{calculus_on-P2}

For $i \in \{1, 2\}$, let $\pi^i \colon (\bR^d)^2 \to \bR^d$ be the projection of $(\bR^d)^2$ onto its $i$-th component. The set of \textit{transport plans} (or \textit{couplings}) between $\mu, \nu \in \sP (\bR^d)$ is defined as
\[
\Gamma (\mu, \nu) \coloneq \{ \varrho \in \sP ({\bR}^d \times {\bR}^d ) ; \mu = \pi^1_\sharp \varrho \text{ and } \nu = \pi^2_\sharp \varrho \} .
\]

The Wasserstein metric $\sW_2$ between $\mu, \nu \in \sP_2 (\bR^d)$ is defined as
\begin{equation} \label{cal:-Wp-def}
\sW_2 (\mu, \nu) \coloneq \inf_{\varrho \in \Gamma (\mu, \nu)} \Big \{ \int_{\bR^d} |x-y|^2 \diff \varrho (x, y) \Big \}^{1/2} .
\end{equation}

To properly state our results, we recall the notion of measure derivative introduced by Lions in his seminal lectures at the Collège de France \cite{lions2007}. A comprehensive presentation can be found in the monograph of Carmona and Delarue \cite{carmona_probabilistic_2018}.

\subsection{Lions derivatives}

We fix $U \colon \sP_2 (\bR^d) \to \bR$ and $\mu \in \sP_2 (\bR^d)$. The \textit{lift} $\tilde U \colon L^2 (\cF_0 ; \bR^d) \to \bR$ of $U$ is defined as $\tilde U (X) \coloneq U (\lbrbrak X \rbrbrak)$.

\begin{definition} \label{cal:ld-def}
$U$ is said to be \textit{L-differentiable} at $\mu$ if there exists $X \in L^2 (\cF_0 ; \bR^d)$ such that $\mu = \lbrbrak X \rbrbrak$ and that $\tilde U$ is Fréchet differentiable at $X$.
\end{definition}

By Riesz representation theorem, the gradient of $\tilde U$ at $X$ can be identified with a unique element $\nabla \tilde U ( X) \in L^2 (\cF_0 ; \bR^d)$. By \cite[Lemma 3.4]{alfonsi_squared_2020}, if $U$ is L-differentiable at $\mu$, then there exists a function $\partial_\mu U (\mu ; \cdot)$ that belongs to $L^2 (\bR^d, \cB (\bR^d), \mu ; \bR^d)$, that depends only on $\mu$ and such that $\nabla \tilde U ( X) = \partial_\mu U (\mu ; X)$ for any $X \in L^2 (\cF_0; \bR^d)$ with $\lbrbrak X \rbrbrak = \mu$. Here $\partial_\mu U (\mu ; X)$ means the evaluation of $\partial_\mu U (\mu ; \cdot) \colon \bR^d \to \bR^d$ at each realization of $X$.

\begin{definition} \label{cal:ld-def2}
If $U$ is L-differentiable at $\mu$, then its \textit{L-derivative} at $\mu$ is defined as $\partial_\mu U (\mu ; \cdot)$.
\end{definition}

We define higher-order Lions derivatives by induction.

\begin{definition} \label{cal:higher-ld-def}
Let $p \in \bN^*$. The map $U$ is said to be \textit{$(p+1)$-times L-differentiable} at $\mu$ if there exists a neighborhood $O$ of $\mu$ such that
\begin{itemize}
\item $U$ is $p$-times L-differentiable at every $\nu \in O$ with $p$-times L-derivative $\partial_\mu^{p} U (\nu ; \cdot) \colon (\bR^d)^{p} \to (\bR^d)^{\otimes p}$.

\item For every $x \in (\bR^d)^{p}$ and $\{ i_1, \ldots, i_{p} \} \subset \llbracket 1, d  \rrbracket$, the map $O \ni \nu \mapsto [ \partial_\mu^{p} U ]_{(i_1, \ldots, i_{p})} (\nu ; x)$ is L-differentiable at $\mu$ according to \zcref{cal:ld-def}.
\end{itemize}

In this case, the \textit{$(p+1)$-th L-derivative} of $U$ at $\mu$ is defined for all $x \in (\bR^d)^{p} , y \in \bR^d$ and $\{ i_1, \ldots, i_{p+1} \} \subset \llbracket 1, d  \rrbracket$ as
\[
\begin{multlined}[t]
[ \partial_\mu^{p+1} U ]_{(i_1, \ldots, i_{p+1})} (\mu ; x, y) \coloneq [ \partial_\mu \{ [\partial_\mu^{p} U ]_{(i_1, \ldots, i_{p})} (\cdot ; x) \} ]_{i_{p+1}} (\mu ; y) .
\end{multlined}
\]
\end{definition}

\subsection{Functions differentiable in measure}

For brevity, we introduce multi-index notation.

\begin{definition} \label{multi-index}
Let $p, q, r \in \bN$. Let $h = (h_1, \ldots, h_r) \in \bN^r, m = (m_1, \ldots, m_q) \in \bN^q$ and $\ell = (\ell_1, \ldots, \ell_p) \in \bN^p$. An ordered tuple of the form $\alpha = (p, h, m, \ell)$ is called a \textit{multi-index}. Its \textit{order} is defined as
\[
|\alpha| \coloneq p + h_1 + \ldots + h_r + 2 (m_1 + \ldots + m_q) + \ell_1+\cdots+\ell_p .
\]

For $U \colon \Delta^q \times \bR^r \times \sP_2 (\bR^d) \to \bR$, its mixed derivative at $s = (s_1, \ldots, s_q) \in \intr \Delta^q, x = (x_1, \ldots, x_r) \in \bR^r, \mu \in \sP_2 (\bR^d)$ and $v = (v_1, \ldots, v_p) \in (\bR^d)^p$, is defined as
\begin{equation} \label{multi-index:eq1}
\begin{multlined}[t]
\rD^{\alpha} U (s, x, \mu ; v) \coloneq \nabla_{v_p}^{\ell_p} \ldots \nabla_{v_1}^{\ell_1} \partial_\mu^{p} \partial_{x_r}^{h_r} \ldots \partial_{x_1}^{h_1} \partial^{m_q}_{s_q} \ldots \partial^{m_1}_{s_1} U (s, x, \mu ; v ) .
\end{multlined}
\end{equation}

If $q=0$ then $U \colon \bR^r \times \sP_2 (\bR^d) \to \bR$. If $r=0$ then $U \colon \Delta^q \times \sP_2 (\bR^d) \to \bR$. If $q=r=0$ then $U \colon \sP_2 (\bR^d) \to \bR$.
\end{definition}

In the order introduced in \zcref{multi-index}, one derivative in time is worth two derivatives in space or in measure.

\begin{definition} \label{cal:space1}
Let $q, r, k \in \bN$ and $U \colon \Delta^q \times \bR^r \times \sP_2 (\bR^d) \to \bR$. The map $U$ is said to be of class $\sM_b^k$ if, there is a constant $C>0$, for every multi-index $\alpha = (p, h, m, \ell)$ with $|\alpha| \le k$, its mixed derivative $\rD^{\alpha} U \colon (\intr \Delta^q) \times \bR^r \times \sP_2 (\bR^d) \times (\bR^d)^p \to (\bR^d)^{\otimes(p+\ell_1+\cdots+\ell_p)}$ exists and is continuous such that
\begin{itemize}
\item for every $s \in \intr \Delta^q; x, x' \in \bR^r; \mu, \mu' \in \sP_2 (\bR^d)$ and $v, v' \in (\bR^d)^p$:
\begin{equation} \label{cal:space1:ineq1}
\begin{aligned}
| \rD^{\alpha} U (s, x, \mu ; v ) | & \le C , \\
| \rD^{\alpha} U (s, x, \mu ; v ) - \rD^{\alpha} U (s, x', \mu' ; v' ) | & \le C \{ |x-x'| + |v-v'| + \sW_2(\mu, \mu') \} ,
\end{aligned}
\end{equation}

\item for every $(x, \mu, v)  \in \bR^r \times \sP_2 (\bR^d) \times (\bR^d)^p$, the map $s \mapsto \rD^{\alpha} U (s, x, \mu ; v )$ is uniformly continuous on $\intr \Delta^q$.
\end{itemize}
The map $U$ is said to be of class $\sM_b^\infty$ and called \textit{smooth} if it is of class $\sM_b^k$ for every $k \in \bN$.
\end{definition}

In \zcref{cal:space1}, $\rD^{\alpha} U$ has a unique continuous extension to $\Delta^q \times \bR^r \times \sP_2 (\bR^d) \times (\bR^d)^p$ and we always work with this extension. \zcref{cal:space1} extends naturally to the case of $\bR^m$-valued maps, i.e., $U \colon \Delta^q \times \bR^r \times \sP_2 (\bR^d) \to \bR^m$ is of class $\sM^k_b$ if its $i$-th component $[U]_i$ is of class $\sM^k_b$ for every $i \in \llbracket 1, m  \rrbracket$. The class $\sM_b^k$ in \zcref{cal:space1} is inspired by \cite[Definition 2.17]{chassagneux_weak_2022} but is required to have higher-order time derivatives. This is because we need to repeat taking time derivative in the \hyperref[main-thm:prf]{proof} of \zcref{main-thm}. We will prove in \zcref{cal:master-PDE}(2) that the solution to the master PDE inherits this extra time regularity.

We recall and establish some symmetry properties of mixed derivatives, and in particular the second-order ones.

\begin{lemma} \label{cal:Clairaut-lm}
The following statements hold:
\begin{enumerate}
\item \cite[Corollary 5.89]{carmona_probabilistic_2018} Let $U \colon \sP_2 (\bR^d) \to \bR$ be of class $\sM_b^2$. For all $\mu \in \sP_2 (\bR^d)$ and $v_1, v_2 \in \bR^d$, we have that $\nabla_{v_1} \partial_\mu U (\mu; v_1)$ is symmetric and that $\partial_\mu^2 U (\mu; v_1, v_2) = [\partial_\mu^2 U ]^\top (\mu; v_2, v_1)$.

\item Let $U \colon \bR^r \times \sP_2 (\bR^d) \to \bR$ be of class $\sM_b^2$. Then $\nabla_x U (x, \cdot)$ is L-differentiable for every $x \in \bR^r$. Moreover, $\partial_\mu \nabla_x U (x, \mu; v) = [\nabla_x \partial_\mu U]^\top (x, \mu; v)$ for all $(\mu, x, v) \in \sP_2 (\bR^d) \times \bR^r \times \bR^d$.

\item Let $U \colon \bT \times \sP_2 (\bR^d) \to \bR$ be of class $\sM^{2}_b$. Assume that $\partial_t U (t, \cdot)$ is of class $\sM^{2}_b$ for every $t \in \bT$, with the constant $C$ in \zcref{cal:space1:ineq1} uniform in $t$; and that $\partial_\mu \partial_t U$ and $\nabla_v \partial_\mu \partial_t U$ are continuous. Then $\partial_\mu U (\cdot, \mu; v)$ and $\nabla_v \partial_\mu U (\cdot, \mu; v)$ are differentiable. Moreover, $\partial_t \partial_\mu U (t, \mu; v) = \partial_\mu \partial_t U (t, \mu; v)$ and $\partial_t \nabla_v \partial_\mu U (t, \mu; v) = \nabla_v \partial_\mu \partial_t U (t, \mu; v)$ for every $(t, \mu, v) \in \bT \times \sP_2 (\bR^d) \times \bR^d$.
\end{enumerate}
\end{lemma}

We establish smoothness of functions of a particular form.

\begin{lemma} \label{par-form}
Let $q, k \in \bN^*$ with $k \ge 2$ and $U \colon \Delta^{q+1} \to \bR$ be of class $\sM_b^k$. We define $V \colon \Delta^{q} \to \bR$ by $V (s) \coloneq \int_0^{s_q} U (s, t) \diff t$ for every $s=(s_1, \ldots, s_q) \in \Delta^q$. Then $V$ is of class $\sM_b^k$.
\end{lemma}

The proof of \zcref{par-form} is straightforward and thus omitted. We establish an analogue of Leibniz integral rule.

\begin{lemma} \label{cal:leibniz}
Let $q, r, k \in \bN$ and $U \colon \Omega \times \Delta^q \times \bR^r \times \sP_2 (\bR^d) \to \bR$ such that
\begin{itemize}
\item For each $(s, x, \mu) \in \Delta^q \times \bR^r \times \sP_2 (\bR^d)$, the map $U (\cdot, s, x, \mu) \colon \Omega \to \bR$ is measurable.
\item For each $\omega \in \Omega$, the map $U (\omega, \cdot) \colon \Delta^q \times \bR^r \times \sP_2 (\bR^d) \to \bR$ is of class $\sM_b^k$ with the constant $C_\omega$ in \zcref{cal:space1:ineq1}.
\item The map $\Omega \ni \omega \mapsto C_\omega$ is $\bP$-integrable.
\end{itemize}

We define $V \colon \Delta^q \times \bR^r \times \sP_2 (\bR^d) \to \bR$ by $V (s, x, \mu) \coloneq \bE [U(\cdot, s, x, \mu)]$. Then $V$ is of class $\sM_b^k$ and for every multi-index $(p,h,m,\ell)$ with $|(p,h,m,\ell)| \le k$
\begin{equation} \label{cal:leibniz:eq1}
\rD^{(p, h, m, \ell)} V (s, x, \mu; v) = \bE [ \rD^{(p, h, m, \ell)} U(\cdot, s, x, \mu ; v)] .
\end{equation}
\end{lemma}

We will encounter functions induced by integrating extra space variables in $\partial_\mu^k U$ w.r.t $\mu$. We establish the smoothness for these functions.

\begin{lemma} \label{cal:exp}
Let $q, r, k \in \bN$. The following statements hold:
\begin{enumerate}
\item Let $U_1, U_2 \colon \Delta^q \times \bR^{r} \times \sP_2 (\bR^d) \to \bR$ be of class $\sM_b^k$. We define $V \colon \Delta^q \times \bR^{r} \times \sP_2 (\bR^d) \to \bR$ by $V (s, x, \mu) \coloneq U_1 (s, x, \mu) U_2 (s, x, \mu)$. Then $V$ is of class $\sM_b^k$.

\item Let $U \colon \bR^d \times \sP_2 (\bR^d) \to \bR$ be of class $\sM_b^k$. We define $V \colon \sP_2 (\bR^d) \to \bR$ by $V (\mu) \coloneq \bE [U(Y, \mu)]$ where $\lbrbrak Y \rbrbrak = \mu$. Then $V$ is $k$-times L-differentiable. For $v=(v_1, \ldots,v_k) \in (\bR^d)^k$ and $v_{-i} = (v_1, \ldots, v_{i-1}, v_{i+1}, \ldots, v_k)$,
\begin{equation} \label{cal:exp:eq0}
\partial_\mu^k V (\mu; v) = \bE[\partial_\mu^k U(Y, \mu; v)] + \sum_{i=1}^k \partial_\mu^{k-i} \nabla_{v_i} \partial_\mu^{i-1} U(v_i, \mu; v_{-i}) .
\end{equation}

\item Let $U \colon \Delta^q \times (\bR^d)^{r} \times \sP_2 (\bR^d) \to \bR$ be of class $\sM_b^k$. We define $V \colon \Delta^q \times \sP_2 (\bR^d) \to \bR$ by $V (s, \mu) \coloneq \bE [U(s, Y_1, \ldots, Y_{r}, \mu)]$ where $\{ Y_1, \ldots, Y_r \}$ are i.i.d according to $\mu$. Then $V$ is of class $\sM_b^k$.

\item Let $U \colon \Delta^{q} \times \sP_2 (\bR^d) \to \bR$ be of class $\sM_b^{k}$ with $k \ge 2$. Assume that $b$ and $\sigma$ are of class $\sM_b^{k-2}$. We define $V \colon \Delta^{q} \to \bR$ by $V (s) \coloneq U(s, \mu_{s_q})$ for every $s = (s_1, \ldots, s_q)$ where $\{ \mu_t \}_{ t \in \bT }$ is the flow of the marginal distributions of the solution to \zcref{intro:main-sde}. Then $V$ is of class $\sM_b^{k}$.
\end{enumerate}
\end{lemma}

The proof of \zcref{cal:exp}(1) is straightforward thus omitted. We recall a generalization of Itô's lemma.

\begin{lemma} \label{cal:Ito-lemma}
\cite[Proposition 5.102]{carmona_probabilistic_2018} Let $\{ b_t \}_{ t \in \bT }, \{ \eta_t \}_{ t \in \bT }, \{ \sigma_t \}_{ t \in \bT }$ and $\{ \gamma_t \}_{ t \in \bT }$ be $\bF$-progressively measurable processes on $(\Omega, \cA, \bF, \bP)$ and with values in $\bR^d, \bR^d, \bR^d \otimes \bR^d$ and $\bR^d \otimes \bR^d$ respectively. We assume that
\begin{align}
\lbrbrak Y_0 \rbrbrak \in \sP_2 (\bR^d)
, \quad
\diff Y_t = b_t \diff t + \sigma_t \diff B_t
, \quad
\diff X_t = \eta_t \diff t + \gamma_t \diff B_t , \\
\bE \Big [ \int_{\bT} \{ |b_t|^2 + |\sigma_t|^4 \} \diff t \Big ] < \infty 
\qtextq{and}
\bP \Big [ \int_{\bT} \{ |\eta_t| + |\gamma_t|^2 \} \diff t < \infty \Big ] = 1 .
\end{align}

Let $U \colon \bT \times \bR^d \times \sP_2 (\bR^d) \to \bR$ be of class $\sM_b^2$. Then, $\bP$-a.s., for all $t \in \bT$, it holds:
\begin{myalign}
& U (t, X_t, \lbrbrak Y_t \rbrbrak ) = U (0, X_0, \lbrbrak Y_0 \rbrbrak ) + \int_0^t [\nabla_x U]^\top (s, X_s, \lbrbrak Y_s \rbrbrak ) \gamma_s \diff B_s \\
& + \int_0^t \bigg \{  \partial_s U (s, X_s, \lbrbrak Y_s \rbrbrak ) + \eta_s \cdot \nabla_x U (s, X_s, \lbrbrak Y_s \rbrbrak ) + \frac{1}{2} \gamma_s \gamma_s^\top : \nabla_x^2 U (s, X_s, \lbrbrak Y_s \rbrbrak ) \\
& + \tilde \bE \Big [ \tilde b_s \cdot \partial_\mu U (s, X_s, \lbrbrak Y_s \rbrbrak ; \tilde Y_s ) + \frac{1}{2} \tilde \sigma_s \tilde \sigma_s^\top : \nabla_v \partial_\mu U (s, X_s, \lbrbrak Y_s \rbrbrak  ; \tilde Y_s ) \Big ] \bigg \} \diff s . \\
\end{myalign}

Above, $(\tilde Y_t, \tilde b_t, \tilde \sigma_t)_{t \in \bT}$ is a copy of $(Y_t, b_t, \sigma_t)_{t \in \bT}$ defined on a copy $(\tilde \Omega, \tilde \cA, \tilde \bP, \tilde \bE)$ of $(\Omega, \cA, \bP, \bE)$.
\end{lemma}

We establish a generalization of the master PDE.

\begin{theorem} \label{cal:master-PDE}
Let $q, k \in \bN$ with $k \ge 2$ and $\tilde \cU \colon \Delta^{q} \times \sP_2 (\bR^d) \to \bR$. Let $(b, \sigma, \tilde \cU)$ be of class $\sM_b^k$. For $(r, \mu) \in [0, T) \times \sP_2 (\bR^d)$, we consider the SDE
\begin{equation} \label{cal:master-PDE:eq1}
\begin{dcases}
 X_t^{r, \mu} & = X_r^{r, \mu} + \int_r^t b (s, X_s^{r, \mu}, \lbrbrak X_s^{r, \mu} \rbrbrak) \diff s + \int_r^t \sigma (s, X_s^{r, \mu}, \lbrbrak X_s^{r, \mu} \rbrbrak) \diff B_s
\qtextq{for} t \in  [r, T] , \\
\lbrbrak X_r^{r, \mu} \rbrbrak & = \mu ,
\end{dcases}
\end{equation}

We define $\tilde \cV \colon \Delta^{q+1} \times \sP_2 (\bR^d) \to \bR$ by $\tilde \cV (s, r, \mu) \coloneq \tilde \cU (s, \lbrbrak X_{s_q}^{r, \mu} \rbrbrak)$ for all $s = (s_1, \ldots, s_{q}) \in \Delta^{q}$ and $(r, \mu) \in [0, s_q] \times \sP_2 (\bR^d)$. Then the following three statements hold:

\begin{enumerate}
\item The SDE \zcref{cal:master-PDE:eq1} has a unique weak solution and thus $\tilde \cV$ is well-defined.

\item The map $\tilde \cV$ is of class $\sM_b^k$.

\item For all $s = (s_1, \ldots, s_{q}) \in \Delta^q$ and $(r, \mu) \in (0, s_q) \times \sP_2 (\bR^d)$,
\begin{equation} \label{cal:master-PDE:eq2}
\begin{dcases}
\begin{multlined}[t]
\partial_r \tilde \cV (s, r, \mu) + \int_{\bR^d} \Big [ b (r, v, \mu) \cdot \partial_\mu \tilde \cV (s, r, \mu ; v) + \frac{1}{2} a (r, v, \mu) : \nabla_v \partial_\mu \tilde \cV (s, r, \mu ; v) \Big ] \diff \mu (v) = 0 , \\
\end{multlined} \\
\tilde \cV(s, s_q, \mu) = \tilde \cU (s, \mu) .
\end{dcases}
\end{equation}
\end{enumerate}
\end{theorem}

The well-posedness of \zcref{cal:master-PDE:eq1} under the given conditions is now a classical fact. See e.g. \cite[Proposition 1.2]{jourdain_nonlinear_2008} and \cite[Theorem 4.21]{carmona2018probabilistic}. The fact that $\tilde \sV$ solves \zcref{cal:master-PDE:eq2} is established in \cite[Theorem 7.2]{buckdahn_mean-field_2017}. Therefore, in \zcref{cal:master-PDE:prf}, we only need to prove the fact that $\tilde \cV$ is of class $\sM^k_b$. This is done in two simplified settings: in the first case, considered by \cite{chassagneux_weak_2022} in their proof to avoid too complicated expressions, the coefficients $(b, \sigma)$ do not depend on spatial variable; and in the second case, important in view of applications, the coefficients $(b, \sigma)$ do not depend on distribution variable.

\cite[Proposition 6.1]{chaudru_de_raynal_well-posedness_2022} considers a more general form where $\tilde \sU$ also depends on spatial variable and the right-hand side of the first equation in \zcref{cal:master-PDE:eq2} is a function. Let us mention some approaches to prove the smoothness of the solution to \zcref{cal:master-PDE:eq2}. \cite{crisan_smoothing_2018} employs Malliavin calculus and considers $\tilde \sU$ that also depends on spatial variable. \cite{chaudru_de_raynal_strong_2020} uses parametrix method and considers $b, \sigma$ of integral form, i.e.,
\[
b(t, x, \mu) = \int_{\bR^d} \tilde b (t, x, y) \diff \mu (y) \qtextq{and} \sigma (t, x, \mu) = \int_{\bR^d} \tilde \sigma (t, x, y) \diff \mu (y),
\]
for some functions $\tilde b \colon \bT \times \bR^d \times \bR^d \to \bR^d$ and $\tilde \sigma \colon \bT \times \bR^d \times \bR^d \to \bR^d \otimes \bR^d$. \cite{buckdahn_mean-field_2017,chassagneux_probabilistic_2022} use a variational approach and prove the smoothness of $\tilde \cV$ in measure up to the second order. \cite{chassagneux_weak_2022} generalizes the result of \cite{buckdahn_mean-field_2017} to arbitrary order. The new ingredient in \zcref{cal:master-PDE} is that $\tilde \cV$ inherits higher-order regularity in time from $b$ and $\sigma$.

For ease of reading, the proofs in this section are deferred until \zcref{calculus_on-P2:proof}.

\section{Main result}

Our main result is

\begin{theorem} \label{main-thm}
Assume that $(b, \sigma, \cU)$ are of class $\sM_b^\infty$. Then there exists a sequence $\{ C_i \}_{i \ge 1} \subset \bR$ such that for every $n, m \in \bN^*$,
\begin{equation} \label{main-thm:cnst}
\cU (\mu^{n}_T) - \cU (\mu_T) = \sum_{i=1}^m \frac{C_i}{n^i} + \cO \Big ( \frac{1}{n^{m+1}} \Big ) .
\end{equation}
\end{theorem}

The remaining of this section is dedicated to the proof of \zcref{main-thm}. We recall from \zcref{intro} that $\eps_n = T/n$ and $t_k = k \eps_n$ for every $k \in \llbracket 0, n \rrbracket$. For each $k \in \llbracket 0, n-1 \rrbracket$, we denote
\begin{equation}\label{main-thm:proof:eq1}
\begin{alignedat}{3} 
\rI_k & \coloneq [t_k, t_{k+1}] , & \quad b_k & \coloneq b(t_k, X^n_{t_k}, \mu^n_{t_k}) , \\
\sigma_k & \coloneq \sigma (t_k, X^n_{t_k}, \mu^n_{t_k}) , & \quad a_k & \coloneq a(t_k, X^n_{t_k} , \mu^n_{t_k}) .
\end{alignedat}
\end{equation}

By \zcref{intro:euler-scheme}, we have on the interval $\rI_k$ that
\[
\diff X^n_t = b_k \diff t + \sigma_k \diff B_t .
\]

For $U \colon \rI_k \times \sP_2 (\bR^d) \to \bR$ of class $\sM_b^2$, we define $\rL_{t_k} U \colon \rI_k \times \sP_2 (\bR^d) \to \bR$ by
\begin{equation} \label{main-thm:proof:diff-oper}
\begin{multlined}[t]
\rL_{t_k} U (t, \mu) \coloneq \partial_t U (t, \mu) + \bE \Big [ b_k \cdot \partial_\mu U (t, \mu ; X^n_t) + \frac{1}{2} a_k : \nabla_{v} \partial_\mu U (t, \mu ; X^n_t) \Big ] .
\end{multlined}
\end{equation}

For $(q, k) \in \bN^* \times \llbracket 0, n \rrbracket$, we denote
\[
\Delta^q_{t_k} \coloneq \{ s = (s_1, \ldots, s_q) \in \Delta^q ; s_q \ge t_k \} .
\]

The next proposition gives an expansion of $\tilde \cU (\tilde s,t_K, \mu^{n}_{t_K}) - \tilde \cU (\tilde s,t_K,\mu_{t_K})$ when $\tilde \cU \colon \Delta^{q} \times \sP_2 (\bR^d) \to \bR$, $K\in\llbracket 0, n \rrbracket$ and the variable $\tilde s$, not needed when $q=1$ otherwise belongs to $\Delta^{q-1}_{t_K}$ (we use the convention that there is no variable $\tilde s\in\Delta^{q-1}_{t_K}$ when $q=1$). We even deal with the case when $\tilde \cU \colon \sP_2 (\bR^d) \to \bR$ by using the similar convention that there is no variable $s\in\Delta^q$ when $q=0$.

Let $\bar \tau^n_t \coloneq t_{k+1}$ if $t \in [t_k, t_{k+1})$ for some $k \in \llbracket 0, n-1  \rrbracket$.

\begin{proposition} \label{main-thm:proof:repre}
Let $q \in \bN$ and $\tilde \cU \colon \Delta^{q} \times \sP_2 (\bR^d)  \to \bR$. Assume that $(b, \sigma, \tilde \cU)$ are of class $\sM_b^\infty$. We define $\tilde \cV \colon \Delta^{q+1} \times \sP_2 (\bR^d)  \to \bR$ by $\tilde \cV (s, t, \mu) \coloneq \tilde \cU (s, \lbrbrak X_{s_q}^{t, \mu} \rbrbrak)$ for $s = (s_1, \ldots, s_{q}) \in \Delta^{q}$ and $(t, \mu) \in [0, s_q] \times \sP_2 (\bR^d)$ with convention $s_0=T$ in case $q=0$. For $(m,k)\in \bN^*\times\llbracket 0, n-1 \rrbracket$, we denote by $\rL^m_{t_k} \coloneq \rL_{t_k} \circ \cdots \circ \rL_{t_k}$ the $m$-times composition of $\rL_{t_k}$. For $s \in \Delta^q_{t_{k+1}}$, we denote by $\rL_{t_k}^m \tilde \cV (s, t, \mu)$ the value at $(t, \mu)\in\rI_k \times \sP_2(\bR^d)$ of the application of $\rL_{t_k}^m$ on $\tilde \cV (s, \cdot, \cdot)$. Then
\begin{enumerate}
\item  For each $(m,k,s)\in \bN^*\times\llbracket 0, n-1 \rrbracket \times \Delta^q_{t_{k+1}}$, the map $\rI_k \times \sP_2(\bR^d)\ni(t,\mu)\mapsto\rL^m_{t_k} \tilde \cV (s,t,\mu)$ is smooth and $\rL_{t_k} \tilde \cV (s, t_k, \mu^n_{t_k}) = 0$.

\item For each $m \in \bN^*$, there exists a map $\varphi_m \colon \Delta^{q+1} \times \sP_2 (\bR^d)  \to \bR$ that is of class $\sM_b^\infty$; depends on $(q, b, \sigma, \tilde \cU)$ but not on $n$; and satisfies  $\varphi_{m} (s, t_k, \mu^n_{t_k})=\frac{1}{m!}\rL_{t_k}^{m+1} \tilde \cV  (s, t_k, \mu^n_{t_k})$ for every $(k, s) \in\llbracket 0, n-1 \rrbracket \times \Delta^q_{t_{k+1}}$.
\item When $q\ge1$, for each $(m, K, \tilde s) \in \bN \times \llbracket 0, n \rrbracket \times \Delta^{q-1}_{t_K}$,
\begin{myalign}
& \tilde \cU (\tilde s,t_K, \mu^{n}_{t_K}) - \tilde \cU (\tilde s,t_K,\mu_{t_K}) \\
& = \begin{myaligned}[t] \label{main-thm:proof:repre:eq1}
& \sum_{i=1}^{m} \int_{0}^{t_K} (\bar \tau^n_t - t)^i \varphi_{i}  (\tilde s,t_K,\tau^n_t, \mu^n_{\tau^n_t}) \diff t \\
& + \frac{1}{(m+1)!}\int_{0}^{t_K} (\bar \tau^n_t - t)^{m+1} \rL_{\tau^n_t}^{m+2}\tilde \cV(\tilde s,t_K, t, \mu^n_t) \diff t  ,
\end{myaligned}
\end{myalign} 
with the convention that there is no summation over $i$ on the right-hand side when $m=0$. When $q=0$,
\begin{myalign} 
& \tilde \cU (\mu^{n}_{T}) - \tilde \cU (\mu_{T}) \\
& = \begin{myaligned}[t] \label{main-thm:proof:repre:eq1bis}
& \sum_{i=1}^{m} \int_{0}^{t_K} (\bar \tau^n_t - t)^i \varphi_{i}  (\tau^n_t, \mu^n_{\tau^n_t}) \diff t \\
& + \frac{1}{(m+1)!}\int_{0}^{t_K} (\bar \tau^n_t - t)^{m+1} \rL_{\tau^n_t}^{m+2}\tilde \cV(t, \mu^n_t) \diff t.
\end{myaligned}
\end{myalign}  
\end{enumerate}
\end{proposition}

Note that the equality $\rL_{t_k} \tilde \cV (s, t_k, \mu^n_{t_k}) = 0$ in \zcref{main-thm:proof:repre}(1), which is a consequence of the master PDE stated in \zcref{cal:master-PDE}(3), explains why there is no term with $i=0$ in the expansion \zcref{main-thm:proof:repre:eq1}, i.e., no term with order $0$ in $\eps_n$ as seen from the next result. \zcref{main-thm:proof:int} is the key to expand time integrals involving $(\tau^n_t,\bar\tau^n_t)$ like those in the summation from $1$ to $m$ on the right-hand side of \zcref{main-thm:proof:repre:eq1} into powers of $\eps_n$ with coefficients not depending on $n$.

\begin{lemma} \label{main-thm:proof:int}
There exist constants $\{ \beta_j \}_{j \in{\mathbb N}} \subset \bR$ with $\beta_0 = 1$  such that for every $(m,K, i) \in \bN^*\times \llbracket 1, n \rrbracket \times \bN$ and every $f \colon \bT \to \bR$ $m$-times continuously differentiable,
\begin{equation} \label{main-thm:proof:int:eq0}
(i+1) \int_0^{t_K} (\bar \tau^n_t - t)^i f (\tau^n_t) \diff t = \sum_{j=0}^{m-1} \beta_j \eps_n^{i+j} \int_0^{t_K} \partial_t^{j} f (t) \diff t + \cO ( \eps_n^{i+m} ) .
\end{equation}

Above, we use the convention that $\partial_t^{0} f = f$.
\end{lemma}

\begin{proof}[Proof of \zcref{main-thm}] \label{main-thm:prf}
Applying \zcref{main-thm:proof:repre}(3) with $q=0$ and $\tilde \cU = \cU$, we get smooth functions $\varphi_{(i_1)} \colon \bT \times \sP_2 (\bR^d)  \to \bR$ for $i_1\in\llbracket 1, m \rrbracket$ that do not depend on $n$ and such that
\begin{myalign}
& \cU (\mu^{n}_T) - \cU (\mu_T) \\
& = \cO ( \eps_n^{m+1} ) + \sum_{i_1=1}^{m} \int_{0}^{T} (\bar \tau^n_{s_1} - s_1)^{i_1} \varphi_{(i_1)} (\tau^n_{s_1}, \mu^n_{\tau^n_{s_1}}) \diff s_1 \label{main-thm:proof:int6} \\
& = \begin{myaligned}[t]
& \cO ( \eps_n^{m+1} ) + \sum_{i_1=1}^{m} \int_{0}^{T} (\bar \tau^n_{s_1} - s_1)^{i_1} \varphi_{(i_1)} (\tau^n_{s_1}, \mu_{\tau^n_{s_1}}) \diff s_1 \\
& + \sum_{i_1=1}^{m} \int_{0}^{T} (\bar \tau^n_{s_1} - s_1)^{i_1} \{ \varphi_{(i_1)} (\tau^n_{s_1}, \mu^n_{\tau^n_{s_1}}) - \varphi_{(i_1)} (\tau^n_{s_1}, \mu_{\tau^n_{s_1}}) \} \diff s_1 .
\end{myaligned} 
\end{myalign}

For $i_1\in\llbracket 1, m \rrbracket$, applying \zcref{main-thm:proof:repre}(3) with $(q, t_K, s) = (1, \tau^n_{s_1}, \tau^n_{s_1})$ and $\tilde \cU(s,\mu) = \varphi_{(i_1)}(s,\mu)$, we get smooth functions $\varphi_{(i_1, i_2)} \colon \Delta^2 \times \sP_2 \to \bR$ that do not depend on $n$ such that
\begin{align}
\begin{multlined}[t]
\varphi_{(i_1)} (\tau^n_{s_1}, \mu^n_{\tau^n_{s_1}}) - \varphi_{(i_1)} (\tau^n_{s_1}, \mu_{\tau^n_{s_1}}) \\
= \cO ( \eps_n^{m-i_1+1} ) + \sum_{i_2=1}^{m-i_1} \int_{0}^{\tau^n_{s_1}} (\bar \tau^n_{s_2} - s_1)^{i_2} \varphi_{(i_1, i_2)} (\tau^n_{s_1}, \tau^n_{s_2}, \mu^n_{\tau^n_{s_2}}) \diff s_2,
\end{multlined} 
\end{align}
where there is no term of order $0$ in $\eps_n$ on the right-hand side. When $i_1=m$, there is no sum over $i_2$ in the right-hand-side and $\int_{0}^{T} (\bar \tau^n_{s_1} - s_1)^{m} \{ \varphi_{(m)} (\tau^n_{s_1}, \mu^n_{\tau^n_{s_1}}) - \varphi_{(m)} (\tau^n_{s_1}, \mu_{\tau^n_{s_1}}) \} \diff s_1= \cO ( \eps_n^{m+1} )$. We deduce that
\begin{myalign}
& \cU (\mu^{n}_T) - \cU (\mu_T) = \cO ( \eps_n^{m+1} ) + \sum_{i_1=1}^{m} \int_{0}^{T} (\bar \tau^n_{s_1} - s_1)^{i_1} \varphi_{(i_1)} (\tau^n_{s_1}, \mu_{\tau^n_{s_1}}) \diff s_1 \\
& + \sum_{i_1=1}^{m-1} \sum_{i_2=1}^{m-i_1} \int_{0}^{T} (\bar \tau^n_{s_1} - s_1)^{i_1} \int_{0}^{\tau^n_{s_1}} (\bar \tau^n_{s_2} - s_2)^{i_2} \varphi_{(i_1,i_2)} (\tau^n_{s_1},\tau^n_{s_2}, \mu^n_{\tau^n_{s_2}}) \diff s_2 \diff s_1.
\end{myalign}

Thus
\begin{myalign}
& \cU (\mu^{n}_T) - \cU (\mu_T) \\
& =  \begin{myaligned}[t]
& \cO ( \eps_n^{m+1} ) + \sum_{p=1}^m \sum_{j=1}^2 \sum_{\substack{i_1, \ldots, i_j \in \bN^* \\ i_1 + \cdots + i_j=p}} \int_{0}^{T} (\bar \tau^n_{s_1} - s_1)^{i_1} \ldots \int_{0}^{\tau^n_{s_{j-1}}} (\bar \tau^n_{s_j} - s_j)^{i_j} \\
& \times \varphi_{(i_1,\ldots,i_j)} (\tau^n_{s_1}, \ldots, \tau^n_{s_j}, \mu_{\tau^n_{s_{j}}}) \diff s_j \ldots \diff s_1 \\
& + \sum_{p=2}^m \sum_{\substack{i_1, i_2 \in \bN^* \\ i_1 + i_2=p}} \int_{0}^{T} (\bar \tau^n_{s_1} - s_1)^{i_1} \int_{0}^{\tau^n_{s_1}} (\bar \tau^n_{s_2} - s_2)^{i_2} \\
& \times \{ \varphi_{(i_1, i_2)} (\tau^n_{s_1}, \tau^n_{s_2}, \mu^n_{\tau^n_{s_2}}) - \varphi_{(i_1,i_2)} (\tau^n_{s_1}, \tau^n_{s_2}, \mu_{\tau^n_{s_2}}) \} \diff s_2 \diff s_1 .
\end{myaligned} 
\end{myalign}

We repeat the above procedure by using \zcref{main-thm:proof:repre}(3) to expand $\varphi_{(i_1, \ldots, i_\ell)} (\tau^n_{s_1}, \ldots, \tau^n_{s_\ell}, \mu^n_{\tau^n_{s_\ell}}) - \varphi_{(i_1, \ldots, i_\ell)} (\tau^n_{s_1}, \ldots, \tau^n_{s_\ell}, \mu_{\tau^n_{s_\ell}})$ when $i_1, \ldots, i_\ell \in \bN^*$ are such that $i_1 + \cdots + i_\ell \le m-1$ (i.e. as long as $(\bar \tau^n_{s_1} - s_1)^{i_1}\times\cdots\times (\bar \tau^n_{s_\ell} - s_\ell)^{i_\ell}$ does not provide the desired order $m$ in $\eps_n$) and to put in the $\cO ( \eps_n^{m+1} )$ term the contribution of this difference when $i_1+\cdots+i_\ell=m$. We get smooth functions $\varphi_{(i_1, \ldots, i_j)} \colon \Delta^j \times \sP_2 (\bR^d)  \to \bR$ that are independent of $n$ and such that
\begin{myalign} 
& \cU (\mu^{n}_T) - \cU (\mu_T) \\
& = \begin{myaligned}[t]
& \cO ( \eps_n^{m+1} ) + \sum_{p=1}^m \sum_{j=1}^p \sum_{\substack{i_1,\ldots,i_j \in \bN^* \\ i_1+\cdots+i_j=p}} \int_{0}^{T} (\bar \tau^n_{s_1} - s_1)^{i_1}
\ldots \int_{0}^{\tau^n_{s_{j-1}}} (\bar \tau^n_{s_j} - s_j)^{i_j} \\
& \times \varphi_{(i_1,\ldots,i_j)} (\tau^n_{s_1},\ldots,\tau^n_{s_j}, \mu_{\tau^n_{s_{j}}}) \diff s_j \ldots \diff s_1 .
\end{myaligned}\label{main-thm:proof:int3}
\end{myalign}

By \zcref{main-thm:proof:int}, 
\begin{align}
\int_{0}^{T} (\bar \tau^n_{t} - t)^{i_1} \varphi_{i_1} (\tau^n_t,\mu_{\tau^n_t}) \diff t = \cO ( \eps_n^{m-p+i_1+1} ) + \sum_{k_1=0}^{m-p}\eps_n^{i_1+k_1}\phi_{(i_1; k_1)},
\end{align}
where, for $k_1 \in \llbracket 0, m-p \rrbracket$, 
\begin{align}
\phi_{(i_1;k_1)} \coloneq \frac{\beta_{k_1}}{i_1+1} \int_0^{T} \diff_t^{k_1} \varphi_{(i_1)} (t,\mu_t) \diff t,\label{defphiik}
\end{align}
with $\diff_t^{k_1} \varphi_{(i_1)} (t,\mu_t)$ denoting the $k_1$-th order derivative of the function $t\mapsto \varphi_{(i_1)} (t,\mu_t)$. When $j\ge 2$, still by \zcref{main-thm:proof:int}, 
\begin{myalign}
& \int_{0}^{\tau^n_{s_{j-1}}} (\bar \tau^n_{t} - t)^{i_j} \varphi_{(i_1, \ldots, i_j)} (\tau^n_{s_1}, \ldots, \tau^n_{s_{j-1}}, \tau^n_t, \mu_{\tau^n_t}) \diff t \\
& = \cO ( \eps_n^{m-p+i_j+1} ) + \sum_{k_j=0}^{m-p}\eps_n^{i_j+k_j}\phi_{(i_1, \ldots, i_j; k_j)} (\tau^n_{s_1}, \ldots, \tau^n_{s_{j-1}}) ,
\end{myalign}
where, for $k_j \in \llbracket 0, m-p \rrbracket$, $\phi_{(i_1, \ldots, i_j; k_j)} \colon \Delta^{j-1} \to \bR$ is defined as
\begin{align}
\phi_{(i_1, \ldots, i_j; k_j)} (s_1, \ldots, s_{j-1}) \coloneq \frac{\beta_{k_j}}{i_j+1} \int_0^{s_{j-1}} \diff_t^{k_j} \varphi_{(i_1, \ldots, i_j)} (s_1, \ldots, s_{j-1}, t,\mu_t) \diff t,
\end{align}
where $\diff_t^{k_j} \varphi_{(i_1, \ldots, i_j)} (s_1, \ldots, s_{j-1}, t,\mu_t)$ denotes the $k_j$-th order derivative of the function $t\mapsto\varphi_{(i_1, \ldots, i_j)} (s_1, \ldots, s_{j-1}, t,\mu_t)$. The functions $\phi_{(i_1, \ldots, i_j; k_j)}$ are smooth according to \zcref{cal:exp}(4) and \zcref{par-form}.

To sum up, we are able to get rid of the integral w.r.t $s_j$ in \zcref{main-thm:proof:int3}. Moreover, we have transferred implicit dependence of $I$ on $(\bar \tau^n_t,\tau^n_t)$ to explicit dependence on $\eps_n$ and functions $\{ \phi_{(i_1, \ldots, i_j; k_j)} \}_{ k_j \in \llbracket 0, m-p \rrbracket }$ (not depending on $n$). Along the way, we have a residual of order $\cO ( \eps_n^{m-p+i_j+1} )$. This residual multiplied by the remaining iterated integral will be of order $\cO (\eps_n^{m+1})$. Hence
\begin{myalign} \label{main-thm:proof:int4}
& \cU (\mu^{n}_T) - \cU (\mu_T) \\
&= \begin{myaligned}[t]
& \cO ( \eps_n^{m+1} ) +\sum_{p=1}^m\sum_{k_1=0}^{m-p}\eps_n^{p+k_1}\phi_{(p;k_1)}\\
& +  \sum_{p=1}^m \sum_{j=2}^p \sum_{\substack{i_1,\ldots,i_j \in \bN^* \\ i_1+\cdots+i_j=p}} \sum_{k_j=0}^{m-p} \eps_n^{i_j+k_j} \int_{0}^{T} (\bar \tau^n_{s_1} - s_1)^{i_1} \ldots \\
& \ldots \int_{0}^{\tau^n_{s_{j-2}}} (\bar \tau^n_{s_{j-1}} - s_{j-1})^{i_{j-1}} \phi_{(i_1,\ldots,i_j;k_j)} (\tau^n_{s_1},\ldots,\tau^n_{s_{j-1}}) \diff s_{j-1} \ldots \diff s_1 .
\end{myaligned}
\end{myalign}

For $j\ge 2$, we define by backward induction 
\begin{align}
\phi_{(i_1, \ldots, i_j; k_j, \ldots, k_\ell)}(s_1,\ldots,s_{\ell-1}) & \coloneq \begin{multlined}[t]
\frac{\beta_{k_\ell}}{i_\ell+1}\int_0^{s_{\ell -1}}\partial^{k_\ell}_t\phi_{(i_1, \ldots, i_j; k_j, \ldots, k_{\ell+1})}(s_1,\ldots,s_{\ell-1},t)\diff t \\
\textq{when} \ell\in\llbracket 2, j-1 \rrbracket ,
\end{multlined} \\
\phi_{(i_1, \ldots, i_j; k_j, \ldots, k_1)} & \coloneq \frac{\beta_{k_1}}{i_1+1}\int_0^{T}\partial^{k_1}_t\phi_{(i_1, \ldots, i_j; k_j, \ldots, k_{2})}(t)\diff t ,
\end{align}
when for $\ell\in\llbracket 1, j-1 \rrbracket$, $k_1,\ldots,k_\ell\in{\mathbb N}$ satisfy $k_1+\cdots+k_\ell\le m-(i_1+\cdots+i_j)$.

Repeating the above procedure to get rid of the integrals with respect to $s_{j-1},\ldots,s_1$, we obtain
\begin{myalign} 
& \cU (\mu^{n}_T) - \cU (\mu_T) \\
& = \begin{myaligned}[t]
& \cO ( \eps_n^{m+1} ) + \sum_{p=1}^m \sum_{j=1}^p \sum_{\substack{i_1,\ldots,i_j \in \bN^* \\ i_1+\cdots+i_j=p}} \sum_{k_j=0}^{m-p} \eps_n^{i_j+k_j} \sum_{k_{j-1}=0}^{m-p-k_j} \eps_n^{i_{j-1}+k_{j-1}} \ldots \\
& \ldots \sum_{k_{1}=0}^{m-p-(k_j+\cdots+k_2)} \eps_n^{i_{1}+k_{1}}\phi_{(i_1, \ldots, i_j; k_j, \ldots, k_1)} 
\end{myaligned} \\
& = \cO ( \eps_n^{m+1} ) + \sum_{p=1}^m \sum_{q=0}^{m-p} \eps_n^{p+q} \sum_{j=1}^p \sum_{\substack{i_1,\ldots,i_j \in \bN^* \\ i_1+\cdots+i_j=p}} \sum_{\substack{k_1,\ldots,k_j \in \bN \\ k_1+\cdots+k_j=q}} \phi_{(i_1, \ldots, i_j; k_j, \ldots, k_1)} \\
& = \cO \left( \frac 1{n^{m+1}}\right) + \sum_{r=1}^m \frac{T^r}{n^r}\sum_{p=1}^r  \sum_{j=1}^p \sum_{\substack{i_1,\ldots,i_j \in \bN^* \\ i_1+\cdots+i_j=p}} \sum_{\substack{k_1,\ldots,k_j \in \bN \\ k_1+\cdots+k_j=r-p}} \phi_{(i_1, \ldots, i_j; k_j, \ldots, k_1)} .\label{main-thm:proof:int5}
\end{myalign}

This is the desired expansion.
\end{proof}

\subsection{Proof of \zcref{main-thm:proof:repre}}

The proof relies on the following result where, for $r\in\bN^*$, we denote $Y^{[r]}_t \coloneq (Y^{(1)}_t, \ldots , Y^{(r)}_t)$ where $\{ (Y^{(i)}_t, B^{(i)}_t)_{t \in \bT} \}_{i \ge 1}$ are i.i.d copies of $(X^n_t, B_t)_{t \in \bT }$. For $(k, i) \in \llbracket 0, n-1  \rrbracket \times \bN^*$, we construct $(b^{(i)}_k, \sigma^{(i)}_k, a^{(i)}_k)$ by replacing $X^n_{t_k}$ in \zcref{main-thm:proof:eq1} with $Y^{(i)}_{t_k}$. 

\begin{lemma} \label{cal:time-diff2}
Let $(q, r, k) \in \bN \times \bN \times \llbracket 0, n-1 \rrbracket$ with $q \ge 2$. Let $U \colon \rI_k \times (\bR^d)^{2r} \times \sP_2 (\bR^d)  \to \bR$ be of class $\sM_b^{q}$. Assume that $(b, \sigma)$ are of class $\sM_b^{q-2}$. We define $V \colon \rI_k \times \sP_2 (\bR^d)  \to \bR$ by $V(t, \mu) \coloneq \bE [U(t, Y^{[r]}_t, Y^{[r]}_{t_k}, \mu)]$. Then $V$ is of class $\sM_b^{q}$. With notation $U(t, y_1, \ldots, y_r, x_1, \ldots, x_r, \mu)$, we have
\[
\partial_t V(t, \mu) = \begin{myaligned}[t]
& \bE \Big [ \partial_t U(t, Y^{[r]}_t, Y^{[r]}_{t_k}, \mu) + \sum_{i=1}^r b_k^{(i)} \cdot \nabla_{y_i} U(t, Y^{[r]}_t, Y^{[r]}_{t_k}, \mu) \\
& + \frac{1}{2} \sum_{i=1}^r a_k^{(i)} : \nabla_{y_i}^2 U(t, Y^{[r]}_t, Y^{[r]}_{t_k}, \mu) \Big ] .
\end{myaligned}
\]
\end{lemma}

\begin{proof}
By \zcref{cal:leibniz}, $\sP_2(\bR^d)\ni\mu \mapsto V(t, \mu)$ is of class $\sM_b^{q}$, and it holds for $(p, v) \in \llbracket 1, q \rrbracket \times (\bR^d)^p$ that $\partial_\mu^p V (t, \mu; v) = \bE [ W_{\mu,v} (t, Y^{[r]}_t, Y^{[r]}_{t_k})]$ where $W_{\mu,v} \colon \rI_k \times (\bR^d)^{2r} \to (\bR^d)^{\otimes p}$ defined as $W_{\mu,v} (t, y, x) \coloneq \partial_\mu^p U (t, y, x, \mu; v)$ is such that $(\mu,v, t, y, x) \mapsto W_{\mu,v} (t, y, x)$ is of class $\sM_b^{q-p}$.

Let us assume that $p\le q-2$. To deal with the time derivatives, we define $\rL W_{\mu,v} \colon \rI_k \times (\bR^d)^{2r} \to (\bR^d)^{\otimes p}$ by $[\rL W_{\mu,v}]_i \coloneq \rL [W_{\mu,v}]_i$ for $i \in \llbracket 1, d \rrbracket^p$, i.e., component-wise with the operator $\rL$ acting on $W \colon \rI_k \times (\bR^d)^{2r} \to \bR$ of class $\sM_b^2$ by
\[
\rL W (t, y, x) \coloneq \begin{myaligned}[t]
& \partial_t W (t, y, x) + \sum_{i=1}^r b (t_k, x_i, \mu^n_{t_k}) \cdot \nabla_{y_i} W (t, y, x) \\
& + \frac{1}{2} \sum_{i=1}^r a(t_k, x_i, \mu^n_{t_k}) : \nabla_{y_i}^2 W (t, y, x) ,
\end{myaligned}
\]
where $y= (y_1, \ldots, y_r)$ and $x = (x_1, \ldots, x_r)$ belong to $(\bR^d)^r$. Then $(\mu,v,t,y,x)\mapsto\rL W_{\mu,v} (t,y,x)$ is of class $\sM_b^{q-p-2}$. Applying Itô's lemma and taking expectations, we get that 
\[
\forall t\in\rI_k,\;\partial_\mu^p V (t, \mu; v)=\partial_\mu^p V (t_k, \mu; v)+\int_{t_k}^t\bE [ \rL W_{\mu,v} (s, Y^{[r]}_s, Y^{[r]}_{t_k}) ] \diff s ,
\]
with $\rI_k\ni s\mapsto \bE [ \rL W_{\mu,v} (s, Y^{[r]}_s, Y^{[r]}_{t_k}) ]$ continuous. Hence $\rI_k\ni t\mapsto \partial_\mu^p V (t, \mu; v)$ is continuously differentiable and such that $\partial_t\partial_\mu^p V (t, \mu; v)=\bE [ \rL W_{\mu,v} (t, Y^{[r]}_t, Y^{[r]}_{t_k}) ]$.

When $m\le \lfloor \frac{q-p}2\rfloor$, we repeat the above procedure by denoting by $\rL^m$ the $m$-times composition of $\rL$. We get that $(\mu,v,t,y,x)\mapsto\rL^m W_{\mu,v} (t,y,x)$ is of class $\sM_b^{q-p-2m}$ and $\rI_k\ni t\mapsto \partial_\mu^p V (t, \mu; v)$ is $m$-times continuously differentiable with $\partial^m_t\partial_\mu^p V (t, \mu; v)=\bE [ \rL^m W_{\mu,v} (t, Y^{[r]}_t, Y^{[r]}_{t_k}) ]$. The required smoothness of $\partial_t^m \partial_\mu^p V (t, \mu; v) $ in $v$ follows from Leibniz integral rule. This completes the proof.
\end{proof}

\begin{proof}[Proof of \zcref{main-thm:proof:repre}] \label{main-thm:proof:repre:prf}
\begin{enumerate}
\item By \zcref{cal:master-PDE}(3) and \zcref{cal:time-diff2}, $\rI_k \times \sP_2(\bR^d) \ni (t,\mu) \mapsto\rL^m_{t_k} \tilde \cV (s,t,\mu)$ is smooth and $\rL_{t_k} \tilde \cV (s, t_k, \mu^n_{t_k}) = 0$.

\item We are going to check by induction on $n\in\bN$ the following smoothness-representation properties that will be called \MakeLinkTarget*{main-thm:proof:sm-rp}\textbf{(P)}:  there exists a finite collection $F_{m+1}$ of pairs $(\ell, \upsilon)$ with
\begin{itemize}
\item the functions $\ell \colon \bT \times (\bR^d)^{m+1} \times \sP_2 (\bR^d)  \to \bR$ and $\upsilon \colon \Delta^{q+1} \times (\bR^d)^{m+1} \times \sP_2 (\bR^d)  \to \bR$ depend on $(m, b, \sigma, \tilde \cU)$ but not on $n$ and are of class $\sM_b^\infty$,

\item for every $k\in \bN$  with $k \le n-1$ and $(s, t, \mu) \in \Delta^q_{t_{k+1}} \times \rI_k \times \sP_2(\bR^d)$,
\begin{equation} \label{main-thm:proof:repre:eq0:a}
\rL_{t_k}^{m+1} \tilde \cV (s, t, \mu) = \sum_{(\ell, \upsilon) \in F_{m+1}} \bE \big [ \ell (t_k, Y^{[m+1]}_{t_k}, \mu^n_{t_k}) \upsilon (s, t, Y^{[m+1]}_t, \mu ) \big ] .
\end{equation}
\end{itemize}

To deduce the desired representation of  $\rL_{t_k}^{m+1} \tilde \cV  (s, t_k, \mu^n_{t_k})$, we remark that when $t=t_k$ then $Y^{[m+1]}_t=Y^{[m+1]}_{t_k}=(Y^{(1)}_{t_k},\ldots,Y^{(m+1)}_{t_k})$ with coordinates i.i.d according to $\mu^n_{t_k}$. This ensures
\[
\frac 1{m!}\rL_{t_k}^{m+1} \tilde \cV  (s, t_k, \mu^n_{t_k}) = \varphi_m (s, t_k, \mu^n_{t_k}) ,
\]
where $\varphi_m$ is defined as
\[
\varphi_{m} (s, t, \mu) \coloneq\frac 1{m!} \bE \sum_{(\ell, \upsilon) \in F_{m+1}} \big [ \ell (t, Z_1, \ldots, Z_{m+1}, \mu ) \upsilon (s, t, Z_1, \ldots, Z_{m+1}, \mu ) \big ] ,
\]
with $Z_1, \ldots, Z_{m+1}$ i.i.d according to $\mu$. By \zcref{cal:exp}(1), $\ell \upsilon$ is of class $\sM_b^\infty$. By \zcref{cal:exp}(3), $\varphi_{m}$ is of class $\sM_b^\infty$.

Let us now prove \hyperlink{main-thm:proof:sm-rp}{\textbf{(P)}}. By \zcref{main-thm:proof:diff-oper}, for $(s,t,\mu)\in\Delta^q_{t_{k+1}}\times\rI_k\times\sP_2(\bR^d)$,
\begin{equation}
\rL_{t_k} \tilde \cV (s, t, \mu) = \begin{myaligned}[t]
& \partial_t \tilde \cV (s, t, \mu) + \bE \Big [ \sum_{{i}=1}^d [b]_{i} (t_k, Y^{(1)}_{t_k}, \mu^n_{t_k}) [\partial_\mu \tilde \cV]_{i} (s, t, \mu; Y^{(1)}_t) \\
& + \frac{1}{2} \sum_{{i},{j}=1}^d [a]_{({i}, {j})} (t_k, Y^{(1)}_{t_k}, \mu^n_{t_k}) [\nabla_{v_1} \partial_\mu \tilde \cV]_{({i}, {j})} (s, t, \mu; Y^{(1)}_t) \Big ].
\end{myaligned}
\end{equation}

Then \hyperlink{main-thm:proof:sm-rp}{\textbf{(P)}} holds for $m=0$ with
\begin{align}
F_1 = \{ &( 1  , \partial_t \tilde \cV (s, t, \mu)) ; ( [b]_{i} (t, y, \mu)  ,  [\partial_\mu \tilde \cV]_{i} (s, t, \mu; y_1) ) ;
 \\
&( [a]_{({i}, {j})} (t, y, \mu)  ,  [ \nabla_{v_1} \partial_\mu \tilde \cV]_{(i, {j})} (s, t, \mu; y_1) ) \}_{1 \le i, j \le d} .
\end{align}

Next we proceed by induction. Assume that \hyperlink{main-thm:proof:sm-rp}{\textbf{(P)}} holds for $m$. For $(\ell, \upsilon) \in F_{m+1}$, we define $V_{\ell, \upsilon} \colon \Delta^{q}_{t_{k+1}} \times \rI_k \times \sP_2 (\bR^d)  \to \bR$ by
\[
V_{\ell, \upsilon} (s, t, \mu) \coloneq \bE \big [ \ell (t_k, Y^{[m+1]}_{t_k}, \mu^n_{t_k}) \upsilon (s, t, Y^{[m+1]}_t, \mu ) \big ] .
\]

We denote $\ell (t, y_1, \ldots, y_{m+1}, \mu)$ and $\upsilon (s, t, y_1, \ldots, y_{m+1}, \mu)$. By \zcref{cal:time-diff2},
\begin{equation}
\partial_t V_{\ell, \upsilon} (s, t, \mu) = \begin{myaligned}[t]
& \bE \Big [ \ell (t_k, Y^{[m+1]}_{t_k}, \mu^n_{t_k}) \partial_t \upsilon (s, t, Y^{[m+1]}_t, \mu ) \\
& + \sum_{i=1}^{m+1} \ell (t_k, Y^{[m+1]}_{t_k}, \mu^n_{t_k}) b_k^{(i)} \cdot \nabla_{y_i} \upsilon (s, t, Y^{[m+1]}_t, \mu ) \\
& + \frac{1}{2} \sum_{i=1}^{m+1} \ell (t_k, Y^{[m+1]}_{t_k}, \mu^n_{t_k}) a_k^{(i)} : \nabla_{y_i}^2 \upsilon (s, t, Y^{[m+1]}_t, \mu ) \Big ] .
\end{myaligned}
\end{equation}

By \zcref{cal:leibniz} and the freezing lemma,
\[
\begin{myaligned}[t]
& \bE [ b_k \cdot \partial_\mu V_{\ell, \upsilon} (s, t, \mu; X^n_t) ] \\
& = \bE \Big [ \ell (t_k, Y^{[m+1]}_{t_k}, \mu^n_{t_k})   b_k^{(m+2)} \cdot \partial_\mu \upsilon (s, t, Y^{[m+1]}_t, \mu; Y^{(m+2)}_t)  \Big ] ,
\end{myaligned}
\]
and
\[
\begin{myaligned}[t]
& \bE \Big [ \frac{1}{2} a_k : \nabla_v \partial_\mu V_{\ell, \upsilon} (s, t, \mu; X^n_t) \Big ] \\
& = \bE \Big [ \frac{1}{2} \ell (t_k, Y^{[m+1]}_{t_k}, \mu^n_{t_k})   a_k^{(m+2)} : \nabla_v \partial_\mu \upsilon (s, t, Y^{[m+1]}_t, \mu ; Y^{(m+2)}_t)  \Big ] .
\end{myaligned}
\]

Recall that $b_k^{(i)} = b(t_k, Y^{(i)}_{t_k}, \mu^n_{t_k})$ and $a_k^{(i)} = a(t_k, Y^{(i)}_{t_k} , \mu^n_{t_k})$. Then
\[
\rL_{t_k} V_{\ell, \upsilon} (s, t, \mu) = \sum_{(\tilde \ell, \tilde \upsilon) \in F_{m+2}^{(\ell, \upsilon)}} \bE [ \tilde \ell (t_k, Y^{[m+2]}_{t_k}, \mu^n_{t_k}) \tilde  \upsilon (s, t, Y^{[m+2]}_t, \mu ) ] ,
\]
where
\begin{align}
F_{m+2}^{(\ell, \upsilon)} \coloneq \{ &( \ell (t, y_1, \ldots, y_{m+1}, \mu)  , \partial_t \upsilon (s, t, y_1, \ldots, y_{m+1}, \mu)) ; \\
&( \ell (t, y_1, \ldots, y_{m+1}, \mu) [b]_{i'} (t, y_i, \mu)  ,   [\nabla_{y_i} \upsilon]_{i'} (s, t, y_1, \ldots, y_{m+1}, \mu ) ) ; \\
&( \frac{1}{2} \ell (t, y_1, \ldots, y_{m+1}, \mu) [a]_{(i', j')} (t, y_i, \mu)  ,   [\nabla_{y_i}^2 \upsilon]_{(i', j')} (s, t, y_1, \ldots, y_{m+1}, \mu ) ) ; \\
&( \ell (t, y_1, \ldots, y_{m+1}, \mu) [b]_{i'} (t, y_{m+2}, \mu)  ,   [\partial_\mu \upsilon]_{i'} (s, t, y_1, \ldots, y_{m+1}, \mu;  y_{m+2}) ) ; \\
&( \frac{1}{2} \ell (t, y_1, \ldots, y_{m+1}, \mu) [a]_{(i', j')} (t, y_{m+2}, \mu)  ,   [\nabla_{y_{m+2}} \partial_\mu \upsilon]_{(i', j')} (s, t, y_1, \ldots, y_{m+1}, \mu; y_{m+2} ) )\\& \}_{1 \le i', j' \le d ; 1 \le i \le m+1} .
\end{align}

Then
\[
\rL_{t_k}^{m+2} \tilde \cV (s, t, \mu) = 
\sum_{(\ell, \upsilon) \in F_{m+1}} \rL_{t_k} V_{\ell, \upsilon} (s, t, \mu) .
\]

Thus \hyperlink{main-thm:proof:sm-rp}{\textbf{(P)}} holds for $m+1$ with
\[
F_{m+2} = \bigcup_{(\ell, \upsilon) \in F_{m+1}} F_{m+2}^{(\ell, \upsilon)} .
\]

\item Let $q\ge 1$ and $\tilde s\in\Delta^{q-1}_{t_K}$. We now prove by induction that the following rewriting of \zcref{main-thm:proof:repre:eq1} holds for $m\in\bN$
\begin{myalign} 
& \tilde \cU (\tilde s,t_K, \mu^{n}_{t_K}) - \tilde \cU (\tilde s,t_K, \mu_{t_K}) \\
& = \begin{myaligned}[t]
& \sum_{i=1}^{m} \frac{1}{i!}\int_{0}^{t_K} (\bar \tau^n_t - t)^i \rL_{\tau^n_t}^{i+1} \tilde \cV  (\tilde s,t_K, \tau^n_t, \mu^n_{\tau^n_t}) \diff t \\
& + \frac{1}{(m+1)!} \int_{0}^{t_K} (\bar \tau^n_t - t)^{m+1} \rL_{\tau^n_t}^{m+2} \tilde \cV  (\tilde s,t_K, t, \mu^n_t) \diff t,
\end{myaligned} \label{main-thm:proof:repre:eq1:a}
\end{myalign}
with the first term on the right-hand side equal to $0$ when $m=0$. The proof of \zcref{main-thm:proof:repre:eq1bis} for $q=0$ is the same.

To see that \zcref{main-thm:proof:repre:eq1:a} implies the same formula with $m$ replaced by $m+1$, we remark that for $m\in\{-1\}\cup\bN$ and $s\in\Delta^q_{t_K}$, by applying \zcref{cal:Ito-lemma} with $U(t, \mu)=\rL_{t_k}^{m+2} \tilde \cV (s, t, \mu)$ and then Fubini's theorem,
\begin{equation}
\begin{myaligned}
&  \int_{0}^{t_K} (\bar \tau^n_t - t)^{m+1} \rL_{\tau^n_t}^{m+2} \tilde \cV  (s, t, \mu^n_t) \diff t-\int_{0}^{t_K} (\bar \tau^n_t - t)^{m+1} \rL_{\tau^n_t}^{m+2} \tilde \cV  (s, \tau^n_t, \mu^n_{\tau^n_t}) \diff t\\
&=\sum_{k=0}^{K-1} \int_{\rI_k} (t_{k+1} - t)^{m+1} \{ \rL_{t_k}^{m+2} \tilde \cV  (s, t, \mu^n_t) - \rL_{t_k}^{m+2} \tilde \cV  (s, t_k, \mu^n_{t_k}) \} \diff t\\
&=\sum_{k=0}^{K-1} \int_{\rI_k} (t_{k+1} - t)^{m+1} \int_{t_k}^t \rL_{t_k}^{m+3} \tilde \cV  (s, r, \mu^n_r) \diff r \diff t\\
&= \frac 1{m+2}\sum_{k=0}^{K-1} \int_{\rI_k} (t_{k+1} - r)^{m+2} \rL_{t_k}^{m+3} \tilde \cV  (s, r, \mu^n_r) \diff r\\
&=\frac 1{m+2}\int_{0}^{t_K} (\bar \tau^n_t - t)^{m+2} \rL_{\tau^n_t}^{m+3} \tilde \cV  (s, t, \mu^n_t) \diff t.\label{initrec}
\end{myaligned}
\end{equation}

Recall that $\bar \tau^n_t = t_{k+1}$ if $t \in [t_k, t_{k+1})$ for some $k \in \llbracket 0, n-1  \rrbracket$. To check that \zcref{main-thm:proof:repre:eq1:a} holds for $m=0$, we first remark that  by the definition of $\tilde \cV$, then the application of \zcref{cal:Ito-lemma} with $U(t, \mu)=\tilde \cV (\tilde s,t_K, t, \mu)$ and the equality $\rL_{\tau^n_t} \tilde \cV (s,\tau^n_t, \mu^n_{\tau^n_t})=0$, 
\begin{myalign}\label{order1}
& \tilde \cU (\tilde s,t_K, \mu^{n}_{t_K}) - \tilde \cU (\tilde s,t_K,\mu_{t_K}) \\
& = \tilde \cU (\tilde s,t_K, \lbrbrak X_{t_K}^{t_K, \mu^{n}_{t_K}} \rbrbrak) - \tilde \cU (\tilde s,t_K, \lbrbrak X_{t_K}^{0, \nu} \rbrbrak)\\&=\tilde \cV (\tilde s, t_K,t_K, \mu^n_{t_K}) - \tilde \cV (\tilde s,t_K, 0, \nu) \\&= \sum_{k=0}^{K-1} \bigl \{ \tilde \cV (\tilde s,t_K, t_{k+1}, \mu^n_{t_{k+1}}) - \tilde \cV (\tilde s,t_K, t_{k}, \mu^n_{t_{k}}) \bigr \} \\
& = \sum_{k=0}^{K-1} \int_{\rI_k} \rL_{t_k} \tilde \cV (\tilde s,t_K, t, \mu^n_t) \diff t\\ & = \int_0^{t_K}\rL_{\tau^n_t} \tilde \cV (\tilde s,t_K, t, \mu^n_t)\diff t-\int_0^{t_K}\rL_{\tau^n_t} \tilde \cV (\tilde s,t_K,\tau^n_t, \mu^n_{\tau^n_t})\diff t.
\end{myalign}

With \zcref{initrec} for $m=-1$, we deduce \zcref{main-thm:proof:repre:eq1:a} for $m=0$.  This completes the proof.
\end{enumerate}
\end{proof}

\begin{remark} \label{remark1}
We can get a more explicit form of $C_1$ in \zcref{main-thm:cnst} as follows. By \zcref{main-thm:proof:int5} written with $m=1$, $C_1=T\phi_{(1;0)}$ while by \zcref{defphiik} and since $\beta_0=1$ according to \zcref{main-thm:proof:int}, $\phi_{(1;0)}=\frac{1}{2} \int_0^T \varphi_{(1)} (t,\mu_t) \diff t$. According to \zcref{main-thm:proof:repre}(2) and its \hyperref[main-thm:proof:repre:prf]{proof}, to calculate $\varphi_{(1)}(t,\mu_t)$, it is enough to compute $\rL^2_{\tau^n_t}\cV(t,\mu)$ with $\cV(t,\mu)=\cU(\lbrbrak X_{T}^{t, \mu} \rbrbrak)$ and replace in \zcref{main-thm:proof:repre:eq0:a} both $(Y^{(1)}_{\tau^n_t},Y^{(2)}_{\tau^n_t})$ and $(Y^{(1)}_{t},Y^{(2)}_{t})$ by $(Z_1,Z_2)$ with coordinates i.i.d according to $\mu$. For $k \in \llbracket 0, n-1 \rrbracket$ and $t\in \rI_k$, we have
\begin{align}
\rL_{t_k}\cV(t,\mu)&= \partial_t \cV (t, \mu) + \bE \Big [ b_k \cdot \partial_\mu \cV (t, \mu ; X^n_t) + \frac{1}{2} a_k : \nabla_{v} \partial_\mu \cV (t, \mu ; X^n_t) \Big ], \\ 
\rL^2_{t_k}\cV(t, \mu) &= \partial_t \rL_{t_k}\cV(t, \mu) + \bE \Big [ b_k \cdot \partial_\mu \rL_{t_k}\cV (t, \mu ; X^n_t) + \frac{1}{2} a_k : \nabla_{v} \partial_\mu \rL_{t_k}\cV (t, \mu ; X^n_t) \Big ].
\end{align}

To make $\rL^2_{t_k}\cV(t, \mu)$ explicit, we recall some tensor operations. For $x, y \in (\bR^d)^{\otimes m}$, their inner product is defined as $x \odot_m y \coloneq \sum_{i \in \llbracket 1, d \rrbracket^m} [x]_i [y]_i$. For $x \in (\bR^d)^{\otimes m}$ and $y \in (\bR^d)^{\otimes n}$, their tensor product $x \otimes y \in (\bR^d)^{\otimes (m+n)}$ is defined as $[x \otimes y]_{(i, j)} \coloneq [x]_i [y]_j$ for $i \in \llbracket 1, d \rrbracket^m$ and $j \in \llbracket 1, d \rrbracket^n$.

By \zcref{cal:time-diff2} and Leibniz integral rule,
\[
\partial_t  \rL_{t_k}\cV(t, \mu) = \begin{myaligned}[t]
& \partial_t^2 \cV (t, \mu) + \bE \Big [ b_k^{(1)} \cdot \partial_t \partial_\mu \cV (t, \mu; Y^{(1)}_t) + b_k^{(1)} \otimes b_k^{(1)} : \nabla_{v_1} \partial_\mu \cV (t, \mu; Y^{(1)}_t) \\
& + \frac{1}{2} b_k^{(1)} \otimes a_k^{(1)} \odot_3 \nabla_{v_1}^2 \partial_\mu \cV (t, \mu; Y^{(1)}_t) + \frac{1}{2} a_k^{(1)} : \partial_t \nabla_{v_1} \partial_\mu \cV (t, \mu; Y^{(1)}_t) \\
& + \frac{1}{2} a_k^{(1)} \otimes b_k^{(1)} \odot_3 \nabla_{v_1}^2 \partial_\mu \cV (t, \mu; Y^{(1)}_t) + \frac{1}{4} a_k^{(1)} \otimes a_k^{(1)} \odot_4 \nabla_{v_1}^3 \partial_\mu \cV (t, \mu; Y^{(1)}_t) \Big ] .
\end{myaligned}
\]

By \zcref{cal:leibniz},
\begin{myalign}
&\bE [ b_k \cdot \partial_\mu \rL_{t_k}\cV(t, \mu; X^n_t) ] \\
& = \begin{myaligned}[t]
& \bE \Big [ b_k^{(1)} \cdot \partial_\mu \partial_t \cV (t, \mu; Y^{(1)}_t) + b_k^{(1)} \otimes b_k^{(2)} : \partial_\mu^2 \cV (t, \mu; Y^{[2]}_t) \\
& + \frac{1}{2} a_k^{(1)} \otimes b_k^{(2)} \odot_3 \partial_\mu \nabla_{v_1} \partial_\mu \cV (t, \mu; Y^{[2]}_t) \Big ] ,
\end{myaligned} 
\end{myalign}
and
\begin{myalign}
& \bE \Big [ \frac{1}{2} a_k : \nabla_v \partial_\mu \rL_{t_k}\cV(t, \mu; X^n_t) \Big ] \\
& = \begin{myaligned}[t]
& \bE \Big [ \frac{1}{2} a_k^{(1)} : \nabla_{v_1} \partial_\mu \partial_t \cV (t, \mu; Y^{(1)}_t) + \frac{1}{2} b_k^{(1)} \otimes a_k^{(2)} \odot_3 \nabla_{v_2} \partial_\mu^2 \cV (t, \mu; Y^{[2]}_t) \\
& + \frac{1}{4} a_k^{(1)} \otimes a_k^{(2)} \odot_4 \nabla_{v_2} \partial_\mu \nabla_{v_1} \partial_\mu \cV (t, \mu; Y^{[2]}_t) \Big ] .
\end{myaligned}
\end{myalign}

Recall that $Y^{[2]}_t = (Y^{(1)}_t, Y^{(2)}_t)$ is an i.i.d sample of $X^n_t$ and that $b_k^{(i)} = b(t_k, Y^{(i)}_{t_k}, \mu^n_{t_k})$ and $a_k^{(i)} = a(t_k, Y^{(i)}_{t_k} , \mu^n_{t_k})$. By Schwarz's theorem and \zcref{cal:Clairaut-lm}[(1) and (2)],
\begin{align}
\partial_t \nabla_{v_1} \partial_\mu \cV (t, \mu; Y^{(1)}_t) & = \nabla_{v_1} \partial_t \partial_\mu \cV (t, \mu; Y^{(1)}_t) , \\
\partial_\mu \partial_t \cV (t, \mu; Y^{(1)}_t) & = \partial_t \partial_\mu \cV (t, \mu; Y^{(1)}_t) , \\
\nabla_{v_1} \partial_\mu \partial_t \cV (t, \mu; Y^{(1)}_t) & = \nabla_{v_1} \partial_t \partial_\mu \cV (t, \mu; Y^{(1)}_t) , \\
b_k^{(1)} \otimes a_k^{(1)} \odot_3 \nabla_{v_1}^2 \partial_\mu \cV (t, \mu; Y^{(1)}_t) & = a_k^{(1)} \otimes b_k^{(1)} \odot_3 \nabla_{v_1}^2 \partial_\mu \cV (t, \mu; Y^{(1)}_t) , \\
a_k^{(1)} \otimes b_k^{(2)} \odot_3 \partial_\mu \nabla_{v_1} \partial_\mu \cV (t, \mu; Y^{[2]}_t) & = b_k^{(1)} \otimes a_k^{(2)} \odot_3 \nabla_{v_2} \partial_\mu^2 \cV (t, \mu; Y^{[2]}_t) .
\end{align}

Thus
\begin{equation} \label{main-thm:cnst2}
\varphi_{(1)} (t, \mu) = \begin{myaligned}[t]
& \partial_t^2 \cV (t, \mu) + \bE \Big [ 2 b (t, Z_1, \mu) \cdot \partial_t \partial_\mu \cV (t, \mu; Z_1) \\
& + a (t, Z_1, \mu) : \nabla_{v_1} \partial_t \partial_\mu \cV (t, \mu; Z_1) \\
& + b (t, Z_1, \mu) \otimes b (t, Z_2, \mu) : \partial_\mu^2 \cV (t, \mu; Z_1, Z_2) \\
& + a (t, Z_1, \mu)\otimes b (t, Z_1, \mu)  \odot_3 \nabla_{v_1}^2 \partial_\mu \cV (t, \mu; Z_1) \\
& + \frac{1}{4} a (t, Z_1, \mu) \otimes a (t, Z_1, \mu) \odot_4 \nabla_{v_1}^3 \partial_\mu \cV (t, \mu; Z_1) \\
& + b (t, Z_1, \mu) \otimes b (t, Z_1, \mu) : \nabla_{v_1} \partial_\mu \cV (t, \mu; Z_1) \\
& + a (t, Z_1, \mu) \otimes b (t, Z_2, \mu) \odot_3 \partial_\mu \nabla_{v_1} \partial_\mu \cV (t, \mu; Z_1, Z_2) \\
& + \frac{1}{4} a (t, Z_1, \mu) \otimes a (t, Z_2, \mu) \odot_4 \nabla_{v_2} \partial_\mu \nabla_{v_1} \partial_\mu \cV (t, \mu; Z_1, Z_2) \Big ].
\end{myaligned} 
\end{equation}

Next we are concerned with a sufficient condition to ensure $|\cU (\mu^{n}_T) - \cU (\mu_T)| \lesssim \frac{1}{n}$. We have $|\bar \tau^n_t - t| \le \frac{T}{n}$. By \zcref{main-thm:proof:repre:eq1:a} with $m=0$,
\[
\cU (\mu^{n}_T) - \cU (\mu_T) =  \int_{0}^{T} (\bar \tau^n_t - t) \rL_{\tau^n_t}^{2} \cV  (t, \mu^n_t) \diff t ,
\]
so it is enough that $\rL_{\tau^n_t}^{2} \cV$ is continuous and bounded uniformly in $n$. By \zcref{main-thm:proof:diff-oper}, for $\rL_{\tau^n_t}^{2} \cV = \rL_{\tau^n_t} \circ \rL_{\tau^n_t} \cV$ to be well-defined, it is enough that $(b, \sigma)$ are of class $\sM_b^1$ and $\rL_{\tau^n_t} \cV$ is of class $\sM_b^2$. By \zcref{main-thm:proof:diff-oper} again, for $\rL_{\tau^n_t} \cV$ to be well-defined, it is enough that $(b, \sigma)$ are of class $\sM_b^1$ and $\cV$ is of class $\sM_b^4$. By regularity of the master PDE in \zcref{cal:master-PDE}(2), it suffices that $(b, \sigma)$ are of class $\sM_b^4$.

Finally, we are concerned with a sufficient condition to ensure $\cU (\mu^{n}_T) - \cU (\mu_T) = \frac{C_1}{n} + \cO ( \frac{1}{n^{2}} )$. Such a condition is that \zcref{main-thm:proof:int6} holds with $m=1$. This is equivalent to \zcref{main-thm:proof:repre:eq1:a} holding with $m=1$, i.e.,
\[
\cU (\mu^{n}_T) - \cU (\mu_T) =  \begin{myaligned}[t]
& \int_{0}^{T} (\bar \tau^n_t - t) \rL_{\tau^n_t}^{2} \cV  (\tau^n_t, \mu^n_{\tau^n_t}) \diff t \\
& + \frac{1}{2} \int_{0}^{T} (\bar \tau^n_t - t)^{2} \rL_{\tau^n_t}^{3} \cV  (t, \mu^n_t) \diff t ,
\end{myaligned} 
\]
so it is enough that $\rL_{\tau^n_t}^{3} \cV$ is continuous and bounded uniformly in $n$. With the same reasoning as in the above paragraph, it suffices that $(b, \sigma)$ are of class $\sM_b^6$.
\end{remark}

\subsection{Proof of \zcref{main-thm:proof:int}}

We need an auxiliary result.

\begin{lemma} \label{prd-fnc}
Let $h_n \colon \bT \to \bR$ be bounded and periodic with period $\eps_n$. Let $\overline{h_n} \coloneq \frac{1}{\eps_n} \int_0^{\eps_n} h_n (t) \diff t$ be the mean of $h_n$. We define the operator $\rL$ applied on $h_n$ by $\rL h_n (t) \coloneq \int_0^t \{ \overline{h_n} - h_n (s) \} \diff s$. Then
\begin{enumerate}
\item $\rL h_n \colon  \bT \to \bR$ is periodic with period $\eps_n$ and absolutely continuous with weak derivative equal to $\overline{h_n} - h_n$. Moreover, $\rL h_n (0) = \rL h_n (\tau^n_t)=0$ and $\| \rL h_n \|_\infty \le \eps_n \| h_n \|_\infty$. 

\item If $f \colon \bT \to \bR$ is absolutely continuous with weak derivative $f'$, then it holds for $t \in \bT $ that
\[
\int_0^{\tau^n_t} h_n (s) f (s) \diff s = \overline{h_n} \int_0^{\tau^n_t} f (s) \diff s + \int_0^{\tau^n_t} \rL h_n (s) f'(s) \diff s .
\]

\item If $h_n (t) = \bar \tau^n_t - t$ for $t \in \bT$, then there exist constants $\{\alpha_i\}_{i \in \bN}$ with $\alpha_0 = \frac{1}{2}$ that do not depend on $n$ and such that
\[
\overline{\rL^i h_n} \coloneq \overline{\underbrace{\rL \circ \cdots \circ \rL}_{i \text{ times}} h_n} = \alpha_i \eps_n^{i+1} 
\qtextq{for} i \in \bN .
\]
\end{enumerate}
\end{lemma}

\begin{proof}
\begin{enumerate}
\item By periodicity of $h_n$,
\begin{align}
\rL h_n (t+\eps_n) & = \rL h_n (t) + \int_{t}^{t+\eps_n} \{ \overline{h_n} - h_n (s) \} \diff s \\
& = \rL h_n (t) + \int_{0}^{\eps_n} \{ \overline{h_n} - h_n (s) \} \diff s = \rL h_n (t).
\end{align}
    
The periodicity of $\rL h_n$ follows. Because $\frac{\tau^n_t}{\eps_n} \in \bN$, we get $\rL h_n (\tau^n_t)=0$ and $\rL h_n (t) = \int_{\tau^n_t}^t \{ \overline{h_n} - h_n (s) \} \diff s$. Notice that $|t-\tau^n_t| \le \eps_n$. So $\| \rL h_n \|_\infty \le \eps_n \| h_n \|_\infty$.
    
\item The second claim follows from the integration by parts formula.
    
\item Let $h_n (t) = \bar \tau^n_t - t$ for $t \in \bT$. We prove by induction on $i$ that $\rL^i h_n (t) = \sum_{j=0}^{i+1} C_{i, j} \eps_n^{i-j+1}t^j$ for $t \in [0, \eps_n)$ and that $\overline{\rL^i h_n}= \alpha_i \eps_n^{i+1}$. For $i=0$, we have $h_n (t) = -t + \eps_n$ for $t \in [0, \eps_n)$ and $\overline{h_n} = \frac{\eps_n}{2}$. Then $C_{0,0} = - C_{0,1}=1$ and $\alpha_0= \frac{1}{2}$. Assume that the claim holds for $i$. Then
\begin{align}
\rL^{i+1} h_n (t) & = \int_0^t \Big \{ \alpha_i \eps_n^{i+1} - \sum_{j=0}^{i+1} C_{i, j} \eps_n^{i-j+1} s^j \Big \} \diff s \\
& = \alpha_i \eps_n^{i+1} t - \sum_{j=0}^{i+1} \frac{C_{i, j} \eps_n^{i-j+1}}{j+1} t^{j+1} \qtextq{for} t \in [0, \eps_n), \\
\overline{\rL^{i+1} h_n} & = \frac{1}{\eps_n} \int_0^{\eps_n}  \rL^{i+1} h_n (t) \diff t = \frac{\alpha_i \eps_n^{i+2}}{2} - \sum_{j=0}^{i+1} \frac{C_{i, j} \eps_n^{i+2}}{(j+1)(j+2)} .
\end{align}
    
Thus the claim holds for $i+1$ with
\begin{align}
C_{i+1, j} & = 
\begin{cases}
0 & \textq{if} j=0, \\
\alpha_i - C_{i, 0} & \textq{if} j=1, \\
\frac{- C_{i, j-1}}{j} & \textq{if} j \in \llbracket 2, i+2 \rrbracket , \\
\end{cases} \\
\alpha_{i+1} & = \frac{\alpha_i}{2} - \sum_{j=0}^{i+1} \frac{C_{i, j}}{(j+1)(j+2)} .
\end{align}

This completes the proof.
\end{enumerate}
\end{proof}

\begin{proof}[Proof of \zcref{main-thm:proof:int}] \label{main-thm:proof:int:proof}
\textit{Step 1:} We verify that if $(K, i) \in \llbracket 1, n  \rrbracket \times \bN$ and $g \colon \bT \to \bR$ is absolutely continuous, then
\begin{equation} \label{main-thm:proof:int:eq2}
\int_0^{t_K} (\bar \tau^n_t - t)^i g (\tau^n_t) \diff t = \frac{\eps_n^i}{i+1} \Big \{ \int_0^{t_K} g (t) \diff t - \int_0^{t_K} (\bar \tau^n_t - t) g'(t) \diff t \Big \} .
\end{equation}

We denote by $J$ the left-hand side of \zcref{main-thm:proof:int:eq2}. Then
\begin{align}
J & = \sum_{k=0}^{K-1} \int_{\rI_k} (t_{k+1}-t)^i g (t_k) \diff t \\
& = \frac{\eps_n^i}{i+1} \sum_{k=0}^{K-1} g (t_k) \eps_n = \frac{\eps_n^i}{i+1} \int_0^{t_K} g (\tau^n_t) \diff t \\
& = \frac{\eps_n^i}{i+1} \Big \{ \int_0^{t_K} g (t) \diff t - \int_0^{t_K} \int_{\tau^n_t}^{t} g'(s) \diff s \diff t \Big \} \qtext{by the mean value theorem} \\
& = \frac{\eps_n^i}{i+1} \Big \{ \int_0^{t_K} g (t) \diff t - \int_0^{t_K} (\bar \tau^n_t - t) g'(t) \diff t \Big \} \qtext{by Fubini's theorem} .
\end{align}

\textit{Step 2:} We denote by $F$ the left-hand side of \zcref{main-thm:proof:int:eq0}, i.e.,
\[
F \coloneq (i+1) \int_0^{t_K} (\bar \tau^n_t - t)^i f (\tau^n_t) \diff t .
\]

Let $h_n (t) = \bar \tau^n_t - t$ for $t \in \bT$. By \zcref{main-thm:proof:int:eq2},
\[
\frac{F}{\eps_n^i} = \int_0^{t_K} f (t) \diff t - \int_0^{t_K} (\bar \tau^n_t - t) f'(t) \diff t.
\]

With \zcref{prd-fnc}(2), we deduce by induction that for $m\in{\mathbb N}^*$,
\begin{align}
\frac{F}{\eps_n^i} 
& = \begin{myaligned}[t]
& \int_0^{t_K} f (t) \diff t - \sum_{j=1}^{m-1} \overline{\rL^{j-1} h_n} \int_0^{t_K} \partial_t^{j} f (t) \diff t \\
& - \int_0^{t_K} \rL^{m-1} h_n (t) \partial_t^{m} f(t) \diff t .
\end{myaligned}
\end{align}

By \zcref{prd-fnc}(1), $\| \rL^{m-1} h_n \|_\infty = \cO (\eps_n^m)$. By \zcref{prd-fnc}(3), there exist constants $\{\alpha_j\}_{j \in \bN} \subset \bR$ with $\alpha_0 = \frac{1}{2}$ that do not depend of $n$ and such that $\overline{\rL^j h_n} = \alpha_j \eps_n^{j+1}$. The required representation follows. More precisely,
\[
\beta_j =
\begin{dcases}
1 & \qtextq{if} j = 0 , \\
\alpha_{j-1} & \qtextq{if} j \in  \llbracket 1, m-1 \rrbracket .
\end{dcases}
\]

In particular, $\beta_1 = \frac{1}{2}$. This completes the proof.
\end{proof}

\section{Proofs of auxiliary results in \zcref{calculus_on-P2}} \label{calculus_on-P2:proof}

\begin{proof}[Proof of \zcref{cal:Clairaut-lm}] \label{cal:Clairaut-lm:prf}
\begin{enumerate}
\setcounter{enumi}{1}
\item \cite[Lemma 5.1]{buckdahn_mean-field_2017} has the same conclusion. However, they assume the existence of $\partial_\mu \nabla_x U (x, \mu)$, which is not our case here. We will modify their argument, which in turn follows that of Schwarz's theorem. For $x,y \in \bR^r$ and $X, Y \in L^2 (\cF_0; \bR^d)$, we have
\begin{align}
I & \coloneq \{ U (x+y, \lbrbrak X+Y \rbrbrak) - U (x+y, \lbrbrak X \rbrbrak) \} - \{ U (x, \lbrbrak X+Y \rbrbrak) - U (x, \lbrbrak X \rbrbrak) \} \label{sym:eq0} \\
& = \int_0^1 \bE [ \{ \partial_\mu U (x+y, \lbrbrak X+tY \rbrbrak ; X+tY) - \partial_\mu U (x, \lbrbrak X+tY \rbrbrak ; X+tY) \} \cdot Y ] \diff t \label{sym:eq1} \\
& = \int_0^1 \int_0^1 \bE [ \nabla_x \partial_\mu U (x + sy, \lbrbrak X+tY \rbrbrak ; X+tY) y  \cdot Y ] \diff s \diff t \label{sym:eq2} \\
& = y \cdot \bE [ [ \nabla_x \partial_\mu U ]^\top (x, \lbrbrak X \rbrbrak ; X) Y ] + y \cdot R_1 (y, Y) , \label{sym:eq2:a}
\end{align}
where
\[
R_1 (y, Y) \coloneq \begin{myaligned}[t]
& \int_0^1 \int_0^1 \bE [ \{ \nabla_x \partial_\mu U (x + sy, \lbrbrak X+tY \rbrbrak ; X+tY) - \nabla_x \partial_\mu U (x, \lbrbrak X \rbrbrak ; X) \}^\top Y ] \diff s \diff t .
\end{myaligned}
\]

On the other hand,
\begin{align}
I & = \{ U (x+y, \lbrbrak X+Y \rbrbrak) - U (x, \lbrbrak X+Y \rbrbrak)  \} - \{ U (x+y, \lbrbrak X \rbrbrak) - U (x, \lbrbrak X \rbrbrak) \} \\
& = \int_0^1 \{ \nabla_x U (x+sy, \lbrbrak X+Y \rbrbrak) - \nabla_x U (x+sy, \lbrbrak X \rbrbrak) \} \cdot y \diff s \label{sym:eq3} \\
& = y \cdot \{ \nabla_x U (x, \lbrbrak X+Y \rbrbrak) - \nabla_x U (x, \lbrbrak X \rbrbrak) \} + y \cdot R_2 (y, Y) \label{sym:eq3:a} ,
\end{align}
where
\[
R_2 (y, Y) \coloneq \begin{myaligned}[t]
& \int_0^1 \{ \nabla_x U (x+sy, \lbrbrak X+Y \rbrbrak) - \nabla_x U (x, \lbrbrak X+Y \rbrbrak) \} \diff s \\
& - \int_0^1 \{ \nabla_x U (x+sy, \lbrbrak X \rbrbrak) - \nabla_x U (x, \lbrbrak X \rbrbrak) \} \diff s .
\end{myaligned}
\]

Above, \zcref{sym:eq1}, \zcref{sym:eq2} and \zcref{sym:eq3} are due to the mean value theorem. By \zcref{sym:eq2:a} and \zcref{sym:eq3:a}, when $Y \neq 0$,
\begin{equation} \label{sym:eq4}
y \cdot \Big \{ \begin{myaligned}[t]
& \frac{\nabla_x U (x, \lbrbrak X+Y \rbrbrak) - \nabla_x U (x, \lbrbrak X \rbrbrak) - \bE [ [ \nabla_x \partial_\mu U ]^\top (x, \lbrbrak X \rbrbrak ; X) Y ]}{\| Y \|_{L^2}} \\
& + \frac{R_2 (y, Y) - R_1 (y, Y)}{\| Y \|_{L^2}} \Big \} = 0.
\end{myaligned}
\end{equation}

Let $C$ be the Lipschitz constant of $ \nabla_x \partial_\mu U$ and $\nabla_x U$. Then $|R_2 (y, Y)| \le 2C |y|$. Substituting $y/m$ for $y$ in \zcref{sym:eq4} multiplied by $m$ and taking the limit $m \to \infty$, we obtain
\begin{equation} \label{sym:eq5}
\begin{multlined}[t]
y \cdot \Big \{ \frac{\nabla_x U (x, \lbrbrak X+Y \rbrbrak) - \nabla_x U (x, \lbrbrak X \rbrbrak) - \bE [ [ \nabla_x \partial_\mu U ]^\top (x, \lbrbrak X \rbrbrak ; X) Y ]}{\| Y \|_{L^2}} - \frac{R_1 (0, Y)}{\| Y \|_{L^2}} \Big \} = 0.
\end{multlined}
\end{equation}

Because \zcref{sym:eq5} holds for arbitrary $y \in \bR^r$, we get
\begin{equation} \label{sym:eq6}
\frac{\nabla_x U (x, \lbrbrak X+Y \rbrbrak) - \nabla_x U (x, \lbrbrak X \rbrbrak) - \bE [ [ \nabla_x \partial_\mu U ]^\top (x, \lbrbrak X \rbrbrak ; X) Y ]}{\| Y \|_{L^2}} = \frac{R_1 (0, Y)}{\| Y \|_{L^2}} .
\end{equation}

We have $|R_1 (0, Y)| \le 2 C \| Y \|_{L^2} \| Y \|_{L^2}$. Thus
\begin{equation}
\lim_{\| Y \|_{L^2} \to 0} \frac{\nabla_x U (x, \lbrbrak X+Y \rbrbrak) - \nabla_x U (x, \lbrbrak X \rbrbrak) - \bE [ [ \nabla_x \partial_\mu U ]^\top (x, \lbrbrak X \rbrbrak ; X) Y ]}{\| Y \|_{L^2}} = 0 .
\end{equation}

Hence $\mu \mapsto \nabla_x U (x, \mu)$ is L-differentiable with $\partial_\mu \nabla_x U (x, \mu; v) = [ \nabla_x \partial_\mu U ]^\top (x, \mu ; v)$.

\item \textit{Step 1: we will prove that $\partial_\mu U (\cdot, \mu; v)$ is differentiable and $\partial_t \partial_\mu U (t, \mu; v) = \partial_\mu \partial_t U (t, \mu; v)$.} This step is a slight modification of the above. Replacing $I$ in \zcref{sym:eq0} with
\[
I \coloneq \{ U (t+s, \lbrbrak X+Y \rbrbrak) - U (t, \lbrbrak X+Y \rbrbrak)  \} - \{ U (t+s, \lbrbrak X \rbrbrak) - U (t, \lbrbrak X \rbrbrak) \}
\]
for $t, s \in \bT$ and $X, Y \in L^2 (\cF_0; \bR^d)$. We check analogously to \zcref{sym:eq4} that
\begin{equation} \label{sym:eq7:b}
\bE \Big [ Y \cdot \Big \{ \begin{myaligned}[t]
& \frac{\partial_\mu U (t+s, \lbrbrak X \rbrbrak ; X) - \partial_\mu U (t, \lbrbrak X \rbrbrak ; X) - s \partial_\mu \partial_t U (t, \lbrbrak X \rbrbrak; X) }{s} \\
& + \frac{R_2 (s, X + rY, X) - R_1 (s, \lbrbrak X + r Y \rbrbrak, X + r Y, \lbrbrak X \rbrbrak, X)}{s} \Big \} \Big ] = 0 ,
\end{myaligned}
\end{equation}
where
\begin{align}
R_1 (s, \mu_1, Z_1, \mu_2, Z_2) & \coloneq \begin{myaligned}[t]
& s \int_0^1 \int_0^1 \{ \partial_\mu \partial_t U (t+hs, \mu_1; Z_1) - \partial_\mu \partial_t U (t, \mu_2; Z_2) \} \diff r \diff h , \\
\end{myaligned} \\
R_2 (s, Z_1, Z_2) & \coloneq \begin{myaligned}[t]
& \int_0^1 \{ \partial_\mu U (t+s, \lbrbrak Z_1 \rbrbrak ; Z_1) - \partial_\mu U (t+s, \lbrbrak Z_2 \rbrbrak ; Z_2) \} \diff r \\
& + \int_0^1 \{ \partial_\mu U (t, \lbrbrak Z_1 \rbrbrak ; Z_1) - \partial_\mu U (t, \lbrbrak Z_2 \rbrbrak ; Z_2) \} \diff r
\end{myaligned}
\end{align}
for $\mu_1, \mu_2 \in \sP_2 (\bR^d)$ and $Z_1, Z_2 \in L^2 (\cF_0; \bR^d)$.

Notice that $L^2 (\cF_0; \bR^d) \ni Y \mapsto R_1 (s, \lbrbrak X + r Y \rbrbrak, X + r Y, \lbrbrak X \rbrbrak, X)$ is continuous; and $|R_2 (s, X+rY, X)| \le 2 C \| Y \|_{L^2}$. Substituting $Y/m$ for $Y$ in \zcref{sym:eq7:b} and taking the limit $m \to \infty$, we get
\begin{equation} \label{sym:eq8:b}
\bE \Big [ Y \cdot \Big \{ \begin{myaligned}[t]
& \frac{\partial_\mu U (t+s, \lbrbrak X \rbrbrak ; X) - \partial_\mu U (t, \lbrbrak X \rbrbrak ; X) - s \partial_\mu \partial_t U (t, \lbrbrak X \rbrbrak; X) }{s} \\
& - \frac{R_1 (s, \lbrbrak X \rbrbrak, X, \lbrbrak X \rbrbrak, X)}{s} \Big \} \Big ] = 0.
\end{myaligned}
\end{equation}

Because \zcref{sym:eq8:b} holds for arbitrary $Y \in L^2 (\cF_0; \bR^d)$, we get
\begin{equation} \label{sym:eq9:b}
\frac{\partial_\mu U (t+s, \lbrbrak X \rbrbrak ; X) - \partial_\mu U (t, \lbrbrak X \rbrbrak ; X) - s \partial_\mu \partial_t U (t, \lbrbrak X \rbrbrak ; X) }{s} = \frac{R_1 (s, \lbrbrak X \rbrbrak, X, \lbrbrak X \rbrbrak, X)}{s} .
\end{equation}

Thus
\begin{equation} \label{sym:eq10:b}
\frac{\partial_\mu U (t+s, \lbrbrak X \rbrbrak ; v) - \partial_\mu U (t, \lbrbrak X \rbrbrak ; v) - s \partial_\mu \partial_t U (t, \lbrbrak X \rbrbrak ; v) }{s} = \frac{R_1 (s, \lbrbrak X \rbrbrak, v, \lbrbrak X \rbrbrak, v)}{s}
\end{equation}
for $\lbrbrak X \rbrbrak$-a.e. $v \in \bR^d$. Let $\eps >0$ and $Z$ be a standard $d$-dimensional normal random variable, which is independent of $X$. Substituting $X + \eps Z$ for $X$ in \zcref{sym:eq10:b}, we get 
\begin{equation} \label{sym:eq11:b}
\begin{myaligned}[t]
& \frac{\partial_\mu U (t+s, \lbrbrak X + \eps Z \rbrbrak ; v) - \partial_\mu U (t, \lbrbrak X + \eps Z \rbrbrak ; v) - s \partial_\mu \partial_t U (t, \lbrbrak X + \eps Z \rbrbrak ; v) }{s} \\
& = \frac{R_1 (s, \lbrbrak X + \eps Z \rbrbrak, v, \lbrbrak X + \eps Z \rbrbrak, v)}{s}
\end{myaligned}
\end{equation}
for $\lbrbrak X + \eps Z  \rbrbrak$-a.e. $v \in \bR^d$. By continuity of $\partial_\mu U, \partial_\mu \partial_t U$ and $v \mapsto R_1 (s, \lbrbrak X + \eps Z \rbrbrak, v, \lbrbrak X + \eps Z \rbrbrak, v)$, \zcref{sym:eq11:b} holds for every $v \in \bR^d$. Taking the limit $\eps \to 0$ in \zcref{sym:eq11:b}, we obtain
\begin{equation} \label{sym:eq12:b}
\frac{\partial_\mu U (t+s, \lbrbrak X \rbrbrak ; v) - \partial_\mu U (t, \lbrbrak X \rbrbrak ; v) - s \partial_\mu \partial_t U (t, \lbrbrak X \rbrbrak ; v) }{s} = \frac{R_1 (s, \lbrbrak X \rbrbrak, v, \lbrbrak X \rbrbrak, v)}{s}
\end{equation}
for every $v \in \bR^d$. By continuity of $\partial_\mu \partial_t U$, we get
\[
\lim_{s \downarrow 0} \frac{R_1 (s, \lbrbrak X \rbrbrak, v, \lbrbrak X \rbrbrak, v)}{s} = 0 .
\]

Thus
\begin{equation} \label{sym:eq13:b}
\lim_{s \downarrow 0} \frac{\partial_\mu U (t+s, \lbrbrak X \rbrbrak ; v) - \partial_\mu U (t, \lbrbrak X \rbrbrak ; v) - s \partial_\mu \partial_t U (t, \lbrbrak X \rbrbrak ; v) }{s} = 0 .
\end{equation}

It follows that $\partial_\mu U (\cdot, \lbrbrak X \rbrbrak ; v)$ is differentiable at $t$ with
\begin{equation} \label{sym:eq13:c}
\partial_t \partial_\mu U (t, \lbrbrak X \rbrbrak ; v) = \partial_\mu \partial_t U (t, \lbrbrak X \rbrbrak ; v) .
\end{equation}

\textit{Step 2: we will prove that $\nabla_v \partial_\mu U (\cdot, \mu; v)$ is differentiable and $\partial_t \nabla_v \partial_\mu U (t, \mu; v) = \nabla_v \partial_\mu \partial_t U (t, \mu; v)$.} Recall that $\partial_t U (t, \cdot)$ is of class $\sM^2_b$, so $\partial_\mu \partial_t U (t, \mu; \cdot)$ is differentiable. By \zcref{sym:eq13:c}, $\partial_t \partial_\mu U (t, \mu ; \cdot)$ is differentiable and
\begin{equation} \label{sym:eq14}
\nabla_v \partial_t \partial_\mu U (t, \mu ; v) = \nabla_v \partial_\mu \partial_t U (t, \mu ; v) .
\end{equation}

Because $\nabla_v \partial_\mu \partial_t U$ is continuous, so is $\nabla_v \partial_t \partial_\mu U$. Recall that $U$ is of class $\sM^2_b$, so $\partial_\mu U (t, \mu ; \cdot)$ is differentiable. By Schwarz's theorem, $\nabla_v \partial_\mu U (\cdot, \mu ; v)$ is differentiable and
\begin{equation} \label{sym:eq15}
\partial_t \nabla_v \partial_\mu U (t, \mu ; v) = \nabla_v \partial_t \partial_\mu U (t, \mu ; v) .
\end{equation}

By \zcref{sym:eq14,sym:eq15}, $\partial_t \nabla_v \partial_\mu U (t, \mu ; v) = \nabla_v \partial_\mu \partial_t U (t, \mu ; v)$. This completes the proof.
\end{enumerate}
\end{proof}

\begin{proof}[Proof of \zcref{cal:leibniz}]
By Leibniz integral rule,
\begin{itemize}
\item We can assume that $q=r=0$. Thus $U \colon \Omega \times \sP_2 (\bR^d) \to \bR$ and $h=m=0$.
\item It suffices to prove that $V \colon \sP_2 (\bR^d) \to \bR$ is $k$-times L-differentiable with $\partial_\mu^k V (\mu; v) = \bE [ \partial_\mu^k U(\cdot, \mu ; v)]$.
\end{itemize}

\textit{Step 1:} First, we prove the result for $k=1$. For brevity, we denote by $\tilde L^2 \coloneq L^2 (\tilde \Omega, \tilde \cF_0, \tilde \bP ; \bR^d)$ where $(\tilde \Omega, \tilde \cF_0, \tilde \bP)$ is a pointwise copy of $(\Omega, \cF_0, \bP)$. Consider
\begin{alignat}{3}
\tilde U & \colon & \Omega \times \tilde L^2 & \to \bR , & \quad (\omega, X) & \mapsto U (\omega, \lbrbrak X \rbrbrak) ; \\
U_{\omega} & \colon & \sP_2 (\bR^d) & \to \bR , & \mu & \mapsto U (\omega, \mu) ; \\
\tilde U_{\omega} & \colon & \tilde L^2 & \to \bR , & X & \mapsto U (\omega, \lbrbrak X \rbrbrak) .
\end{alignat}

In particular, $\tilde U_{\omega}$ is the lift of $U_{\omega}$. Clearly, $V$ and its lift $\tilde V$ are bounded Lipschitz. We have
\begin{align}
\| \nabla \tilde U_\omega (X) - \nabla \tilde U_\omega (Y) \|_{\tilde L^2} & = \| \partial_\mu U_\omega (\lbrbrak X \rbrbrak, X) - \partial_\mu U_\omega (\lbrbrak Y \rbrbrak, Y) \|_{\tilde L^2} \\
& \le C_\omega \| \sW_2 (\lbrbrak X \rbrbrak, \lbrbrak Y \rbrbrak) + |X-Y| \|_{\tilde L^2} \\
& \le 2 C_\omega \| X-Y \|_{\tilde L^2} .
\end{align}

Then $\nabla \tilde U_\omega$ is bounded Lipschitz. By an unnumbered corollary of \cite[Page 32]{doi:10.1137/1.9781611971309}, $\tilde U_\omega$ is strictly differentiable in the sense of \cite[Section 2.2]{doi:10.1137/1.9781611971309}. By \cite[Proposition 2.3.6]{doi:10.1137/1.9781611971309}, $\tilde U_\omega$ is regular in the sense of \cite[Definition 2.3.4]{doi:10.1137/1.9781611971309}. By \cite[Theorem 2.7.2 and Proposition 2.2.4]{doi:10.1137/1.9781611971309}, $\tilde V$ is strictly differentiable with $\nabla \tilde V (X) = \bE [ \nabla_X \tilde U (\cdot, X)]$. Then $V$ is L-differentiable with $\partial_\mu V (\mu; v) = \bE [\partial_\mu U (\cdot, \mu ; v)]$.

\textit{Step 2:} Next we proceed by induction. The base case $k=1$ has been proved above. Assume that the claim holds for $k$. We now prove it for $k+1$. Let $(v, u) \in (\bR^d)^k \times \bR^d$. By induction hypothesis, $\partial_\mu^k V (\mu; v) = \bE [\partial_\mu^k U (\cdot, \mu ; v)]$. Notice that $\mu \mapsto \partial_\mu^k U (\omega, \mu ; v)$ is L-differentiable with bounded Lipschitz L-derivative. By part (1), $\mu \mapsto \partial_\mu^k V (\mu; v)$ is L-differentiable with $\partial_\mu^{k+1} V (\mu; v, u) = \bE [\partial_\mu^{k+1} U (\cdot, \mu ; v, u)]$. This completes the proof.
\end{proof}

\begin{proof}[Proof of \zcref{cal:exp}]
\begin{enumerate}
\setcounter{enumi}{1}
\item Clearly, $V$ is bounded Lipschitz. First, we prove the result for $k=1$. We denote $L^2 \coloneq L^2 (\cF_0; \bR^d)$. We define $\varphi \colon L^2 \times L^2 \to \bR$ by $\varphi (Y, Z) \coloneq \bE [U(Y, \lbrbrak Z \rbrbrak)]$. By the differentiation rules for product Banach spaces (see e.g. \cite[Theorems 8.4 and 8.11]{jost_postmodern_2005}) and \zcref{cal:leibniz},
\[
\nabla \varphi (Y, Z) = \big ( \nabla_v U (Y, \lbrbrak Z \rbrbrak) , \tilde \bE [\partial_\mu U (\tilde Y, \lbrbrak Z \rbrbrak; Z)] \big ) .
\]

Above, $\tilde Y$ is a pointwise copy of $Y$ defined on a pointwise copy $(\tilde \Omega, \tilde \cA, \tilde \bP, \tilde \bE)$ of $(\Omega, \cA, \bP, \bE)$. Let $\tilde V \colon L^2 \to \bR$ be the lift of $V$. Then $\tilde V (Y) = \varphi (Y, Y)$ and thus
\[
\nabla \tilde V (Y) = \tilde \bE [\partial_\mu U (\tilde Y, \lbrbrak Y \rbrbrak; Y)] + \nabla_v U (Y, \lbrbrak Y \rbrbrak) .
\]

So
\begin{equation} \label{cal:exp:eq1}
\partial_\mu V (\mu; v) = \bE [\partial_\mu U (Y, \mu; v)] + \nabla_v U (v, \mu)
\qtextq{where} \lbrbrak Y \rbrbrak = \mu.
\end{equation}

Next we proceed by induction. The base case $k=1$ has been proved above. Assume that the claim holds for $k$. We now prove it for $k+1$. Let $u \in \bR^d, v=(v_1, \ldots,v_k) \in (\bR^d)^k$ and $w \coloneq (v_1, \ldots, v_k, u) \in (\bR^d)^{k+1}$. By induction hypothesis,
\begin{equation} \label{cal:exp:eq1:a}
\partial_\mu^k V (\mu; v) = \bE[\partial_\mu^k U(Y, \mu; v)] + \sum_{i=1}^k \partial_\mu^{k-i} \nabla_{v_i} \partial_\mu^{i-1} U(v_i, \mu; v_{-i}) .
\end{equation}

Taking the Lions derivative in \zcref{cal:exp:eq1:a}, we obtain
\begin{equation} \label{cal:exp:eq2}
\partial_\mu^{k+1} V (\mu; w) = \begin{myaligned}[t]
& \partial_\mu \{ \bE[\partial_\mu^k U(Y, \cdot ; v)] \} (\mu, u) + \sum_{i=1}^k \partial_\mu^{k-i+1} \nabla_{v_i} \partial_\mu^{i-1} U(v_i, \mu; v_{-i}, u) .
\end{myaligned} 
\end{equation}

Next we use the base case to treat the first term on the right-hand side of \zcref{cal:exp:eq2}. Notice that $(y, \mu) \mapsto \partial_\mu^k U(y, \mu; v)$ is of class $\sM_b^1$. By \zcref{cal:exp:eq1},
\begin{equation} \label{cal:exp:eq3}
\partial_\mu \{ \bE[\partial_\mu^k U(Y, \cdot; v)] \} (\mu, u) = \bE[\partial_\mu^{k+1} U(Y, \mu; w)] + \nabla_{u} \partial_\mu^k U(u, \mu; v).
\end{equation}

Notice that $w_i=v_i$ and $w_{-i} = (v_{-i}, u)$ for $i \in \llbracket 1, k \rrbracket$. By \zcref{cal:exp:eq2} and \zcref{cal:exp:eq3}, the identity \zcref{cal:exp:eq0} holds for $k+1$.

\item By Leibniz integral rule, we can assume that $q=0$. Thus $U \colon (\bR^d)^{r} \times \sP_2 (\bR^d) \to \bR$ and $V \colon \sP_2 (\bR^d) \to \bR$. We define a collection $\{F_0, \ldots, F_r\}$ of functions $F_i \colon (\bR^d)^{r-i} \times \sP_2 (\bR^d) \to \bR$ recursively by $F_0 \coloneq U$ and
\[
F_{i+1} (y_1, \ldots, y_{r-i-1}, \mu) \coloneq \bE [F_i (y_1, \ldots, y_{r-i-1}, Y, \mu)] ,
\]
for all $y_1, \ldots, y_{r-i-1} \in \bR^d$ and $\lbrbrak Y \rbrbrak = \mu \in \sP_2 (\bR^d)$. Notice that $V= F_r$ by Fubini's theorem.

By \zcref{cal:exp}(2), $F_1 (y, \cdot)$ is $k$-times L-differentiable and for $p \in \llbracket 1, k \rrbracket$,
\begin{equation} \label{cal:exp2:eq1}
\partial_\mu^{p} F_1 (y, \mu; v) = \bE[\partial_\mu^{p} U(y, Y, \mu; v)] + \sum_{i=1}^{p} \partial_\mu^{p-i} \nabla_{v_{i}} \partial_\mu^{i-1} U(y, v_{i}, \mu; v_{-i})
\end{equation}
for all $y \in (\bR^d)^{r-1}$ and $\lbrbrak Y \rbrbrak = \mu \in \sP_2 (\bR^d)$.

By the smoothness of $U$, Leibniz integral rule and \zcref{cal:exp2:eq1}, the required smoothness of $\partial_\mu^{p} F_1 (y, \mu; v)$ in $(y, v)$ follows. Thus $F_1$ is of class $\sM_b^k$. By induction, $F_i$ is of class $\sM_b^k$ for every $i \in \llbracket 1, r \rrbracket$.

\item WLOG, we assume $q=2$. Then $V(s, t) = U(s, t, \mu_t)$ for $(s, t) \in \Delta^2$. We fix $(m_1, m_2) \in \bN^q$ such that $2 (m_1 + m_2) \le k$. We fix $s \in \bT$ and define $\varphi \colon [0, s] \times \sP_2  (\bR^d) \to \bR$ by $\varphi (t, \mu) \coloneq \partial^{m_1}_{s}  U(s, t, \mu)$. Clearly, $\varphi$ is of class $\sM_b^{k-2m_1}$ and $\partial^{m_1}_{s} V(s, t) = \varphi (t, \mu_t)$. We define an operator $\rL$ acting on $F \colon [0, s] \times \sP_2 (\bR^d) \to \bR$ of class $\sM_b^2$ by
\begin{equation} \label{cal:exp2:eq3}
\rL F (t, \mu) \coloneq \begin{myaligned}[t]
& \bE \Big [ \partial_t F (t, \mu) + b(t, Y, \mu) \cdot \partial_\mu F (t, \mu; Y) \\
& + \frac{1}{2} a(t, Y, \mu) : \nabla_v \partial_\mu F (t, \mu; Y ) \Big ] 
\qtextq{where} \lbrbrak Y \rbrbrak = \mu .
\end{myaligned} 
\end{equation}

By \zcref{cal:exp}[(1) and (3)], $\rL \varphi$ is of class $\sM_b^{k-2m_1-2}$. By \zcref{cal:Ito-lemma}, we have for all $t \in [0, s)$ that
\[
\partial^{m_1}_{s} V (s, t+h) - \partial^{m_1}_{s} V (s, t) = \int_t^{t+h} \rL \varphi (u, \mu_u) \diff u .
\]

Then for all $t \in [0, s)$,
\[
\lim_{h \downarrow 0} \frac{\partial^{m_1}_{s} V (s, t+h) - \partial^{m_1}_{s} V (s, t)}{h} = \rL \varphi (t, \mu_t).
\]

Similarly, we have for all $t \in (0, s]$ that
\[
\lim_{h \downarrow 0} \frac{ \partial^{m_1}_{s} V (s, t) - \partial^{m_1}_{s} V (s, t-h)}{h} = \rL \varphi (t, \mu_t) .
\]

Thus $[0, s] \ni t \mapsto \partial^{m_1}_{s} V (s, t)$ is differentiable with $\partial_t \partial^{m_1}_{s} V (s, t) = \rL \varphi (t, \mu_t)$. Let $\rL^{m_2} \coloneq \rL \circ \cdots \circ \rL$ be the $m_2$-times composition of $\rL$. Repeating the above procedure, we obtain
\begin{itemize}
\item $\rL^{m_2} \varphi$ is of class $\sM_b^{k-2 (m_1+m_2)}$.

\item $[0, s] \ni t \mapsto \partial^{m_1}_{s} V (s, t)$ is $m_2$-times differentiable with $m_2$-th derivative $\partial_t^{m_2} \partial^{m_1}_{s} V (s, t) = \rL^{m_2} \varphi (t, \mu_t)$. This completes the proof.
\end{itemize}
\end{enumerate}
\end{proof}

\section{Proof of regularity of $\tilde{\cV}$ in \zcref{cal:master-PDE}} \label{cal:master-PDE:prf}
We need an auxiliary lemma:

\begin{lemma} \label{par-form:2}
Let $q, r, k \in \bN$ and $U \colon \Delta^q \times \bR^r \times \sP_2 (\bR^d) \to \bR$ be of class $\sM^k_b$. Let $m \in \bN$ such that $m \ge d$. Let $\pi \colon \bR^m \to \bR^d$ such that $\pi (x_1, \ldots, x_m) = (x_1, \ldots, x_d)$ for every $(x_1, \ldots, x_m) \in \bR^m$. We define $V \colon \Delta^q \times \bR^r \times \sP_2 (\bR^m) \to \bR$ by $V (s, x, \varrho) \coloneq U (s, x, \pi_\sharp \varrho)$. Recall that $\pi_\sharp \varrho$ denotes the push-forward measure of $\varrho$ through $\pi$. Then $V$ is of class $\sM^k_b$.
\end{lemma}

\begin{proof}
By Leibniz integral rule, we can assume that $q=r=0$. We will prove the result for $(m, d) = (2, 1)$. The proof for the general case is similar. Then $U \colon \sP_2 (\bR) \to \bR$ and $V \colon \sP_2 (\bR^2) \to \bR$ with $V (\varrho) = U (\pi_\sharp \varrho)$.

For brevity, we denote $L^2 \coloneq L^2 (\cF_0; \bR)$. Let $\tilde U \colon L^2 \to \bR$ be the lift of $U$. Clearly, $L^2 (\cF_0; \bR^2)$ is isomorphic to $L^2 \times L^2$. Let $\tilde V \colon L^2 \times L^2 \to \bR$ be the lift of $V$. Let $\Psi \colon L^2 \times L^2 \to L^2$ be the projection of $L^2 \times L^2$ onto its first component. Then $\Psi$ is linear continuous and $\tilde V = \tilde U \circ \Psi$. Let $X = (X_1, X_2) \in L^2 \times L^2$. Notice that $\tilde U$ is Fréchet differentiable at $X_1$. By chain rule (see e.g. \cite[Theorem 3.3]{amann_analysis_2008}), $\tilde V$ is Fréchet differentiable at $X$ with $\partial \tilde V (X) = \partial \tilde U (\Psi (X)) \circ \partial \Psi (X)$. Then for all $Y = (Y_1, Y_2) \in L^2 \times L^2$,
\[
\partial \tilde V (X) (Y) = \langle \nabla \tilde U (X_1) , Y_1 \rangle_{L^2} = \langle ( \nabla \tilde U (X_1), 0) , Y \rangle_{L^2 \times L^2} .
\]

Thus
\begin{equation} \label{cal:id-proj:eq0}
\nabla \tilde V (X) = ( \nabla \tilde U (X_1), 0) = (\partial_\mu U (\lbrbrak X_1 \rbrbrak ; X_1), 0) .
\end{equation}

Notice that $\partial_\mu V \colon \sP_2 (\bR^2) \times \bR^2 \to \bR^2$. By \zcref{cal:id-proj:eq0}, we have for all $v_1 \in \bR^2$ that
\begin{align}
[\partial_\mu V]_1 (\mu ; v_1) & = \partial_\mu U (\pi_\sharp \mu ; \pi (v_1) ) , \\
[\partial_\mu V]_2 (\mu ; v_1) & = 0 .
\end{align}

Let $p \in \llbracket 1, k  \rrbracket$ and $\mathbf{1} = (1, \ldots, 1) \in \{ 1, 2 \}^p$. Notice that $\partial_\mu^p V \colon \sP_2 (\bR^2) \times (\bR^2)^p \to (\bR^2)^{\otimes p}$. Repeating the above procedure, we have for all $v_1, \ldots, v_p \in \bR^2$,
\begin{align} \label{cal:id-proj:eq1}
[\partial_\mu^p V]_{\mathbf{1}} (\mu ; v_1, \ldots, v_p) & = \partial_\mu^p U (\pi_\sharp \mu ; \pi (v_1), \ldots, \pi (v_p)) , \\
[\partial_\mu^p V]_{i} (\mu ; v_1, \ldots, v_p) & = 0 \qtextq{for every} i \in \{ 1, 2 \}^p \setminus {\mathbf{1}} .
\end{align}

The required smoothness of $V$ follows from \zcref{cal:id-proj:eq1} and that of $U$. This completes the proof.
\end{proof}

Let us prove the regularity of $\tilde{\cV}$. We will adopt the variational approach as in \cite{buckdahn_mean-field_2017,chassagneux_probabilistic_2022}. Recall that $\tilde \cV (s_1, \ldots, s_{q}, r, \mu) = \tilde \cU (s_1, \ldots, s_{q}, \lbrbrak X_{s_q}^{r, \mu} \rbrbrak)$ for all $(s_1, \ldots, s_{q}, r) \in \Delta^{q+1}$ and $\mu \in \sP_2 (\bR^d)$. If $(m_1, \ldots, m_{q-1}) \in \bN^{q-1}$ such that $2(m_1 + \cdots + m_{q-1}) \le k$, then
\[
\partial^{m_{q-1}}_{s_{q-1}} \ldots \partial^{m_1}_{s_1} \tilde \cV (s_1, \ldots, s_{q-1}, s_{q}, r, \mu) = \partial^{m_{q-1}}_{s_{q-1}} \ldots \partial^{m_1}_{s_1} \tilde \cU (s_1, \ldots, s_{q-1}, s_{q}, \lbrbrak X_{s_q}^{r, \mu} \rbrbrak) .
\]

So the crux of the problem is the regularity of the map
\[
\Delta^2 \times \sP_2 (\bR^d) \ni ((s_q, r), \mu) \mapsto \partial^{m_{q-1}}_{s_{q-1}} \ldots \partial^{m_1}_{s_1} \tilde \cU (s_1, \ldots, s_{q-1}, s_{q}, \lbrbrak X_{s_q}^{r, \mu} \rbrbrak) .
\]

Therefore, we can assume $q=1$. Then $\cV (s, r, \mu) = \tilde \cU (s, \lbrbrak X_{s}^{r, \mu} \rbrbrak)$ for all $(s, r) \in \Delta^2$ and $\mu \in \sP_2 (\bR^d)$.

For $(x, \xi) \in \bR^d \times L^2 (\cF_r; \bR^d)$, let $(X_t^{r, \xi}, Y_t^{r, x, \lbrbrak \xi \rbrbrak})_{t \in  [r, T]}$ be the unique solution to the system
\begin{equation} \label{cal:master-PDE:prf:eq1}
\begin{dcases}
X_t^{r, \xi} & = \xi + \int_r^t b (s, X_s^{r, \xi}, \lbrbrak X_s^{r, \xi} \rbrbrak ) \diff s  + \int_r^t \sigma (s, X_s^{r, \xi}, \lbrbrak X_s^{r, \xi} \rbrbrak ) \diff B_s , \\
Y_t^{r, x, \lbrbrak \xi \rbrbrak} & = x + \int_r^t b (s, Y_s^{r, x, \lbrbrak \xi \rbrbrak}, \lbrbrak X_s^{r, \xi} \rbrbrak ) \diff s + \int_r^t \sigma (s, Y_s^{r, x, \lbrbrak \xi \rbrbrak}, \lbrbrak X_s^{r, \xi} \rbrbrak ) \diff B_s . 
\end{dcases}
\end{equation}

$\{ Y_t^{r, x, \lbrbrak \xi \rbrbrak} \}_{t \in  [r, T]}$ is called the \textit{decoupled} process because it does not depend on its own marginal distribution. By weak uniqueness of the first equation in \zcref{cal:master-PDE:prf:eq1}, $\lbrbrak X_t^{r, \xi} \rbrbrak$ depends on $\xi$ only through its distribution $\lbrbrak \xi \rbrbrak$. As a consequence, $Y_t^{r, x, \lbrbrak \xi \rbrbrak}$ depends on $\xi$ only through $\lbrbrak \xi \rbrbrak$. The idea is to express the derivatives of $\xi \mapsto X_t^{r, \xi}$ through those of $\xi \mapsto Y_t^{r, x, \lbrbrak \xi \rbrbrak}$ and of $x \mapsto Y_t^{r, x, \lbrbrak \xi \rbrbrak}$. With the same reasoning as in \cite[Lemma 13.8]{bonaccorsi2006brownian}, we obtain
\begin{equation} \label{cal:master-PDE:prf:eq4}
\sup_{t \in [r, T]} \big | X_t^{r, \xi} - Y_t^{r, x, \lbrbrak \xi \rbrbrak} |_{x=\xi} \big | = 0 \quad \bP \text{-a.s.},
\end{equation}
where
\[
\{ Y_t^{r, x, \lbrbrak \xi \rbrbrak} |_{x=\xi} \} (\omega) \coloneq Y_t^{r, \xi (\omega), \lbrbrak \xi \rbrbrak} (\omega)
\qtextq{for}
\omega \in \Omega.
\]

We introduce a sequence $(\Omega^{(n)}, \bF^{(n)}, \bP^{(n)}, \bE^{(n)})_{n \in \bN^*}$ of pointwise copies of $(\Omega, \bF, \bP, \bE)$. With any random variable $\eta$ defined on $(\Omega, \bF, \bP)$, we associate a pointwise copy $\eta^{(n)}$ on $(\Omega^{(n)}, \bF^{(n)}, \bP^{(n)})$. For $t \in \bR_+$, we denote by $B_t = \{ B_{t, \ell} \}_{1 \le \ell \le d}$ the components of the Brownian motion.

\begin{enumerate}
\item \textit{First-order derivative of $x \mapsto Y_t^{r, x, \lbrbrak \xi \rbrbrak}$.} By Kunita's theory of stochastic flows (see e.g. \cite[Theorem 3.3.2]{kunita_stochastic_2019} or \cite[Lemma 4.1]{buckdahn_mean-field_2017}), for each $(x, \xi) \in \bR^d \times L^2 (\cF_r; \bR^d)$, there exists an $\bR^d \otimes \bR^d$-valued process $\{ \nabla_x Y_t^{r, x, \lbrbrak \xi \rbrbrak} \}_{t \in [r, T]}$ such that
\[
\lim_{\bR^d \ni y \to 0} \bE \Big [ \sup_{t \in [r, T]} \Big | \frac{Y_t^{r, x+y, \lbrbrak \xi \rbrbrak} - Y_t^{r, x, \lbrbrak \xi \rbrbrak} - \nabla_{x} Y_t^{r, x, \lbrbrak \xi \rbrbrak} y }{|y|}  \Big |^2 \Big ] = 0 .
\]

We remind that the term $\nabla_{x} Y_t^{r, x, \lbrbrak \xi \rbrbrak} y$ denotes the multiplication of the matrix $\nabla_{x} Y_t^{r, x, \lbrbrak \xi \rbrbrak}$ and the column vector $y$. Moreover, $\nabla_x Y_t^{r, x, \lbrbrak \xi \rbrbrak} = \{ \partial_{x_j} Y_{t, i}^{r, x, \lbrbrak \xi \rbrbrak} \}_{1 \le i, j \le d}$ is the unique solution of the SDE
\begin{equation} \label{cal:master-PDE:prf:eq:-1}
\partial_{x_j} Y_{t, i}^{r, x, \lbrbrak \xi \rbrbrak} = \begin{myaligned}[t]
& \delta_{i, j} + \sum_{k=1}^d \int_r^t \diff_k b_{i} (s, Y_s^{r, x, \lbrbrak \xi \rbrbrak}, \lbrbrak X_s^{r, \xi} \rbrbrak) \partial_{x_j} Y_{s, k}^{r, x, \lbrbrak \xi \rbrbrak}  \diff s \\
& + \sum_{k, \ell=1}^d \int_r^t \diff_k \sigma_{i, \ell} (s, Y_s^{r, x, \lbrbrak \xi \rbrbrak}, \lbrbrak X_s^{r, \xi} \rbrbrak) \partial_{x_j} Y_{s, k}^{r, x, \lbrbrak \xi \rbrbrak} \diff B_{s, \ell} .
\end{myaligned}
\end{equation}
where $t \in [r, T]$ and $1 \le i, j \le d$.

Above, $\delta_{i, j}$ denotes the Kronecker symbol, i.e., $\delta_{i, j} = 1$ if $i=j$ and $0$ otherwise; $\diff_k b_{i} (s, y, \mu)$ (resp. $\diff_k \sigma_{i, \ell} (s, y, \mu)$) denotes the partial derivative of $b_i$ (resp. $\sigma_{i, \ell}$) w.r.t $y_k$. In comparison to \cite[Equation (4.1)]{buckdahn_mean-field_2017}, \zcref{cal:master-PDE:prf:eq:-1} additionally has time component in the coefficients $(b, \sigma)$ and their derivatives. By \cite[Inequality (4.2)]{buckdahn_mean-field_2017}, for each $p \in [2, \infty)$, there exists a constant $C_p \in \bR$ such that for all $r \in [0, T]; x, x' \in \bR^d$ and $\xi, \xi' \in L^2 (\cF_r; \bR^d)$:
\begin{align}
\bE \Big [ \sup_{t \in [r, T]} | \nabla_x Y_t^{r, x, \lbrbrak \xi \rbrbrak} |^p \Big ] & \le C_p, \\
\bE \Big [ \sup_{t \in [r, T]} | \nabla_x Y_t^{r, x, \lbrbrak \xi \rbrbrak} - \nabla_x Y_t^{r, x', \lbrbrak \xi' \rbrbrak} |^p \Big ] & \le C_p \{|x-x'|^p + \sW_2^p (\lbrbrak \xi \rbrbrak, \lbrbrak \xi' \rbrbrak)\} .
\end{align}

\item \textit{First-order derivative of $\xi \mapsto Y_t^{r, x, \lbrbrak \xi \rbrbrak}$.} This is accomplished in \cite[Section 4]{buckdahn_mean-field_2017} which we summarize in the following. Let $\cL (L^2 (\cF_r; \bR^d) ; L^2 (\cF_t; \bR^d))$ be the space of bounded linear operators from $L^2 (\cF_r; \bR^d)$ to $L^2 (\cF_t; \bR^d)$. By \cite[Theorem 4.1]{buckdahn_mean-field_2017},
\[
L^2 (\cF_r; \bR^d) \to L^2 (\cF_t; \bR^d) , \, \xi \mapsto Y_t^{r, x, \lbrbrak \xi \rbrbrak}
\]
is Fréchet differentiable with its Fréchet derivative of the form
\begin{equation} \label{cal:master-PDE:prf:eq3:a}
\begin{array}{lrcl}
& L^2 (\cF_r; \bR^d) & \longrightarrow & \cL (L^2 (\cF_r; \bR^d) ; L^2 (\cF_t; \bR^d)) \\
& \xi & \longmapsto & \{ \eta \mapsto \bE^{(1)} [ U^{r, x, \lbrbrak \xi \rbrbrak}_t (\xi^{(1)}) \eta^{(1)} ] \}
\end{array} ,
\end{equation}
for some $\bR^d \otimes \bR^d$-valued process $\{U^{r, x, \lbrbrak \xi \rbrbrak}_t (y) \}_{ t \in [r, T]}$ adapted to $\{ \cF_t \}_{ t \in [r, T]}$. We remind that the term $U^{r, x, \lbrbrak \xi \rbrbrak}_t (\xi^{(1)}) \eta^{(1)}$ denotes the multiplication of the matrix $U^{r, x, \lbrbrak \xi \rbrbrak}_t (\xi^{(1)})$ and the column vector $\eta^{(1)}$.

By \cite[Lemma 4.3]{buckdahn_mean-field_2017}, for each $p \in [2, \infty)$, there exists a constant $C_p \in \bR$ such that for all $r \in [0, T]; x, x', y, y' \in \bR^d$ and $\xi, \xi' \in L^2 (\cF_r; \bR^d)$:
\begin{align}
\bE \Big [ \sup_{t \in [r, T]} | U_t^{r, x, \lbrbrak \xi \rbrbrak} (y) |^p \Big ] & \le C_p, \\
\bE \Big [ \sup_{t \in [r, T]} | U_t^{r, x, \lbrbrak \xi \rbrbrak} (y) - U_t^{r, x', \lbrbrak \xi' \rbrbrak} (y') |^p \Big ] & \le C_p \{|x-x'|^p + |y-y'|^p + \sW_2^p (\lbrbrak \xi \rbrbrak, \lbrbrak \xi' \rbrbrak)\} .
\end{align}

By \cite[Theorem 4.1]{buckdahn_mean-field_2017}, $U_t^{r, x, \lbrbrak \xi \rbrbrak} (y) = \{ U_{t, i, j}^{r, x, \lbrbrak \xi \rbrbrak} (y) \}_{1 \le i, j \le d}$ is the unique solution to the SDE
\begin{equation} 
U_{t, i, j}^{r, x, \lbrbrak \xi \rbrbrak} (y) = \begin{myaligned}[t]
& \sum_{k, \ell=1}^d \int_r^t \diff_k \sigma_{i, \ell} (s, Y_s^{r, x, \lbrbrak \xi \rbrbrak}, \lbrbrak X_s^{r, \xi} \rbrbrak) U_{s, k, j}^{r, x, \lbrbrak \xi \rbrbrak} (y) \diff B_{s, \ell} \\
& + \sum_{k=1}^d \int_r^t \diff_k b_{i} (s, Y_s^{r, x, \lbrbrak \xi \rbrbrak}, \lbrbrak X_s^{r, \xi} \rbrbrak) U_{s, k, j}^{r, x, \lbrbrak \xi \rbrbrak} (y) \diff s \\
& + \sum_{k, \ell=1}^d \int_r^t \bE \Big [ (\partial_\mu \sigma_{i, \ell})_k (s, z, \lbrbrak X_s^{r, \xi} \rbrbrak; Y_s^{r, y, \lbrbrak \xi \rbrbrak}) \partial_{x_j} Y_{s, k}^{r, y, \lbrbrak \xi \rbrbrak} \\
& + (\partial_\mu \sigma_{i, \ell})_k (s, z, \lbrbrak X_s^{r, \xi} \rbrbrak; X_s^{r, \xi}) U_{s, k, j}^{r, \xi} (y) \Big ] \Big |_{z=Y_s^{r, x, \lbrbrak \xi \rbrbrak}} \diff B_{s, \ell} \\
& + \sum_{k=1}^d \int_r^t \bE \Big  [ (\partial_\mu b_{i})_k (s, z, \lbrbrak X_s^{r, \xi} \rbrbrak; Y_s^{r, y, \lbrbrak \xi \rbrbrak}) \partial_{x_j} Y_{s, k}^{r, y, \lbrbrak \xi \rbrbrak} \\
& + (\partial_\mu b_{i})_k (s, z, \lbrbrak X_s^{r, \xi} \rbrbrak; X_s^{r, \xi}) U_{s, k, j}^{r, \xi} (y) \Big ] \Big |_{z=Y_s^{r, x, \lbrbrak \xi \rbrbrak}} \diff s ,
\end{myaligned} \label{cal:master-PDE:prf:eq2}
\end{equation}
where $t \in [r, T]$ and $1 \le i, j \le d$; and the process $U_t^{r, \xi} (y) = \{ U_{t, i, j}^{r, \xi} (y) \}_{1 \le i, j \le d}$, which appears in the 4-th and last lines of \zcref{cal:master-PDE:prf:eq2}, is the unique solution to the SDE
\begin{equation} 
U_{t, i, j}^{r, \xi} (y) = \begin{myaligned}[t]
& \sum_{k, \ell=1}^d \int_r^t \diff_k \sigma_{i, \ell} (s, X_s^{r, \xi}, \lbrbrak X_s^{r, \xi} \rbrbrak) U_{s, k, j}^{r, \xi} (y) \diff B_{s, \ell} \\
& + \sum_{k=1}^d \int_r^t \diff_k b_{i} (s, X_s^{r, \xi}, \lbrbrak X_s^{r, \xi} \rbrbrak) U_{s, k, j}^{r, \xi} (y) \diff s \\
& + \sum_{k, \ell=1}^d \int_r^t \bE \Big [ (\partial_\mu \sigma_{i, \ell})_k (s, z, \lbrbrak X_s^{r, \xi} \rbrbrak; Y_s^{r, y, \lbrbrak \xi \rbrbrak}) \partial_{x_j} Y_{s, k}^{r, y, \lbrbrak \xi \rbrbrak} \\
& + (\partial_\mu \sigma_{i, \ell})_k (s, z, \lbrbrak X_s^{r, \xi} \rbrbrak; X_s^{r, \xi}) U_{s, k, j}^{r, \xi} (y) \Big ] \Big |_{z=X_s^{r, \xi}} \diff B_{s, \ell} \\
& + \sum_{k=1}^d \int_r^t \bE \Big  [ (\partial_\mu b_{i})_k (s, z, \lbrbrak X_s^{r, \xi} \rbrbrak; Y_s^{r, y, \lbrbrak \xi \rbrbrak}) \partial_{x_j} Y_{s, k}^{r, y, \lbrbrak \xi \rbrbrak} \\
& + (\partial_\mu b_{i})_k (s, z, \lbrbrak X_s^{r, \xi} \rbrbrak; X_s^{r, \xi}) U_{s, k, j}^{r, \xi} (y) \Big ] \Big |_{z=X_s^{r, \xi}} \diff s ,
\end{myaligned}
\end{equation}
where $t \in [r, T]$ and $1 \le i, j \le d$.

Notice that $U_{t, i, j}^{r, x, \lbrbrak \xi \rbrbrak} (y)$ depends on $\xi$ only through $\lbrbrak \xi \rbrbrak$. In comparison to \cite[Equation (4.48)]{buckdahn_mean-field_2017}, \zcref{cal:master-PDE:prf:eq2} additionally has time component in the coefficients $(b, \sigma)$ and their derivatives.

The derivative of $Y_t^{r, x, \lbrbrak \xi \rbrbrak}$ in measure is defined as
\[
\partial_\mu Y_t^{r, x, \lbrbrak \xi \rbrbrak} (y) \coloneq U^{r, x, \lbrbrak \xi \rbrbrak}_t (y)
\qtextq{for} t \in [r, T]; x, y \in \bR^d .
\]

For $p \in \bN$ with $p \ge 2$, the $p$-th derivative of $Y_t^{r, x, \lbrbrak \xi \rbrbrak}$ in measure is defined inductively as
\[
\partial_\mu^p Y_t^{r, x, \lbrbrak \xi \rbrbrak} (v_1, \ldots, v_p) \coloneq \partial_\mu \{ \partial_\mu^{p-1} Y_t^{r, x, \lbrbrak \xi \rbrbrak} (v_1, \ldots, v_{p-1}) \} (v_p)
\]
for all $v_1, \ldots, v_p \in \bR^d$ and $t \in [r, T]$, provided that they actually exist. For $d=1$, the readers can find the formula for $\partial_\mu^2 Y_t^{r, x, \lbrbrak \xi \rbrbrak} (v_1, v_2)$ in \cite[Equation (5.15)]{buckdahn_mean-field_2017}, which is one-page long.

\item \textit{First-order derivative of $\mu \mapsto \tilde \cV (t, r, \mu)$ where $(t, r) \in \Delta^2$}. We define the directional derivatives of $\xi \mapsto Y_t^{r, x, \lbrbrak \xi \rbrbrak}$ by
\begin{equation} \label{cal:master-PDE:prf:eq3}
\rD_\xi Y_t^{r, x, \lbrbrak \xi \rbrbrak} (\eta) \coloneq \lim_{\bR \ni h \to 0} \frac{Y_{t}^{r, x, \lbrbrak \xi + h \eta \rbrbrak} - Y_{t}^{r, x, \lbrbrak \xi \rbrbrak}}{h} 
\qtextq{for} \xi, \eta \in L^2 (\cF_r; \bR^d) .
\end{equation}

The limit in \zcref{cal:master-PDE:prf:eq3} is taken in $L^2 (\cF_t; \bR^d)$. We denote
\begin{align} \label{cal:master-PDE:prf:eq:-3}
\nabla_x X_t^{r, \xi, \lbrbrak \xi \rbrbrak} & \coloneq \nabla_x Y_t^{r, x, \lbrbrak \xi \rbrbrak} |_{x=\xi} , \\
\partial_\mu X_t^{r, \xi, \lbrbrak \xi \rbrbrak} (y) & \coloneq \partial_\mu Y_t^{r, x, \lbrbrak \xi \rbrbrak} (y) |_{x=\xi} .
\end{align}

By \zcref{cal:master-PDE:prf:eq4} and then formal chain rule, we have
\begin{align} 
\partial_h X_t^{r, \xi + h \eta} |_{h=0} & = \partial_h \{ Y_t^{r, x, \lbrbrak \xi + h \eta \rbrbrak} |_{x = \xi + h \eta} \} |_{h=0} \\
& = \{ \nabla_x Y_t^{r, x, \lbrbrak \xi \rbrbrak} |_{x=\xi} \} \eta + \rD_\xi Y_t^{r, x, \lbrbrak \xi \rbrbrak} (\eta) |_{x=\xi} \\
& = \nabla_x X_t^{r, \xi, \lbrbrak \xi \rbrbrak} \eta + \rD_\xi Y_t^{r, x, \lbrbrak \xi \rbrbrak} (\eta) |_{x=\xi} . \label{cal:master-PDE:prf:eq4:a}
\end{align}

By \zcref{cal:master-PDE:prf:eq3:a} and \zcref{cal:master-PDE:prf:eq4:a},
\[
\Phi^r_t \colon L^2 (\cF_r; \bR) \to L^2 (\cF_t; \bR) , \, \xi \mapsto X_t^{r, \xi}
\]
is Fréchet differentiable with Fréchet derivative of the form
\begin{equation} \label{cal:master-PDE:prf:eq5}
\begin{array}{lrcl}
& \rD \Phi^r_t  \colon L^2 (\cF_r; \bR) & \longrightarrow & \cL (L^2 (\cF_r; \bR) ; L^2 (\cF_t; \bR)) \\
& \xi & \longmapsto & \{ \eta \mapsto \nabla_x X_t^{r, \xi, \lbrbrak \xi \rbrbrak} \eta + \bE^{(1)} [ \partial_\mu X_t^{r, \xi, \lbrbrak \xi \rbrbrak} (\xi^{(1)}) \eta^{(1)} ] \}
\end{array} .
\end{equation}

Let $\hat \sV_{t, r} \colon L^2 (\cF_{r}; \bR^d) \to \bR$ be the lift of $\tilde \sV (t, r, \cdot)$. Let $\hat \sU_t \colon L^2 (\cF_{t}; \bR^d) \to \bR$ be the lift of $\tilde \sU (t, \cdot)$. This means $\hat \sV_{t, r} (\xi) = \tilde \sV (t, r, \lbrbrak \xi \rbrbrak)$ and $\hat \sU_t (\xi) = \tilde \sU (t, \lbrbrak \xi \rbrbrak)$. Then $\hat \sV_{t, r} (\xi) = \hat \sU_t (\Phi^r_t (\xi))$.

By the chain rule for Fréchet derivatives,
\[
\rD \hat \cV_{t, r} (\xi) = \rD \hat \sU_t (X_t^{r, \xi}) \circ \rD \Phi^r_t (\xi) .
\]

By \zcref{cal:master-PDE:prf:eq5}, we have for all $\xi, \eta \in L^2 (\cF_{r}; \bR^d)$:
\begin{align}
\rD \hat \cV_{t, r} (\xi) (\eta) & = \bE [ \partial_\mu \tilde \sU (t, \lbrbrak X_{t}^{r, \xi} \rbrbrak ; X_{t}^{r, \xi} ) \cdot \{ \nabla_x X_t^{r, \xi, \lbrbrak \xi \rbrbrak} \eta + \bE^{(1)} [ \partial_\mu X_t^{r, \xi, \lbrbrak \xi \rbrbrak} (\xi^{(1)}) \eta^{(1)} ] \} ] \\
& = \begin{myaligned}[t] 
& \bE [ \eta \cdot [ \{ \nabla_x X_t^{r, \xi, \lbrbrak \xi \rbrbrak} \}^\top \partial_\mu \tilde \sU (t, \lbrbrak X_{t}^{r, \xi} \rbrbrak ; X_{t}^{r, \xi} ) ]] \\
& + \bE \bE^{(1)} [ \eta^{(1)} \cdot [ \{ \partial_\mu X_t^{r, \xi, \lbrbrak \xi \rbrbrak} (\xi^{(1)}) \}^\top \partial_\mu \tilde \sU (t, \lbrbrak X_{t}^{r, \xi} \rbrbrak ; X_{t}^{r, \xi} ) ] ] .
\end{myaligned}
\end{align}

By \zcref{cal:master-PDE:prf:eq4}, \zcref{cal:master-PDE:prf:eq:-3} and the freezing lemma,
\begin{myalign}
&  \bE [ \eta \cdot [ \{ \nabla_x X_t^{r, \xi, \lbrbrak \xi \rbrbrak} \}^\top \partial_\mu \tilde \sU (t, \lbrbrak X_{t}^{r, \xi} \rbrbrak ; X_{t}^{r, \xi} ) ]  ] \\
& = \bE [ \eta \cdot \bE^{(1)} [ \{ \nabla_x (Y^{(1)})_t^{r, x, \lbrbrak \xi \rbrbrak} \}^\top \partial_\mu \tilde \sU (t, \lbrbrak X_{t}^{r, \xi} \rbrbrak ; (Y^{(1)})_{t}^{r, x, \lbrbrak \xi \rbrbrak} ) ] |_{x = \xi} ] .
\end{myalign}

By Fubini's theorem,
\[
\begin{myaligned}[t]
& \bE \bE^{(1)} [ \eta^{(1)} \cdot [ \{ \partial_\mu X_t^{r, \xi, \lbrbrak \xi \rbrbrak} (\xi^{(1)}) \}^\top \partial_\mu \tilde \sU (t, \lbrbrak X_{t}^{r, \xi} \rbrbrak ; X_{t}^{r, \xi} ) ] ] \\
& = \bE [ \eta \cdot \bE^{(1)} [ \{ \partial_\mu (X^{(1)})_{t}^{r, \xi^{(1)}, \lbrbrak \xi \rbrbrak} ( \xi ) \}^\top \partial_\mu \tilde \sU (t, \lbrbrak X_{t}^{r, \xi} \rbrbrak ; (X^{(1)})_{t}^{r, \xi^{(1)}} )  ] ]  \\
& = \bE [ \eta \cdot \bE^{(1)} [ \{ \partial_\mu (X^{(1)})_{t}^{r, \xi^{(1)}, \lbrbrak \xi \rbrbrak} ( x ) \}^\top \partial_\mu \tilde \sU (t, \lbrbrak X_{t}^{r, \xi} \rbrbrak ; (X^{(1)})_{t}^{r, \xi^{(1)}} )  ] |_{x = \xi} ]  .
\end{myaligned}
\]

Hence we have for all $x \in \bR^d$,
\begin{equation} \label{cal:master-PDE:prf:value-function}
\partial_\mu \tilde \sV (t, r, \lbrbrak \xi \rbrbrak; x ) = \begin{myaligned}[t] 
& \bE^{(1)} [ \{ \nabla_x (Y^{(1)})_t^{r, x, \lbrbrak \xi \rbrbrak} \}^\top \partial_\mu \tilde \sU (t, \lbrbrak X_{t}^{r, \xi} \rbrbrak ; (Y^{(1)})_{t}^{r, x, \lbrbrak \xi \rbrbrak} ) \\
& + \{ \partial_\mu (X^{(1)})_{t}^{r, \xi^{(1)}, \lbrbrak \xi \rbrbrak} ( x ) \}^\top \partial_\mu \tilde \sU (t, \lbrbrak X_{t}^{r, \xi} \rbrbrak ; (X^{(1)})_{t}^{r, \xi^{(1)}} )  ] .
\end{myaligned}
\end{equation}
\end{enumerate}

\begin{remark}
The readers probably recognize that higher-order Lions differentiation will unavoidably lead to complicated expressions. Therefore, in the rest of the proof, we assume that at least one of the following two simplifications holds:
\begin{itemize}
\item In the first setting called \MakeLinkTarget*{cal:master-PDE:prf:setting:A}\textbf{Setting A}, important in view of applications, $d=1, b(t, x, \mu) = b(t, x)$ and $\sigma (t, x, \mu) = \sigma (t, x)$.
\item In the second setting called \MakeLinkTarget*{cal:master-PDE:prf:setting:B}\textbf{Setting B}, as in \cite[Theorem 2.18]{chassagneux_weak_2022}, $d=1, b=0$ and $\sigma (t, x, \mu) = \sigma (t, \mu)$.
\end{itemize}
\end{remark}

\begin{theorem} \label{cal:master-PDE:prf:long-thm}
Let $(b, \sigma, \tilde \cU)$ be of class $\sM^k_b$ and $0 \le r \le t \le T$. Let either of \hyperlink{cal:master-PDE:prf:setting:A}{\textbf{Setting A}} or \hyperlink{cal:master-PDE:prf:setting:B}{\textbf{Setting B}} hold. Then $\sP_2 (\bR) \times \bR^p \ni (\mu, v) \mapsto \tilde \sV (t, r, \mu; v)$ is of class $\sM^k_b$. Moreover, if $p \in \llbracket 1, k \rrbracket$ and $\ell = (\ell_1, \ldots, \ell_p) \in \bN^p$ such that $|\ell|_1 \coloneq \ell_1 + \cdots + \ell_p \le k-p$, there exist functions $(F_{(1, p, \ell)}, F_{(2, p, l)}, H_{(p, \ell)})$ with the following properties:
\begin{itemize}
\item $(F_{(1, p, \ell)}, F_{(2, p, l)}, H_{(p, \ell)})$ depend only on $(b, \sigma, \tilde \cU, p, \ell)$ and are of class $\sM^{k-p-|\ell|_1}_b$.
\item For all $(\xi, v) \in L^2 (\cF_r; \bR) \times \bR^p$,
\[
\rD^{(p, 0, 0, \ell)} \tilde \sV (t, r, \lbrbrak \xi \rbrbrak; v ) = H_{(p, \ell)} (t, \lbrbrak \mathbf{Y}^{r}_t (\xi; v)  \rbrbrak ) ,
\]
where
$\{ \mathbf{Y}^{r}_t (\xi; v) \}_{t \in [r, T]}$ is the unique solution to the multi-dimensional SDE
\begin{equation}
\mathbf{Y}^{r}_t (\xi; v) = \begin{myaligned}[t]
& \mathbf{Y}^{r}_r (\xi; v) + \int_r^t F_{(1, p, \ell)} (s, \mathbf{Y}^{r}_s (\xi; v), \lbrbrak \mathbf{Y}^{r}_s (\xi; v) \rbrbrak) \diff s \\
& + \int_r^t F_{(2, p, \ell)} (s, \mathbf{Y}^{r}_s (\xi; v), \lbrbrak \mathbf{Y}^{r}_s (\xi; v) \rbrbrak ) \diff \begin{bmatrix}
B^{(1)}_s \\
\vdots \\
B^{(p)}_s
\end{bmatrix} ,
\end{myaligned}
\end{equation}
and the map $(\xi, v) \mapsto \mathbf{Y}^{r}_r (\xi; v)$ is affine continuous and does not depend on $r$.
\end{itemize}
\end{theorem}

In the below proof of \zcref{cal:master-PDE:prf:long-thm}, the formula of $\mathbf{Y}^{r}_r (\xi; v)$ is explicit and we repeatedly use the following strategy. We will express the quantity of interest, e.g., $\partial_\mu^p \tilde \sV (t, r, \lbrbrak \xi \rbrbrak; v )$ in \zcref{cal:master-PDE:prf:eq:-8:a,cal:master-PDE:prf:eq6:a:b}; and $\rD^{(p, 0, 0, \ell)} \tilde \sV (t, r, \lbrbrak \xi \rbrbrak; v )$ in \zcref{cal:master-PDE:prf:system:e:a,cal:master-PDE:prf:system:e:b}, as a function of the marginal distribution of a (high-dimensional) SDE. This function as well as the coefficients of the SDE have a specific degree of smoothness due to their explicit constructions. For the notations of mixed derivatives, we refer the readers to \zcref{multi-index}.

\begin{proof}
First, we consider \hyperlink{cal:master-PDE:prf:setting:A}{\textbf{Setting A}} where $d=1, b(t, x, \mu) = b(t, x)$ and $\sigma (t, x, \mu) = \sigma (t, x)$. In this setting, we only start assuming $d=1$ at the end of part (1A).

\begin{itemize}
\item Notice that $( Y_t^{r, x, \lbrbrak \xi \rbrbrak} , \nabla_x Y_t^{r, x, \lbrbrak \xi \rbrbrak} )_{t \in  [r, T]}$ does not depend on $\xi$. Thus we write $Y_t^{r, x} \coloneq Y_t^{r, x, \lbrbrak \xi \rbrbrak}$ and $\nabla_x Y_t^{r, x} \coloneq \nabla_x Y_t^{r, x, \lbrbrak \xi \rbrbrak}$.

\item Also, $\partial_\mu X_t^{r, \xi, \lbrbrak \xi \rbrbrak} (y) = \partial_\mu Y_t^{r, x, \lbrbrak \xi \rbrbrak} (y) = 0$ for all $(t, x, \xi) \in [r, T] \times \bR^d \times L^2 (\cF_r; \bR^d)$.
\end{itemize}

First, \zcref{cal:master-PDE:prf:eq1} becomes
\begin{equation} 
\begin{dcases}
X_t^{r, \xi} & = \xi + \int_r^t b (s, X_s^{r, \xi} ) \diff s + \int_r^t \sigma (s, X_s^{r, \xi} ) \diff B_s , \\ 
Y_t^{r, x} & = x + \int_r^t b (s, Y_s^{r, x} ) \diff s + \int_r^t \sigma (s, Y_s^{r, x} ) \diff B_s . \label{cal:master-PDE:prf:eq1:a}
\end{dcases}
\end{equation}

Second, \zcref{cal:master-PDE:prf:eq:-1} becomes
\begin{equation} \label{cal:master-PDE:prf:eq:-1:a}
\partial_{x_j} Y_{t, i}^{r, x} = \begin{myaligned}[t]
& \delta_{i, j} + \sum_{k=1}^d \int_r^t \diff_k b_{i} (s, Y_s^{r, x}) \partial_{x_j} Y_{s, k}^{r, x}  \diff s + \sum_{k, \ell=1}^d \int_r^t \diff_k \sigma_{i, \ell} (s, Y_s^{r, x}) \partial_{x_j} Y_{s, k}^{r, x} \diff B_{s, \ell} .
\end{myaligned}
\end{equation}

It follows from \zcref{cal:master-PDE:prf:eq1:a,cal:master-PDE:prf:eq:-1:a} that
\begin{equation} \label{cal:master-PDE:prf:eq:-2:a}
\mathbf{U}^{r}_t (\xi; x) \coloneq \Big ( X^{r, \xi}_t,  Y^{r, x}_{t}, \{ \partial_{x_j} Y_{t, i}^{r, x} \}_{1 \le i, j \le d} \Big ) ,
\end{equation}
where $(\xi, x) \in L^2 (\cF_r; \bR^d) \times \bR^d$, can be written as the unique solution of the SDE
\begin{equation} \label{cal:master-PDE:prf:eq:-4:a}
\mathbf{U}^{r}_t (\xi; x) = \mathbf{U}^{r}_r (\xi; x) + \int_r^t F_1 (s, \mathbf{U}^{r}_s (\xi; x)  ) \diff s + \int_r^t F_2 (s, \mathbf{U}^{r}_s (\xi; x)  ) \diff B_s ,
\end{equation}
for some functions $(F_1, F_2)$ of class $\sM^{k-1}_b$.

\begin{enumerate}[label=(\theenumi A)]
\item \textit{Higher-order derivatives of $\mu \mapsto \tilde \cV (t, r, \mu)$ where $(t, r) \in \Delta^2$.} Notice that \zcref{cal:master-PDE:prf:value-function} becomes
\begin{equation} \label{cal:master-PDE:prf:value-function:a}
\partial_\mu \tilde \sV (t, r, \lbrbrak \xi \rbrbrak; x ) = \bE^{(1)} [ \{ \nabla_x (Y^{(1)})_t^{r, x} \}^\top \partial_\mu \tilde \sU (t, \lbrbrak X_{t}^{r, \xi} \rbrbrak ; (Y^{(1)})_{t}^{r, x} ) ] .
\end{equation}

Hence we have for all $(i_1, v_1) \in \llbracket 1, d \rrbracket \times \bR^d$ that
\begin{equation}
[\partial_\mu \tilde \sV]_{i_1} (t, r, \lbrbrak \xi \rbrbrak; v_1 ) = \sum_{j_1=1}^d \bE^{(1)} \Big [ \partial_{x_{i_1}} (Y^{(1)})_{t, j_1}^{r, v_1} [\partial_\mu \tilde \sU]_{j_1} (t, \lbrbrak X_{t}^{r, \xi} \rbrbrak ; (Y^{(1)})_{t}^{r, v_1} ) \Big ] .
\end{equation}

We denote
\[
\bE^{[p]} \coloneq \bE^{(1)} \ldots \bE^{(p)} .
\]

Iterating the above procedure, we obtain for all $\xi \in L^2 (\cF_r; \bR^d), p \in \llbracket 1, k \rrbracket, (i_1, \ldots, i_p) \in \llbracket 1, d \rrbracket^p$ and $v=(v_1, \ldots, v_p) \in (\bR^d)^p$ that
\begin{myalign} 
& [\partial_\mu^p \tilde \sV]_{(i_1, \ldots, i_p)} (t, r, \lbrbrak \xi \rbrbrak; v ) \\
& = \begin{myaligned}[t] \label{cal:master-PDE:prf:eq6:a}
& \sum_{j_1, \ldots, j_p=1 }^d \bE^{[p]} \Big [ \prod_{\ell=1}^p \partial_{x_{i_\ell}} (Y^{(\ell)})_{t, j_\ell}^{r, v_\ell} [\partial_\mu^p \tilde \sU]_{(j_1, \ldots, j_p)} (t, \lbrbrak X_{t}^{r, \xi} \rbrbrak ; (Y^{(1)})_{t}^{r, v_1} , \ldots, (Y^{(p)})_{t}^{r, v_p} ) \Big ] .
\end{myaligned}
\end{myalign}

Let us put the right-hand side of \zcref{cal:master-PDE:prf:eq6:a} into a more convenient form. This is done in three steps.
\begin{itemize}
\item \textit{Step 1:} For $\xi \in L^2 (\cF_r; \bR^d)$ and $v = (v_1, \ldots, v_p) \in (\bR^d)^p$, we denote
\[
\mathbf{X}^{r}_t (\xi; v) \coloneq \{ (\mathbf{U}^{(i)})^{r}_t (\xi^{(i)}; v_i) \}_{1 \le i \le p} .
\]

Recall that $(\mathbf{U}^{(i)})^{r}_t (\xi^{(i)}; v_i)$ is a pointwise copy of $\mathbf{U}^{r}_t (\xi; v_i)$ which is defined by \zcref{cal:master-PDE:prf:eq:-2:a}. In particular, $\mathbf{U}^{r}_t (\xi; v_i)$ takes values in
\begin{equation} \label{cal:master-PDE:prf:eq:-5:a}
\mathrm{R}_d \coloneq \bR^d \times \bR^d \times (\bR^d \otimes \bR^d) .
\end{equation}

Then $\mathbf{X}^{r}_t (\xi; v)$ takes values in $(\mathrm{R}_d)^p$. It follows from \zcref{cal:master-PDE:prf:eq:-4:a} that the process $\{\mathbf{X}^{r}_t (\xi; v)\}_{r \in [r, T]}$, where $(\xi, v) \in L^2 (\cF_r; \bR^d) \times (\bR^d)^p$, can be written as the unique solution of an SDE
\begin{equation} \label{cal:master-PDE:prf:system:a}
\mathbf{X}^{r}_t (\xi; v) = \mathbf{X}^{r}_r (\xi; v) + \int_r^t F_{(1, p)} (s, \mathbf{X}^{r}_s (\xi; v)  ) \diff s + \int_r^t F_{(2, p)} (s, \mathbf{X}^{r}_s (\xi; v)  ) \diff \begin{bmatrix}
B^{(1)}_s \\
\vdots \\
B^{(p)}_s
\end{bmatrix} ,
\end{equation}
for some functions $(F_{(1, p)}, F_{(2, p)})$ of class $\sM^{k-1}_b$.

\item \textit{Step 2:} Notice that all the random variables on the right-hand side of \zcref{cal:master-PDE:prf:eq6:a} (except for $X_{t}^{r, \xi}$ which only appears through its distribution $\lbrbrak X_{t}^{r, \xi} \rbrbrak$) are components of $\mathbf{X}^{r}_t (\xi; v)$. For $(i_1, \ldots, i_p) \in \llbracket 1, d \rrbracket^p$, we define
\[
G_{(i_1, \ldots, i_p)} \colon \bT \times \sP_2 ( \bR^d ) \times (\mathrm{R}_d)^p \to \bR
\]
by
\begin{myalign}
& G_{(i_1, \ldots, i_p)} \Big (t, \mu, \big ( x^{(\ell)}, y^{(\ell)}, \{z^{(\ell)}_{(j, k)}\}_{1\le j, k \le d} \big )_{1 \le \ell \le p} \Big ) \\
& \coloneq  \sum_{j_1, \ldots, j_p=1 }^d \prod_{\ell=1}^p z^{(\ell)}_{(j_\ell, i_\ell)} [\partial_\mu^p \tilde \sU]_{(j_1, \ldots, j_p)} (t, \mu ; y^{(1)} , \ldots , y^{(p)} ) .
\end{myalign}

Then \zcref{cal:master-PDE:prf:eq6:a} can be re-written as
\begin{equation} \label{cal:master-PDE:prf:eq:-6:a}
[\partial_\mu^p \tilde \sV]_{(i_1, \ldots, i_p)} (t, r, \lbrbrak \xi \rbrbrak; v ) = \bE^{[p]} [ G_{(i_1, \ldots, i_p)} (t, \lbrbrak X_{t}^{r, \xi} \rbrbrak, \mathbf{X}^{r}_t (\xi; v) ) ] .
\end{equation} 

It follows from the smoothness of $\tilde \sU$ that $G_{(i_1, \ldots, i_p)}$ is of class $\sM^{k-p}_b$.

\item \textit{Step 3:} Notice that $\mathbf{X}^{r}_t (\xi; v)$ contains pointwise copies of $X_{t}^{r, \xi}$, so $\lbrbrak X_{t}^{r, \xi} \rbrbrak$ is the distribution of the first copy and can thus be expressed as a function of the distribution of $\mathbf{X}^{r}_t (\xi; v)$. For $(i_1, \ldots, i_p) \in \llbracket 1, d \rrbracket^p$, we define
\[
H_{(i_1, \ldots, i_p)} \colon \bT \times \sP_2 ( (\mathrm{R}_d)^p ) \to \bR
\]
by
\[
H_{(i_1, \ldots, i_p)} (t, \bm{\mu}) \coloneq \bE^{[p]} [ G_{(i_1, \ldots, i_p)} (t, \pi_\sharp \bm{\mu}, \mathbf{Y})] ,
\]
where $\lbrbrak \mathbf{Y} \rbrbrak = \bm{\mu}$ and the projection $\pi \colon (\mathrm{R}_d)^p \to \bR^d$ is defined as
\[
\pi \Big ( \big ( x^{(\ell)}, y^{(\ell)}, \{z^{(\ell)}_{(j, k)}\}_{1\le j, k \le d} \big )_{1 \le \ell \le p} \Big ) \coloneq x^{(1)} .
\]

Then \zcref{cal:master-PDE:prf:eq:-6:a} can be re-written as
\begin{equation} \label{cal:master-PDE:prf:eq:-7:a}
[\partial_\mu^p \tilde \sV]_{(i_1, \ldots, i_p)} (t, r, \lbrbrak \xi \rbrbrak; v )  = H_{(i_1, \ldots, i_p)} (t, \lbrbrak \mathbf{X}^{r}_t (\xi; v) \rbrbrak ) .
\end{equation} 

By the smoothness of $G_{(i_1, \ldots, i_p)}$, \zcref{par-form:2} and \zcref{cal:exp}(3), $H_{(i_1, \ldots, i_p)}$ is of class $\sM^{k-p}_b$.
\end{itemize}

Higher-order derivatives of $(\bR^d)^p \ni v \mapsto [\partial_\mu^p \tilde \sV]_{(i_1, \ldots, i_p)} (t, r, \lbrbrak \xi \rbrbrak; v )$ will lead to complicated expressions. To reduce the burden of heavy notations, we assume $d=1$ in the rest of the proof for \hyperlink{cal:master-PDE:prf:setting:A}{\textbf{Setting A}}. Then $\partial_\mu^p \tilde \sV$ is a real-valued function and $\mathbf{X}^{r}_t (\xi; v)$ takes values in $(\mathrm{R}_1)^p$ where $\mathrm{R}_1$ defined by \zcref{cal:master-PDE:prf:eq:-5:a} is equal to $\bR^3$. For convenience, we denote
\[
H_p \coloneq H_{(1, 1, \ldots, 1)} ,
\]
where $(1, 1, \ldots, 1)$ is the $p$-dimensional vector whose all elements are equal to $1$.

Then \zcref{cal:master-PDE:prf:eq:-7:a} becomes
\begin{equation} \label{cal:master-PDE:prf:eq:-8:a}
\partial_\mu^p \tilde \sV (t, r, \lbrbrak \xi \rbrbrak; v )  = H_{p} (t, \lbrbrak \mathbf{X}^{r}_t (\xi; v) \rbrbrak )
\end{equation} 
for any $(p, \xi, v) \in \llbracket 1, k \rrbracket \times L^2 (\cF_r; \bR) \times \bR^p$.

\item \textit{Derivatives of $\bR^p \ni v \mapsto \partial_\mu^p \tilde \sV (t, r, \lbrbrak \xi \rbrbrak; v )$ where $(t, r) \in \Delta^2$.} For $i = (i_1, \ldots, i_p)$ and $j = (j_1, \ldots, j_p)$ in $\bN^p$, we write $i \le j$ if and only if $i_1 \le j_1, i_2 \le j_2, \ldots, i_p \le j_p$. Let $\ell = (\ell_1, \ldots, \ell_p) \in \bN^p$ with $|\ell|_1 \coloneq \ell_1 + \cdots + \ell_p \le k-p$. For $(\xi, v) \in L^2 (\cF_r; \bR) \times \bR^p$ , we denote
\[
\mathbf{Y}^{r}_t (\xi; v) \coloneq \{ \partial_{v_p}^{i_p} \ldots \partial_{v_1}^{i_1} \mathbf{X}^{r}_t (\xi; v) \}_{(i_1, \ldots, i_p) \le \ell} .
\]

Recall that $\{ \mathbf{X}^{r}_t (\xi; v) \}_{t \in [r, T]}$ is the unique solution to the SDE \zcref{cal:master-PDE:prf:system:a} whose coefficients are of class $\sM^{k-1}_b$. Repeatedly applying Kunita's theory of stochastic flows (see e.g. \cite[Theorem 3.3.2]{kunita_stochastic_2019}), we get functions $(F_{(1, p, \ell)}, F_{(2, p, \ell)})$ of class $\sM_b^{k-1-|\ell|_1}$ such that $\{ \mathbf{Y}^{r}_t (\xi; v) \}_{t \in [r, T]}$ is the unique solution to the SDE
\begin{equation} \label{cal:master-PDE:prf:system:c:a}
\mathbf{Y}^{r}_t (\xi; v) = \begin{myaligned}[t]
& \mathbf{Y}^{r}_r (\xi; v) + \int_r^t F_{(1, p, \ell)} (s, \mathbf{Y}^{r}_s (\xi; v) ) \diff s + \int_r^t F_{(2, p, \ell)} (s, \mathbf{Y}^{r}_s (\xi; v) ) \diff \begin{bmatrix}
B^{(1)}_s \\
\vdots \\
B^{(p)}_s
\end{bmatrix} .
\end{myaligned}
\end{equation}

Differentiating \zcref{cal:master-PDE:prf:eq:-8:a} w.r.t $v_1$, we get
\begin{equation} \label{cal:master-PDE:prf:system:d:a}
\partial_{v_1} \partial_\mu^p \tilde \sV (t, r, \lbrbrak \xi \rbrbrak; v ) = \bE^{[p]} [ \partial_\mu H_p (t, \lbrbrak \mathbf{X}^{r}_t (\xi; v) \rbrbrak ; \mathbf{X}^{r}_t (\xi; v) ) \cdot \partial_{v_1} \mathbf{X}^{r}_t (\xi; v) ] .
\end{equation}

The right-hand side of \zcref{cal:master-PDE:prf:system:d:a} can be written compactly as the evaluation of a function on $\bT \times \sP_2 ((\mathrm{R}_1)^{2p})$ at time $t$ and the joint distribution of $(\mathbf{X}^{r}_t (\xi; v) , \partial_{v_1} \mathbf{X}^{r}_t (\xi; v))$.  By the smoothness of $H_p$, \zcref{par-form:2} and \zcref{cal:exp}[(1) and (3)], this function is of class $\sM_b^{k-p-1}$. Repeating the previous procedure, we get a function $H_{(p, \ell)}$ of class $\sM_b^{k-p-|\ell|_1}$ such that for all $(t, r) \in \Delta^2$ and $(\xi, v) \in L^2 (\cF_r; \bR) \times \bR^p$:
\begin{align}
\rD^{(p, 0, 0, \ell)} \tilde \sV (t, r, \lbrbrak \xi \rbrbrak; v ) & \coloneq \partial_{v_p}^{\ell_p} \ldots \partial_{v_1}^{\ell_1} \partial_\mu^p \tilde \sV (t, r, \lbrbrak \xi \rbrbrak; v ) \\
& = H_{(p, \ell)} (t, \lbrbrak \mathbf{Y}^{r}_t (\xi; v)  \rbrbrak ). \label{cal:master-PDE:prf:system:e:a}
\end{align}
\end{enumerate}

This completes the proof for \hyperlink{cal:master-PDE:prf:setting:A}{\textbf{Setting A}}.

Next we consider \hyperlink{cal:master-PDE:prf:setting:B}{\textbf{Setting B}} where $d=1, b=0$ and $\sigma (t, x, \mu) = \sigma (t, \mu)$. Despite this simplification, we still face the burden of heavy notations. First, \zcref{cal:master-PDE:prf:eq1} becomes
\begin{equation} \label{cal:master-PDE:prf:eq1:b}
\begin{dcases}
X_t^{r, \xi} & = \xi + \int_r^t \sigma (s, \lbrbrak X_s^{r, \xi} \rbrbrak ) \diff B_s , \\
Y_t^{r, x, \lbrbrak \xi \rbrbrak} & = x + \int_r^t \sigma (s, \lbrbrak X_s^{r, \xi} \rbrbrak ) \diff B_s . 
\end{dcases}
\end{equation}

Second, \zcref{cal:master-PDE:prf:eq:-1} becomes $\partial_x Y_t^{r, x, \lbrbrak \xi \rbrbrak} = 1$. Third, \zcref{cal:master-PDE:prf:eq2} becomes
\begin{align}
U_t^{r, x, \lbrbrak \xi \rbrbrak} (y) & = \int_r^t \bE [ \partial_\mu \sigma (s, \lbrbrak X_s^{r, \xi} \rbrbrak ; Y_s^{r, y, \lbrbrak \xi \rbrbrak} ) + \partial_\mu \sigma (s, \lbrbrak X_s^{r, \xi} \rbrbrak ; X_s^{r, \xi} ) U_s^{r, \xi} (y) ] \diff B_s \\
& = U_t^{r, \xi} (y) \label{cal:master-PDE:prf:measure-derivative}
\end{align}
for all $(x, \xi) \in \bR \times L^2 (\cF_r; \bR)$.

\begin{enumerate}[label=(\theenumi B)]
\item \textit{Higher-order derivatives of $\xi \mapsto Y_t^{r, x, \lbrbrak \xi \rbrbrak}$}. This is accomplished in \cite[Section 2.4]{chassagneux_weak_2022} which we summarize in the following.  By \zcref{cal:master-PDE:prf:measure-derivative}, $U_t^{r, x, \lbrbrak \xi \rbrbrak} (y)$ does not depend on $x$. Hence we can define
\[
\partial_\mu^p X_t^{r, \lbrbrak \xi \rbrbrak} (v_1, \ldots, v_p) \coloneq \partial_\mu^p Y_t^{r, x, \lbrbrak \xi \rbrbrak} (v_1, \ldots, v_p) .
\]
for all $v_1, \ldots, v_p \in \bR$ and $t \in [r, T]$.

We need to introduce some notations:

\begin{itemize}
\item For $p \in \bN^*$ and $i \in \llbracket 1, p \rrbracket$, we denote by $\Lambda_{i, p}$ the set of all strictly increasing functions $\theta \colon \llbracket 1, i \rrbracket \to \llbracket 1, p \rrbracket$.

\item For $p \in \bN^*$, we denote by $\mathrm{R}_p$ the set of all $y = \{ y_{(j, \ell)} \}_{1 \le j, \ell \le p}$ where $y_{(j, \ell)} \in \bR$.

\item For $p \in \bN^*$, we denote by $\mathrm{T}_p$ the set of all $z = \{ z_{(j, i, \theta)} \}_{\substack{1 \le j, i \le p \\ \theta \in \Lambda_{i, p}}}$ where $z_{(j, i, \theta)} \in \bR$.
\item For a function
\[
f \colon \bT \times \sP_2 (\bR) \times \bR^p \times \mathrm{R}_p \times \mathrm{T}_p \to \bR, \, (s, \mu, x, y, z) \mapsto f (s, \mu, x, y, z) ,
\]
we denote by $\partial_{x_j} f$ the partial derivative of $f$ w.r.t $x_j$, by $\partial_{y_{(j, \ell)}} f$ the partial derivative of $f$ w.r.t $y_{(j, \ell)}$, and by $\partial_{z_{(j, i, \theta)}} f$ the partial derivative of $f$ w.r.t $z_{(j, i, \theta)}$.

\item For $v = (v_1, \ldots, v_p) \in \bR^p$ and $\theta \in \Lambda_{i, p}$, we denote $v_{\theta} \coloneq (v_{\theta(1)}, \ldots, v_{\theta(i)}) \in \bR^i$.
\end{itemize}

It is straightforward to see that $\dim \mathrm{R}_p = p^2$ and $\dim \mathrm{T}_p = p \sum_{i=1}^p \binom{p}{i} = p(2^p - 1)$.  Then $\bR^p \times \mathrm{R}_p \times \mathrm{T}_p$ is isomorphic to $\bR^{p(p+2^p)}$. 

By \cite[Theorem 3.4]{chassagneux_weak_2022}, for each $p \in \llbracket 1, k \rrbracket$ and $v \in \bR^p$, the $p$-th derivative in measure $\partial_\mu^p X_t^{r, \lbrbrak \xi \rbrbrak} (v)$ exists and is the unique solution to the SDE
\begin{equation} \label{cal:master-PDE:prf:der-in-mea:b}
\partial_\mu^p X_t^{r, \lbrbrak \xi \rbrbrak} (v) = \begin{myaligned}[t]
& \int_r^t \bE^{[p]} [ \varphi_p (s, \lbrbrak X_s^{r, \xi} \rbrbrak , \mathbf{X}^{r}_s (\xi; v) ) ] \diff B_s ,
\end{myaligned}
\end{equation}
where
\begin{align} \label{cal:master-PDE:prf:master-process:b}
\bE^{[p]} & \coloneq \bE^{(1)} \bE^{(2)} \ldots \bE^{(p)} , \\
\mathbf{X}^{r}_s (\xi; v) & \coloneq \begin{myaligned}[t]
& \Big ( \{ ( X^{(j)} )_s^{r, \xi^{(j)}} \}_{1 \le j \le p} ; \{ ( Y^{(j)} )_s^{r, v_\ell, \lbrbrak \xi \rbrbrak} \}_{1 \le j, \ell \le p} ; \{ \partial_\mu^i (X^{(j)})^{r, \lbrbrak \xi \rbrbrak}_s (v_{\theta}) \}_{\substack{1 \le j, i \le p \\ \theta \in \Lambda_{i, p}}} \Big ) ,
\end{myaligned} 
\end{align}
and the sequence $\{ \varphi_p \}_{p \in \llbracket 1, k \rrbracket}$ of functions
\[
\varphi_p \colon \bT \times \sP_2 (\bR) \times \bR^p \times \mathrm{R}_p \times \mathrm{T}_p \to \bR
\]
is defined as
\[
\varphi_1 (s, \mu, x, y, z) \coloneq \partial_\mu \sigma (s, \mu; y) + \partial_\mu \sigma (s, \mu; x) z
\]
and the recurrence relation
\begin{myalign}
& \varphi_{p+1} (s, \mu, x, y, z) \\
& \coloneq \begin{myaligned}[t]
& \partial_\mu \varphi_p ( s, \mu, x^{[p]}, y^{[p]}, z^{[p]} ; y_{(p+1, p+1)} ) \\
& + \partial_\mu \varphi_p ( s, \mu, x^{[p]}, y^{[p]}, z^{[p]} ; x_{p+1} ) z_{(p+1, 1, P_{p+1})} \\
& + \sum_{j=1}^p \partial_{x_j} \varphi_p (s, \mu, (x_1, \ldots, x_{j-1}, y_{(j, p+1)}, x_{j+1}, \ldots, x_p), y^{[p]}, z^{[p]} ) \\
& + \sum_{j=1}^p \partial_{x_j} \varphi_p (s, \mu, x^{[p]}, y^{[p]} , z^{[p]} ) z_{(j, 1 , P_{p+1})} \\
& + \sum_{j, \ell=1}^p \partial_{y_{(j, \ell)}} \varphi_p (s, \mu, x^{[p]}, y^{[p]} , z^{[p]} ) z_{(j, 1 , P_{p+1})} \\
& + \sum_{j, i=1}^p \sum_{\theta \in \Lambda_{i, p}} \partial_{z_{(j, i, \theta)}} \varphi_p (s, \mu, x^{[p]}, y^{[p]}, z^{[p]} ) z_{(j, i+1 , \theta_{i+1, p+1})} .
\end{myaligned} \label{cal:master-PDE:prf:recurrence:b}
\end{myalign}

In \zcref{cal:master-PDE:prf:recurrence:b}, we use the notations
\begin{itemize}
\item the truncations of $(x, y, z) \in \bR^{p+1} \times \mathrm{R}_{p+1} \times \mathrm{T}_{p+1}$ are defined as
\begin{align}
x^{[p]} & \coloneq \{ x_j \}_{1 \le j \le p} \in \bR^p, \\
y^{[p]} & \coloneq \{ y_{(j, \ell)} \}_{1 \le j, \ell \le p} \in \mathrm{R}_p, \\
z^{[p]} & \coloneq \{ z_{(j, i, \theta)} \}_{\substack{1 \le j, i \le p \\ \theta \in \Lambda_{i, p}}} \in \mathrm{T}_p ,
\end{align}

\item $P_{p+1} \in \Lambda_{1, p+1}$ is defined as $P_{p+1} (1) \coloneq p+1$.

\item $\theta_{i+1, p+1} \in \Lambda_{i+1, p+1}$ is defined as $\theta_{i+1, p+1} \restriction \llbracket 1, i \rrbracket \coloneq \theta$ and $\theta_{i+1, p+1} (i+1) \coloneq p+1$.
\end{itemize}

By the smoothness of $\sigma$ and \zcref{cal:exp}(1), $\varphi_p$ is of class $\sM^{k-p}_b$. Let us put the right-hand side of \zcref{cal:master-PDE:prf:der-in-mea:b} into a more convenient form. For this purpose, we define
\[
\psi_p \colon \bT \times \sP_2 (\bR^p \times \mathrm{R}_p \times \mathrm{T}_p) \to \bR
\]
by
\[
\psi_p (s, \bm{\mu}) \coloneq \bE[\varphi_p (s, \pi_\sharp \bm{\mu}, \mathbf{Y}_1, \mathbf{Y}_2, \mathbf{Y}_3)] ,
\]
where $\lbrbrak (\mathbf{Y}_1, \mathbf{Y}_2, \mathbf{Y}_3) \rbrbrak = \bm{\mu}$ and $\pi \colon \bR^p \times \mathrm{R}_p \times \mathrm{T}_p \to \bR$ such that $\pi (x, y, z) \coloneq x_1$.

Then \zcref{cal:master-PDE:prf:der-in-mea:b} can be re-written as
\begin{equation} \label{cal:master-PDE:prf:der-in-mea:a:b}
\partial_\mu^p X_t^{r, \lbrbrak \xi \rbrbrak} (v) = \int_r^t \psi_p (s, \lbrbrak \mathbf{X}^{r}_s (\xi; v) \rbrbrak ) \diff B_s .
\end{equation}

By the smoothness of $\varphi_p$, \zcref{par-form:2} and \zcref{cal:exp}(3), $\psi_p$ is of class $\sM^{k-p}_b$.

\item \textit{Higher-order derivatives of $\mu \mapsto \tilde \cV (t, r, \mu)$ where $(t, r) \in \Delta^2$.} This is accomplished in \cite{chassagneux_weak_2022} which we summarize in the following. We have that \zcref{cal:master-PDE:prf:value-function} becomes
\begin{myalign}
& \partial_\mu \tilde \sV (t, r, \lbrbrak \xi \rbrbrak; x ) \\
& = \bE^{(1)} [ \partial_\mu \tilde \sU (t, \lbrbrak X_{t}^{r, \xi} \rbrbrak ; (Y^{(1)})_{t}^{r, x, \lbrbrak \xi \rbrbrak} ) + \partial_\mu (X^{(1)})_{t}^{r, \xi^{(1)}} (x) \partial_\mu \tilde \sU (t, \lbrbrak X_{t}^{r, \xi} \rbrbrak ; (X^{(1)})_{t}^{r, \xi^{(1)}} ) ] .
\end{myalign}
for every $(x, \xi) \in \bR \times L^2 (\cF_r; \bR)$. Iterating the previous procedure, we obtain for any $(p, \xi, v) \in \llbracket 1, k \rrbracket \times L^2 (\cF_r; \bR) \times \bR^p$ that
\begin{equation} \label{cal:master-PDE:prf:eq6:b}
\partial_\mu^p \tilde \sV (t, r, \lbrbrak \xi \rbrbrak; v ) = \begin{myaligned}[t]
& \bE^{[p]} [ G_p (t, \lbrbrak X_t^{r, \xi} \rbrbrak , \mathbf{X}^{r}_t (\xi; v) ) ] .
\end{myaligned}
\end{equation}

Recall that $\bE^{[p]}$ and $\mathbf{X}^{r}_t (\xi; v)$ are defined by \zcref{cal:master-PDE:prf:master-process:b}. Above, the sequence $\{ G_p \}_{p \in \llbracket 1, k \rrbracket}$ of functions
\[
G_p \colon \bT \times \sP_2 (\bR) \times \bR^p \times \mathrm{R}_p \times \mathrm{T}_p \to \bR
\]
satisfies \zcref{cal:master-PDE:prf:recurrence:b} with $G_1$ defined as
\[
G_1 (t, \mu, x, y, z) \coloneq \partial_\mu \tilde \sU (t, \mu; y) + \partial_\mu \tilde \sU (t, \mu; x) z .
\]

It follows from the smootheness of $\tilde \sU$ that $G_p$ is of class $\sM^{k-p}_b$. As for the right-hand side of \zcref{cal:master-PDE:prf:der-in-mea:b}, let us put that of \zcref{cal:master-PDE:prf:eq6:b} into a more convenient form. For this purpose, we define
\[
H_p \colon \bT \times \sP_2 (\bR^p \times \mathrm{R}_p \times \mathrm{T}_p) \to \bR
\]
by
\[
H_p (t, \bm{\mu}) \coloneq \bE [ G_p (t, \pi_\sharp \bm{\mu}, \mathbf{Y}_1, \mathbf{Y}_2, \mathbf{Y}_3)] ,
\]
where $\lbrbrak  (\mathbf{Y}_1, \mathbf{Y}_2, \mathbf{Y}_3) \rbrbrak = \bm{\mu}$ and $\pi \colon \bR^p \times \mathrm{R}_p \times \mathrm{T}_p \to \bR$ such that $\pi (x, y, z) \coloneq x_1$.

Then \zcref{cal:master-PDE:prf:eq6:b} can be re-written as
\begin{equation} \label{cal:master-PDE:prf:eq6:a:b}
\partial_\mu^p \tilde \sV (t, r, \lbrbrak \xi \rbrbrak; v ) = H_p (t, \lbrbrak \mathbf{X}^{r}_t (\xi; v) \rbrbrak ) .
\end{equation}

By the smoothness of $G_p$, \zcref{par-form:2} and \zcref{cal:exp}(3), $H_p$ is of class $\sM^{k-p}_b$.

\item \textit{Derivatives of $\bR^p \ni v \mapsto \partial_\mu^p \tilde \sV (t, r, \lbrbrak \xi \rbrbrak; v )$ where $(t, r) \in \Delta^2$.} By \zcref{cal:master-PDE:prf:eq1:b} and \zcref{cal:master-PDE:prf:der-in-mea:a:b}, given $(p, \xi) \in \llbracket 1, k \rrbracket \times L^2 (\cF_r; \bR)$ and $v=(v_1, \ldots, v_p) \in \bR^p$, the process $\{ \mathbf{X}^{r}_t (\xi; v) \}_{t \in [r, T]}$ defined by \zcref{cal:master-PDE:prf:master-process:b} is the unique solution to the system
\begin{equation} \label{cal:master-PDE:prf:system:b}
\begin{dcases}
( X^{(j)} )_t^{r, \xi^{(j)}} & = \xi^{(j)} + \int_r^t \sigma (s, \lbrbrak X_s^{r, \xi} \rbrbrak ) \diff B_s^{(j)} , \\
( Y^{(j)} )_t^{r, v_\ell, \lbrbrak \xi \rbrbrak} & = v_\ell + \int_r^t \sigma (s, \lbrbrak X_s^{r, \xi} \rbrbrak ) \diff B_s^{(j)} , \\
\partial_\mu^i ( X^{(j)} )_t^{r, \lbrbrak \xi \rbrbrak} (v_{\theta}) & = \begin{myaligned}[t]
& \int_r^t \psi_i (s, \lbrbrak \mathbf{X}^{r}_s (\xi; v_{\theta}) \rbrbrak ) \diff B_s^{(j)} ,
\end{myaligned}
\end{dcases}
\end{equation}
where $1 \le j, i \le p; \theta \in \Lambda_{i, p}$ and $v_{\theta} \coloneq (v_{\theta(1)}, \ldots, v_{\theta(i)})$.

Then \zcref{cal:master-PDE:prf:system:b} can be re-written as
\begin{equation} \label{cal:master-PDE:prf:system:a:b}
\mathbf{X}^{r}_t (\xi; v) = \mathbf{X}^{r}_r (\xi; v) + \int_r^t F_{(2, p)} (s, \lbrbrak \mathbf{X}^{r}_s (\xi; v) \rbrbrak ) \diff \begin{bmatrix}
B^{(1)}_s \\
\vdots \\
B^{(p)}_s
\end{bmatrix}
\end{equation}
for a function
\[
F_{(2, p)} \colon \bT \times \sP_2 (\bR^p \times \mathrm{R}_p \times \mathrm{T}_p) \to (\bR^p \times \mathrm{R}_p \times \mathrm{T}_p) \otimes \bR^p
\]
that we specify in the following. Recall that $\bR^p \times \mathrm{R}_p \times \mathrm{T}_p$ is isomorphic to $\bR^{p(p+2^p)}$. Each row of $F_{(2, p)} (t, \bm{\mu})$ contains at most one non-zero entry.
\begin{itemize}
\item For the rows corresponding to $( X^{(j)} )_t^{r, \xi^{(j)}}$ and $( Y^{(j)} )_t^{r, v_\ell, \lbrbrak \xi \rbrbrak}$, such non-zero entry is located at the $j$-th coordinate with value $\sigma (t, \pi_\sharp \bm{\mu} )$. Here the projection $\pi \colon \bR^p \times \mathrm{R}_p \times \mathrm{T}_p \to \bR$ is such that $\pi (x, y, z) \coloneq x_1$.

\item For $\theta \in \Lambda_{i, p}$, we define the projection $\pi^\theta \colon \bR^p \times \mathrm{R}_p \times \mathrm{T}_p \to \bR^i \times R_i \times T_i$ by
\[
\pi^\theta (x, y, z) \coloneq \Big ( \{ x_{j'} \}_{1 \le j' \le i} ; \{ y_{(j', \ell')} \}_{1 \le j', \ell' \le i} ; \{ z_{(j', i', \theta \circ \theta')} \}_{\substack{1 \le j', i' \le i \\ \theta' \in \Lambda_{i', i}}} \Big ) .
\]

\item For the row corresponding to $\partial_\mu^i ( X^{(j)} )_t^{r, \lbrbrak \xi \rbrbrak} (v_{\theta})$, such non-zero entry is located at the $j$-th coordinate with value $\psi_i (t, \pi^\theta_\sharp \bm{\mu})$.
\end{itemize}

Recall the notation that $\pi_\sharp \bm{\mu}$ and $\pi^\theta_\sharp \bm{\mu}$ are the push-forward measures of $\bm{\mu}$ through $\pi, \pi^\theta$ respectively.

By the smoothness of $(\sigma, \{\psi_i\}_{1 \le i \le p})$ and \zcref{par-form:2}, $F_{(2, p)}$ is of class $\sM^{k-p}_b$. It follows from \zcref{cal:master-PDE:prf:system:b} that $\mathbf{X}^{r}_r (\xi; v)$ is a column vector with a specific form. For the row corresponding to $( X^{(j)} )_t^{r, \xi^{(j)}}$, its entry is $\xi^{(j)}$. For the row corresponding to $( Y^{(j)} )_t^{r, v_\ell, \lbrbrak \xi \rbrbrak}$, its entry is $v_\ell$. For the row corresponding to $\partial_\mu^i ( X^{(j)} )_t^{r, \lbrbrak \xi \rbrbrak} (v_{\theta})$, its entry is $0$.

Differentiating \zcref{cal:master-PDE:prf:system:a:b} w.r.t $v_1$, we get
\begin{equation} \label{cal:master-PDE:prf:system:b:b}
\partial_{v_1} \mathbf{X}^{r}_t (\xi; v) = \partial_{v_1} \mathbf{X}^{r}_r (\xi; v) + \int_r^t \bE^{[p]} [ \partial_\mu F_{(2, p)} (s, \lbrbrak \mathbf{X}^{r}_s (\xi; v) \rbrbrak ; \mathbf{X}^{r}_s (\xi; v) ) \partial_{v_1} \mathbf{X}^{r}_s (\xi; v) ] \diff \begin{bmatrix}
B^{(1)}_s \\
\vdots \\
B^{(p)}_s
\end{bmatrix} .
\end{equation}

The term inside the stochastic integral of \zcref{cal:master-PDE:prf:system:b:b} can be written compactly as the evaluation of a function on $\bT \times \sP_2 ((\bR^p \times \mathrm{R}_p \times \mathrm{T}_p)^2)$ at time $s$ and the joint distribution of $(\mathbf{X}^{r}_s (\xi; v) , \partial_{v_1} \mathbf{X}^{r}_s (\xi; v))$.  By the smoothness of $F_{(2, p)}$, \zcref{par-form:2} and \zcref{cal:exp}[(1) and (3)], this function is of class $\sM_b^{k-p-1}$.

For $i = (i_1, \ldots, i_p)$ and $j = (j_1, \ldots, j_p)$ in $\bN^p$, we write $i \le j$ if and only if $i_1 \le j_1, i_2 \le j_2, \ldots, i_p \le j_p$. We denote
\[
\mathbf{Y}^{r}_t (\xi; v) \coloneq \{ \partial_{v_p}^{i_p} \ldots \partial_{v_1}^{i_1} \mathbf{X}^{r}_t (\xi; v) \}_{(i_1, \ldots, i_p) \le \ell} .
\]

Repeating the previous procedure, we get a function $F_{(2, p, \ell)}$ of class $\sM_b^{k-p-|\ell|_1}$ such that $\{ \mathbf{Y}^{r}_t (\xi; v) \}_{t \in [r, T]}$ is the unique solution to the SDE
\begin{equation} \label{cal:master-PDE:prf:system:c:b}
\mathbf{Y}^{r}_t (\xi; v) = \mathbf{Y}^{r}_r (\xi; v) + \int_r^t F_{(2, p, \ell)} (s, \lbrbrak \mathbf{Y}^{r}_s (\xi; v)  \rbrbrak ) \diff \begin{bmatrix}
B^{(1)}_s \\
\vdots \\
B^{(p)}_s
\end{bmatrix} .
\end{equation}

Differentiating \zcref{cal:master-PDE:prf:eq6:a:b} w.r.t $v_1$, we get
\begin{equation} \label{cal:master-PDE:prf:system:d:b}
\partial_{v_1} \partial_\mu^p \tilde \sV (t, r, \lbrbrak \xi \rbrbrak; v ) = \bE^{[p]} [ \partial_\mu H_p (t, \lbrbrak \mathbf{X}^{r}_t (\xi; v) \rbrbrak ; \mathbf{X}^{r}_t (\xi; v) ) \partial_{v_1} \mathbf{X}^{r}_t (\xi; v) ] .
\end{equation}

The right-hand side of \zcref{cal:master-PDE:prf:system:d:b} can be written compactly as the evaluation of a function on $\bT \times \sP_2 ((\bR^p \times \mathrm{R}_p \times \mathrm{T}_p)^2)$ at time $t$ and the joint distribution of $(\mathbf{X}^{r}_t (\xi; v) , \partial_{v_1} \mathbf{X}^{r}_t (\xi; v))$.  By the smoothness of $H_p$, \zcref{par-form:2} and \zcref{cal:exp}[(1) and (3)], this function is of class $\sM_b^{k-p-1}$. Repeating the previous procedure, we get a function $H_{(p, \ell)}$ of class $\sM_b^{k-p-|\ell|_1}$ such that
\begin{align}
\rD^{(p, 0, 0, \ell)} \tilde \sV (t, r, \lbrbrak \xi \rbrbrak; v ) & \coloneq \partial_{v_p}^{\ell_p} \ldots \partial_{v_1}^{\ell_1} \partial_\mu^p \tilde \sV (t, r, \lbrbrak \xi \rbrbrak; v ) \\
& = H_{(p, \ell)} (t, \lbrbrak \mathbf{Y}^{r}_t (\xi; v)  \rbrbrak ). \label{cal:master-PDE:prf:system:e:b}
\end{align}
\end{enumerate}

This completes the proof for \hyperlink{cal:master-PDE:prf:setting:B}{\textbf{Setting B}}.
\end{proof}

To our best knowledge, the final two parts about higher-order derivatives in time are new in the literature.

\begin{enumerate}
\item \textit{Derivatives of $[r, T] \ni t \mapsto \rD^{(p, 0, 0, \ell)}  \tilde \sV (t, r, \mu; v )$ where $p \in \llbracket 1, k \rrbracket$ and $\ell = (\ell_1, \ldots, \ell_p) \in \bN^p$ such that $p + |\ell|_1 \le k$.} Recall that $|\ell|_1 \coloneq \ell_1 + \cdots + \ell_p$. Let $\beta \coloneq (p, 0, 0, \ell)$. Let the functions $(F_{(1, p, \ell)}, F_{(2, p, l)}, H_{(p, \ell)})$ be given by \zcref{cal:master-PDE:prf:long-thm}. Then
\begin{itemize}
\item $(F_{(1, p, \ell)}, F_{(2, p, l)}, H_{(p, \ell)})$ are of class $\sM^{k-p-|\ell|_1}_b$.
\item For all $(\xi, v) \in L^2 (\cF_r; \bR) \times \bR^p$,
\begin{equation} \label{cal:master-PDE:prf:system:e:a:1}
\rD^{\beta} \tilde \sV (t, r, \lbrbrak \xi \rbrbrak; v ) = H_{(p, \ell)} (t, \lbrbrak \mathbf{Y}^{r}_t (\xi; v)  \rbrbrak ) ,
\end{equation}
where
$\{ \mathbf{Y}^{r}_t (\xi; v) \}_{t \in [r, T]}$ is the unique solution to the multi-dimensional SDE
\begin{equation}
\mathbf{Y}^{r}_t (\xi; v) = \begin{myaligned}[t] \label{cal:master-PDE:prf:system:c:a:1}
& \mathbf{Y}^{r}_r (\xi; v) + \int_r^t F_{(1, p, \ell)} (s, \mathbf{Y}^{r}_s (\xi; v), \lbrbrak \mathbf{Y}^{r}_s (\xi; v) \rbrbrak) \diff s \\
& + \int_r^t F_{(2, p, \ell)} (s, \mathbf{Y}^{r}_s (\xi; v), \lbrbrak \mathbf{Y}^{r}_s (\xi; v) \rbrbrak ) \diff \begin{bmatrix}
B^{(1)}_s \\
\vdots \\
B^{(p)}_s
\end{bmatrix} .
\end{myaligned}
\end{equation}
\end{itemize}

WLOG, we assume $r < t < T$. Let $h>0$ such that $t-h \ge 0$ and $t+h \le T$. We define an operator $\rL_{(p, \ell)}$ acting on $f$ of class $\sM^2_b$ by
\[
\rL_{(p, \ell)} f (s, \bm{\mu}) \coloneq \begin{myaligned}[t]
& \partial_s f (s, \bm{\mu}) + \bE \Big [ F_{(1, p, \ell)} (s, \mathbf{Z}, \bm{\mu}) \cdot \partial_\mu  f(s, \bm{\mu}; \mathbf{Z}) \\
& + \frac{1}{2} \{ F_{(2, p, \ell)} F_{(2, p, \ell)}^\top \} (s, \mathbf{Z}, \bm{\mu}) : \nabla_v \partial_\mu  f(s, \bm{\mu}; \mathbf{Z}) \Big ] ,
\end{myaligned}
\]
where $\lbrbrak \mathbf{Z} \rbrbrak = \bm{\mu}$.

By \zcref{cal:master-PDE:prf:system:e:a:1}, \zcref{cal:master-PDE:prf:system:c:a:1} and \zcref{cal:Ito-lemma},
\begin{equation} \label{cal:master-PDE:prf:time:eq1:a}
\begin{myaligned} 
& \rD^{\beta} \tilde \sV (t+h, r, \lbrbrak \xi \rbrbrak; v) - \rD^{\beta} \tilde \sV (t, r, \lbrbrak \xi \rbrbrak; v)  \\
& = \int_{t}^{t+h} \rL_{(p, \ell)} H_{(p, \ell)} (s, \lbrbrak \mathbf{Y}^{r}_s (\xi; v) \rbrbrak ) \diff s .
\end{myaligned}
\end{equation}

By \zcref{cal:exp}[(1) and (3)] and the smoothness of $(F_{(1, p, \ell)}, F_{(2, p, l)}, H_{(p, \ell)})$, the map $\rL_{(p, \ell)} H_{(p, \ell)}$ is of class $\sM^{k-p-|\ell|_1-2}_b$. Applying \zcref{cal:exp}(4) to the SDE \zcref{cal:master-PDE:prf:system:c:a:1}, we obtain that
\[
[r, T] \ni s \mapsto \rL_{(p, \ell)} H_{(p, \ell)} (s, \lbrbrak \mathbf{Y}^{r}_s (\xi; v) \rbrbrak )
\]
is continuous. By \zcref{cal:master-PDE:prf:time:eq1:a},
\[
\begin{myaligned}
& \lim_{h \downarrow 0} \frac{\rD^{\beta}  \tilde \sV (t+h, r, \lbrbrak \xi \rbrbrak; v) - \rD^{\beta}  \tilde \sV (t, r, \lbrbrak \xi \rbrbrak; v)}{h}  \\
& = \rL_{(p, \ell)} H_{(p, \ell)} (t, \lbrbrak \mathbf{Y}^{r}_t (\xi; v) \rbrbrak ) .
\end{myaligned}
\]

Similarly,
\[
\begin{myaligned}
& \lim_{h \downarrow 0} \frac{\rD^{\beta}  \tilde \sV (t, r, \lbrbrak \xi \rbrbrak; v) - \rD^{\beta}  \tilde \sV (t-h, r, \lbrbrak \xi \rbrbrak; v)}{h}  \\
& = \rL_{(p, \ell)} H_{(p, \ell)} (t, \lbrbrak \mathbf{Y}^{r}_t (\xi; v) \rbrbrak ) .
\end{myaligned}
\]

Therefore,
\begin{equation} \label{cal:master-PDE:prf:time:eq2:a}
\partial_t \rD^{\beta}  \tilde \sV (t, r, \lbrbrak \xi \rbrbrak; v) = \rL_{(p, \ell)} H_{(p, \ell)} (t, \lbrbrak \mathbf{Y}^{r}_t (\xi; v) \rbrbrak) .
\end{equation}

Let $m \in \bN$ such that $p + |\ell|_1 + 2m \le k$. Repeating the above procedure, we get that
\[
t \mapsto \rD^{\beta} \tilde \sV (t, r, \lbrbrak \xi \rbrbrak; v)
\]
is $m$-times continuously differentiable and that for $\alpha \coloneq (p, 0, (m, 0), \ell)$:
\begin{align}  
\rD^{\alpha} \tilde \sV (t, r, \lbrbrak \xi \rbrbrak; v) & \coloneq  \partial_t^{m} \rD^{\beta} \tilde \sV (t, r, \lbrbrak \xi \rbrbrak; v) \\
& = H_{(p, \ell, m)} (t, \lbrbrak \mathbf{Y}^{r}_t (\xi; v) \rbrbrak ) , \label{cal:master-PDE:prf:time:eq1:a:a}
\end{align}
where
\[
H_{(p, \ell, m)} \coloneq \rL_{(p, \ell)} \circ \cdots \circ \rL_{(p, \ell)} H_{(p, \ell)}
\]
is the $m$-times composition of $\rL_{(p, \ell)}$ applied on $H_{(p, \ell)}$. Clearly, $H_{(p, \ell, m)}$ is of class $\sM_b^{k-p-|\ell|_1-2m}$. By the smootheness of $(F_{(1, p, \ell)}, F_{(2, p, \ell)}, H_{(p, \ell, m)})$ and classical stability estimates for \zcref{cal:master-PDE:prf:system:c:a:1}, we have the following properties that will be called \MakeLinkTarget*{cal:master-PDE:prf:pro-P}$(\mathbf{P})$:
\begin{itemize}
\item the function $\rD^{\alpha}  \tilde \sV$ is continuous.
\item there exists a constant $C$ (depending only on $(b, \sigma, \tilde \cU, k)$) such that for all $(t, r) \in \Delta^2; \mu, \mu' \in \sP_2 (\bR)$ and $v, v' \in \bR^p$:
\begin{equation} \label{cal:master-PDE:prf:stability-bound:a}
\begin{aligned}
| \rD^{\alpha}  \tilde \sV (t, r, \mu; v ) | & \le C , \\
| \rD^{\alpha}  \tilde \sV (t, r, \mu; v ) - \rD^{\alpha}  \tilde \sV (t, r, \mu'; v' ) | & \le C \{ |v-v'| + \sW_2(\mu, \mu') \} .
\end{aligned}
\end{equation}
\end{itemize}

With all above preparation, we are in the position to prove the smoothness of $\tilde \sV$. Recall that $d=1$ in both \hyperlink{cal:master-PDE:prf:setting:A}{\textbf{Setting A}} and \hyperlink{cal:master-PDE:prf:setting:B}{\textbf{Setting B}}.

\item \textit{Mixed derivatives of $\tilde \sV (t, r, \mu)$}. By \cite[Theorem 2.18]{chassagneux_weak_2022}, the map
\[
[0, t] \times \sP_2 (\bR) \ni (r, \mu) \mapsto \tilde \sV(t, r, \mu)
\]
is of class $\sM^2_b$ and satisfies the PDE
\begin{equation} \label{cal:master-PDE:prf:recur-time-diff:1:a}
- \partial_r \tilde \sV (t, r, \mu) =  \begin{myaligned}[t]
& \bE \Big [ b (r, Z, \mu) \partial_\mu \tilde \sV (t, r, \mu; Z) + \frac{1}{2} \sigma^2 (r, Z, \mu) \partial_v \partial_\mu \tilde \sV (t, r, \mu; Z) \Big ]
\end{myaligned}
\end{equation}
where $\lbrbrak Z \rbrbrak = \mu$. That being said, $(t, r, \mu) \mapsto \partial_r \tilde \sV (t, r, \mu)$ can be expressed through the derivatives of $b(r, v, \mu), \sigma^2 (r, v, \mu)$ and $\tilde \sV (t, r, \mu)$ w.r.t $(v, \mu)$. On the other hand, \hyperlink{cal:master-PDE:prf:pro-P}{$(\mathbf{P})$} provides good properties of the derivatives of $(t, \mu) \mapsto \tilde \sV (t, r, \mu)$. Therefore, $(t, r, \mu) \mapsto \partial_r \tilde \sV (t, r, \mu)$ inherits these good properties.

More precisely, by \zcref{cal:master-PDE:prf:recur-time-diff:1:a}, \hyperlink{cal:master-PDE:prf:pro-P}{$(\mathbf{P})$} and \zcref{cal:exp}[(1) and (3)],
\begin{itemize}
\item $[r, T] \times \sP_2 (\bR) \ni (t, \mu) \mapsto \partial_r \tilde \sV (t, r, \mu)$ exists and is of class $\sM_b^{k-2}$ with the constant $C$ in \zcref{cal:space1:ineq1} uniform in $r$.
\item for any multi-index $\alpha = (p, 0, (m, 0), \ell)$ of order $|\alpha| =p+|\ell|_1+2m \le k-2$, the map $\rD^\alpha \partial_r \tilde \sV$ is continuous; and there exists a constant $C$ (depending only on $(b, \sigma, \tilde \cU, k)$) such that for all $(t, r) \in \Delta^2; \mu, \mu' \in \sP_2 (\bR)$ and $v, v' \in \bR^p$:
\begin{equation}
\begin{aligned}
| \rD^\alpha \partial_r \tilde \sV (t, r, \mu; v ) | & \le C , \\
| \rD^\alpha \partial_r \tilde \sV (t, r, \mu; v ) - \rD^\alpha \partial_r \tilde \sV (t, r, \mu'; v' ) | & \le C \{ |v-v'| + \sW_2(\mu, \mu') \} .
\end{aligned}
\end{equation}
\end{itemize}

We can take the derivative w.r.t $r$ in \zcref{cal:master-PDE:prf:recur-time-diff:1:a}. Then $(t, r, \mu) \mapsto \partial_r^2 \tilde \sV (t, r, \mu)$ can be expressed through the derivatives of $b(r, v, \mu), \sigma^2 (r, v, \mu), \tilde \sV (t, r, \mu)$ and $\partial_r \tilde \sV (t, r, \mu)$ w.r.t $(v, \mu)$. So the above procedure can be repeated. More precisely, we will prove by induction that for $i \in \llbracket 1, \lfloor k/2 \rfloor \rrbracket$ that
\begin{itemize}
\item $[0, t] \times \sP_2 (\bR) \ni (r, \mu) \mapsto \partial_r^{i-1}  \tilde \sV (t, r, \mu)$ exists and is of class $\sM^2_b$.
\item $[r, T] \times \sP_2 (\bR) \ni (t, \mu) \mapsto \partial_r^{i} \tilde \sV (t, r, \mu)$ exists and is of class $\sM_b^{k-2i}$ with the constant $C_i$ in \zcref{cal:space1:ineq1} uniform in $r$. It has the form
\begin{equation} \label{cal:master-PDE:prf:recur-time-diff:2:a}
- \partial_r^i \tilde \sV (t, r, \mu) = \begin{myaligned}[t]
& \sum_{j=0}^{i-1} \begin{pmatrix} i-1 \\ j \end{pmatrix} \bE \Big [ \partial_r^{i-j-1} b (r, Z, \mu) \partial_\mu \partial_r^{j} \tilde \sV (t, r, \mu; Z) \\
& + \frac{1}{2} \partial_r^{i-j-1} \sigma^2 (r, Z, \mu) \partial_v \partial_\mu \partial_r^{j} \tilde \sV (t, r, \mu; Z) \Big ] ,
\end{myaligned}
\end{equation}
where $\lbrbrak Z \rbrbrak = \mu$.
\item for any multi-index $\alpha = (p, 0, (m, 0), \ell)$ of order $|\alpha| =p+|\ell|_1+2m \le k-2i$, if $i < \lfloor k/2 \rfloor$, the map $\rD^\alpha \partial_r^i \tilde \sV$ is continuous; and there exists a constant $C$ (depending only on $(b, \sigma, \tilde \cU, k)$) such that for all $(t, r) \in \Delta^2; \mu, \mu' \in \sP_2 (\bR)$ and $v, v' \in \bR^p$:
\begin{equation} \label{cal:master-PDE:prf:stability-bound:2:a}
\begin{aligned}
| \rD^\alpha \partial_r^i \tilde \sV (t, r, \mu; v ) | & \le C , \\
| \rD^\alpha \partial_r^i \tilde \sV (t, r, \mu; v ) - \rD^\alpha \partial_r^i \tilde \sV (t, r, \mu'; v' ) | & \le C \{ |v-v'| + \sW_2(\mu, \mu') \} .
\end{aligned}
\end{equation}
\end{itemize}

The base case $i=1$ is already proved. Let $i \in \bN^*$ with $i < \lfloor k/2 \rfloor$. Assume that the claim holds for all $j \in \llbracket 1, i \rrbracket$. We will prove that the claim holds for $i+1$.

\begin{itemize}
\item By induction hypothesis, $\mu \mapsto \partial_r^{i} \tilde \sV (t, r, \mu)$ exists and is of class $\sM^2_b$ with the constant in \zcref{cal:space1:ineq1} uniform in $r$.

\item By induction hypothesis and \zcref{cal:Clairaut-lm}(3), we get for $j \in \llbracket 0, i-1 \rrbracket$ that $r \mapsto \partial_\mu \partial_r^j \tilde \sV  (t, r, \mu, v)$ and $r \mapsto \partial_v \partial_\mu \partial_r^j \tilde \sV  (t, r, \mu, v)$ are differentiable; and
\begin{align}
\partial_r \partial_\mu \partial_r^j \tilde \sV  (t, r, \mu, v) & = \partial_\mu \partial_r^{j+1} \tilde \sV  (t, r, \mu, v)  , \\
\partial_r \partial_v \partial_\mu \partial_r^j \tilde \sV  (t, r, \mu, v) & = \partial_v \partial_\mu \partial_r^{j+1} \tilde \sV  (t, r, \mu, v) .
\end{align}
\end{itemize}

This allows us to take the derivative w.r.t $r$ in \zcref{cal:master-PDE:prf:recur-time-diff:2:a}. Then $r \mapsto \tilde \sV (t, r, \mu)$ is $(i+1)$-th times differentiable. By the product rule of differentiation applied to \zcref{cal:master-PDE:prf:recur-time-diff:2:a} and as in the proof of binomial theorem, we obtain
\begin{equation} \label{cal:master-PDE:prf:recur-time-diff:4:a}
- \partial_r^{i+1} \tilde \sV  (t, r, \mu) = \begin{myaligned}[t]
& \sum_{j=0}^{i} \begin{pmatrix} i \\ j \end{pmatrix} \bE \Big [ \partial_r^{i-j} b (r, Z, \mu) \partial_\mu \partial_r^{j} \tilde \sV  (t, r, \mu; Z) \\
& + \frac{1}{2} \partial_r^{i-j} \sigma^2 (r, Z, \mu) \partial_v \partial_\mu \partial_r^{j} \tilde \sV  (t, r, \mu; Z) \Big ] ,
\end{myaligned}
\end{equation}
where $\lbrbrak Z \rbrbrak = \mu$. In particular, $- \partial_r^{i+1} \tilde \sV  (t, r, \mu)$ has the desired form. By induction hypothesis, $(t, \mu) \mapsto \partial_r^{j} \tilde \sV  (t, r, \mu)$ is of class $\sM^{k-2j}_b$ for every $j \in \llbracket 0, i \rrbracket$. Notice that $(r, v, \mu) \mapsto \partial_r^{i-j} b (r, v, \mu)$ and $(r, v, \mu) \mapsto \partial_r^{i-j} \sigma^2 (r, v, \mu)$ are of class $\sM^{k-2(i-j)}_b$ for every $j \in \llbracket 0, i \rrbracket$. This combined with \zcref{cal:exp}[(1) and (3)] and \zcref{cal:master-PDE:prf:recur-time-diff:4:a} implies the required smoothness of $(t, \mu) \mapsto \partial_r^{i+1} \tilde \sV  (t, r, \mu)$. Therefore, the claim holds for $i+1$. Hence $\tilde \sV$ is of class $\sM^k_b$. This completes the proof.
\end{enumerate}

\printbibliography
\end{document}